\documentclass[11pt,reqno,twoside]{amsart}
\usepackage{tikz}
\usepackage{amssymb}
\usepackage{amsmath}
\usepackage{amsthm}
\usepackage{amsfonts}
\usepackage{amsthm,amscd}
\usepackage{amsbsy}
\usepackage{bm}
\usepackage{url} 
\usepackage{setspace}
\usepackage{float}
\usepackage{psfrag}
\usepackage{cite}
\usepackage{tikz}
\usepackage{float}
\usepackage{epstopdf}

\AtBeginDocument{% this will not print MR number in Bibliography
  \def\MR#1{}
}

\usetikzlibrary{positioning}

\usepackage[colorlinks, linkcolor=blue]{hyperref}

\restylefloat{figure}

\newtheorem{Theorem}{Theorem}[section]
\newtheorem{Lemma}[Theorem]{Lemma}

\newtheorem{Proposition}[Theorem]{Proposition}
\newtheorem{Remark}[Theorem]{Remark}

\newtheorem{Definition}[Theorem]{Definition}

\numberwithin{equation}{section}
\usepackage{setspace}
\usepackage{anysize}
\marginsize{0.7in}{0.7in}{0.43in}{0.43in}
\def\be{\begin{equation}}
	\def\ee{\end{equation}}
\def\ben{\begin{eqnarray}}
	\def\een{\end{eqnarray}}

\newcommand{\ncom}{\newcommand}
\ncom{\ov}{\overline}
\ncom{\Nb}{\mathbb N}
\ncom{\Ab}{\mathbb A}
\ncom{\Bb}{\mathbb B}
\ncom{\Mb}{\mathbb M}
\ncom{\Qb}{\mathbb Q}
\ncom{\Zb}{\mathbb Z}
\ncom{\Fb}{\mathbb{F}} 
\ncom{\Hb}{\mathbb{H}}
\ncom{\Cb}{\mathbb{C}} 
\ncom{\Rb}{\mathbb{R}}
\ncom{\Qf}{\mathfrak{q}} 
\ncom{\Mf}{\mathfrak{m}}
\ncom{\Nf}{\mathfrak{n}}
\ncom{\kk}{\mathbf{k}}
\ncom{\Fc}{\mathcal{F}}
\ncom{\Gc}{\mathcal{G}}
\ncom{\Ic}{{\mathcal{I}}}
\ncom{\Jc}{\mathcal{J}} 
\ncom{\Lc}{\mathcal{L}}
\ncom{\Rc}{\mathcal{R}}
\ncom{\Ac}{\mathcal{A}}
\ncom{\Hc}{\mathcal{H}}
\ncom{\Kc}{\mathcal{K}}
\ncom{\Dc}{\mathcal{D}}
\ncom{\Bc}{\mathcal{B}}
\ncom{\Oc}{\mathcal{O}}
\ncom{\II}{\mathscr{I}}
\ncom{\OO}{\mathscr{O}}

\ncom{\Af}{\mathbf{A}}
\ncom{\Bf}{\mathbf{B}}
\ncom{\Hf}{\mathbf{H}}
\ncom{\Gf}{\mathbf{G}}
\ncom{\Uf}{\mathbf{U}}
\ncom{\Yf}{\mathbf{Y}}
\ncom{\If}{\mathbf{I}}
\ncom{\T}{\mathcal{T}}
\ncom{\Yt}{\widetilde{\mathbf{Y}}}
\ncom{\ut}{\widetilde{u}}
\ncom{\Aw}{\mathbf{A}_\omega}
\ncom{\Pf}{\mathbf{P}}
\ncom{\Ff}{\mathbf{K}}
\ncom{\Lf}{\mathbf{L}}
\ncom{\Xf}{\mathbf{X}}
\ncom{\Vf}{\mathbf{V}}
\ncom{\Sf}{\mathbf{S}}
\ncom{\Tf}{\mathbf{T}}
\ncom{\Wf}{\mathbf{W}}

\ncom{\Lt}{L^2(\Omega)}
\ncom{\Hio}{H^1_0(\Omega)}
\ncom{\LL}{L^2(0,\infty;L^2(\Omega))}
\ncom{\LU}{L^2(0,\infty;\Uf)}
\ncom{\Lh}{L^2(0,\infty;\Hf)}
\ncom{\LH}{L^2(0,\infty;H^1(\Omega))}
\ncom{\LHH}{L^2(0,\infty;H^2(\Omega))}
\ncom{\HL}{H^1(0,\infty;L^2(\Omega))}
\ncom{\HH}{H^1(0,\infty;H^2(\Omega))}

\ncom{\ym}{\Yt_h(t)}
\ncom{\um}{\ut_h(t)}
\ncom{\uc}{\ut(t)}
\ncom{\yc}{\Yt(t)}
\ncom{\ap}{\Aw{_{,\Pf}}}
\ncom{\ahp}{\Af{_{{\omega_h},\Pf}}}
\ncom{\aph}{\Aw{_{,\Pf_h}}}
\ncom{\ahph}{\Af{_{\omega_h,\Pf_h}}}
\ncom{\ymm}{\Yt_h(t)}
\ncom{\umm}{\ut_h(t)}
\ncom{\ucc}{\ut(t)}
\ncom{\ycc}{\Yt(t)}

\ncom{\Mcl}{\mathcal{M}_h}
\ncom{\Ach}{\mathcal{A}_{\omega_h}}
\ncom{\Bch}{\mathcal{B}_h}
\ncom{\Yc}{\mathcal{Y}}
\ncom{\Mrm}{\bf\rm{M}}
\ncom{\Krm}{\bf\rm{K}}

\ncom{\nuh}{\widehat{\nu}}
\ncom{\wt}{\widetilde}
\ncom{\wh}{\widehat}
\ncom{\n}{\normalfont}

\newcommand{\vertiii}[1]{{\left\vert\kern-0.25ex\left\vert\kern-0.25ex\left\vert #1 
		\right\vert\kern-0.25ex\right\vert\kern-0.25ex\right\vert}}

\long\def\/*#1*/{}
\title{Feedback stabilization of parabolic coupled system and its numerical study}
\author{{WASIM AKRAM, DEBANJANA MITRA, NEELA NATARAJ,}
\and{MYTHILY RAMASWAMY}}\address{Wasim Akram \newline\indent Mathematics Department, Indian Institute of Technology Bombay, \newline \indent 
Powai, Mumbai - 400076, India, Email- {\normalfont{wakram@math.iitb.ac.in, wakram2k11@gmail.com}} \vspace{0.28cm}
\newline\indent  Debanjana Mitra \newline\indent Mathematics Department, Indian Institute of Technology Bombay, \newline \indent 
Powai, Mumbai - 400076, India, Email- {\normalfont {deban@math.iitb.ac.in, debanjana.math@gmail.com}}
\vspace{0.28cm}
\newline\indent  Neela Nataraj \newline\indent Mathematics Department, Indian Institute of Technology Bombay, \newline \indent 
Powai, Mumbai - 400076, India, Email- {\normalfont {neela@math.iitb.ac.in, nataraj.neela@gmail.com}}
\vspace{0.28cm}
\newline\indent  Mythily Ramaswamy \newline\indent International Centre for Theoretical Sciences, \newline \indent 
Bengaluru - 560 089, India, Email- {\normalfont {mythily.r@icts.res.in, mythily54@gmail.com}}}

\begin{document}
	\maketitle
	\pagenumbering{arabic}

%	\tableofcontents	
	\begin{abstract}
 In the first part of this article, we study feedback stabilization of a parabolic coupled system by using localized interior controls. The system is feedback stabilizable with exponential decay $-\omega<0$ for any $\omega>0$. A stabilizing control is found in feedback form by solving a suitable algebraic Riccati equation. In the second part, a conforming finite element method is employed to approximate the continuous system by a finite dimensional discrete system. The approximated system is also feedback stabilizable (uniformly) with exponential decay $-\omega+\epsilon$, for any $\epsilon>0$ and the feedback control is obtained by solving a discrete algebraic Riccati equation. The error estimate of stabilized solutions as well as stabilizing feedback controls are obtained. We validate the theoretical results by numerical implementations.
	\end{abstract}
\noindent \textbf{2020 MSC:} 35K20, 93D15, 65M12, 65M22, 65M60 \\
	\textbf{Keywords.} Parabolic coupled system, Stabilizability, Algebraic Riccati equation, Error estimates
	\section{Introduction}
\subsection{Model problem}
Coupled parabolic systems extensively arise to model many physical problems, for example, in mathematical biology, chemical reactions, electrophysiology, and so on. 
The control problems for coupled systems are interesting due to the interplay between equations and controls, and obtaining the results with minimum number of controls acting in the system is a challenging issue. We consider $\Omega$, a bounded domain of class $C^2$ in $\mathbb{R}^n$, $n\in \mathbb{N},$ with boundary $\partial\Omega$. Set $Q=\Omega \times(0,+\infty)$. Let $\chi_{\mathcal{O}}$ denote the characteristic function on a non-empty open subset $\mathcal{O}$ of $\Omega$. Consider the control problem for parabolic coupled equations that seeks $(y,z)$ such that 
\begin{equation} \label{PCE_mod} 
		\begin{aligned} 
			&  y_t- \eta_0\Delta y +\nu_0 y +\eta_1 z=u\chi_{\mathcal{O}} \text{ in } Q,\\
			&   z_t-\beta_0 \Delta z +(\kappa+\nu_0) z -y=0 \text{ in } Q,\\
			&  y=0, \; z=0 \text{ on } \partial\Omega\times(0,\infty), \\
			& y(\cdot,0)=y_0 \text{,  }\, z(\cdot,0)= z_0 \text{ in } \Omega.
		\end{aligned}
\end{equation}
Here $y, z$ are the state variables, $u$ is the control variable,  $\eta_0>0,\;\beta_0>0, \; \kappa>0,$   $\eta_1\in \mathbb{R},$ and $\nu_0\in \mathbb{R}$ are given constants.\\[2mm] 
\noindent Our main goal in this article is to study the feedback stabilization of \eqref{PCE_mod} with any exponential decay $-\omega<0$, when control acts only in one equation,
and its numerical analysis with error estimates for the solution and the feedback control.\\[2mm]
\noindent The system \eqref{PCE_mod} is written in the space $\Hf:=\Lt\times\Lt$ defined over the complex field as
\begin{equation} \label{eqn:main_system_PCE}
\Yf'(t)=\Af\Yf(t)+\Bf u(t)  \text{ for all } t>0, \quad  \Yf(0)=\Yf_0,
\end{equation} 
for $\Yf (t):=\left( \begin{matrix}	y(\cdot,t)\\z(\cdot,t)	\end{matrix} \right)$, $\Yf_0=\begin{pmatrix} y_0\\z_0 \end{pmatrix},$ 
where $\Af$ is the linear operator and $\Bf$ is the control operator associated to \eqref{PCE_mod}.

\begin{Definition}[stabilizability]
The pair $\n(\Af,\Bf)$ in \eqref{eqn:main_system_PCE} or equivalently \eqref{eqn:main_system_PCE} is said to be feedback stabilizable with exponential decay $-\omega<0,$ if there exists $\n K\in \Lc(L^2(\Omega),\Hf) $ such that the operator $\Af+\omega\If+\Bf K$ generates an exponentially stable semigroup on $\n\Hf$, that is,
\begin{align}
    \|e^{t(\Af+\omega\If+\Bf K)}\|_{\Lc(\Hf)} \le C e^{-\gamma t}\text{ for all } t>0,
\end{align}
and for some positive constants $\gamma,\, C.$
\end{Definition}

%%----------------------------------------------------------

\subsection{Literature survey}
%
%
%Feedback stabilization of heat equation with memory which equivalently coupled in a parabolic equation and ODE in $\Lt\times\Lt$ by finite dimensional control via a finite dimensional Riccati is studied in \cite{WKR}.  
%%------------------------------------------------------------------	
\noindent We first mention some available results in this context. 
 Feedback stabilization of a system using the solution of an algebraic Riccati equation is well-studied, for example, see \cite{BDDM, Lasiecka1} and references therein. 
 This technique has been used extensively to study the stabilization of parabolic equations, such as incompressible Navier-Stokes equations, coupled parabolic-ODE systems and other few models in \cite{JPR, JPR2010, BreKun17, WKR}. A characterization of the  stabilization of parabolic systems is obtained in \cite{BadTak14}.\\[2.mm]
 \noindent
Some results on controllability of the parabolic equations and coupled equations can be found in \cite{Khod05,Khodja11,CEPC08,Kun21,Kun21EECT,BadTak14} and the references therein.\\[2.mm]
\noindent
Numerical study of parabolic equations is also well-established. Error estimate for parabolic equations using a standard finite element Galerkin method with a standard energy technique and a duality argument can be found in several articles, for example, \cite{ThomBramb77,Thom79,Thomee}. In these articles, the authors obtain an $L^2-$error estimate of order $O(h^2)$ for the semidiscrete solution, when the initial condition is in $L^2(\Omega).$\\[2.mm]
\noindent
In \cite{Lasiecka1}, the numerical theory as the counterpart of the known continuous theory for feedback stabilization has been developed for abstract parabolic systems under certain hypothesis. In this book, the authors provide numerical approximation theory of continuous dynamics and  algebraic Riccati equations. The error estimates for the trajectories and feedback controls have been obtained with the `optimal rate' of convergence in this set up. The application of this theory and related works can be found in \cite{Las-Tr-91, Ls-Tr-87-I,Ls-Tr-87-II} and references therein. \\[2.mm]
\noindent
In \cite{KK91}, the authors consider linear quadratic control problems for parabolic equations with variable coefficients. They provide the approximation of the Riccati equation and obtain the rate of convergence for the optimal controls and optimal trajectories. \\[2.mm]
\noindent
Numerical stabilization for a Boussinesq system is investigated in \cite{CRRS} and for two-dimensional Navier-Stokes equations by the boundary stabilization are established in \cite{JPR6-2017}, where the authors determine the best control location. In this approach, the semi-discrete system is projected onto an unstable finite dimensional system using degree of stabilizaility and a feedback finite dimensional stabilizing control is constructed by solving an algebraic Riccati equation. The control thus obtained stabilizes the whole system.

\subsection{Methodology and Contributions.}
In this article, we study the feedback stabilization of \eqref{PCE_mod} in $\Hf$ with arbitrary exponential decay $-\omega<0$, and the associated numerical analysis. 
For this, the system \eqref{PCE_mod} is written in an equivalent operator form in \eqref{eqn:main_system_PCE}, where $\Af$ and $\Bf$ are the associated linear operator and the control operator defined in \eqref{eqdef-A_PCE}-\eqref{eqcontrol_PCE}. We show that $(\Af, D(\Af))$ forms an analytic semigroup on $\Hf$. Moreover, the spectrum of $\Af$ is inside a sector in the complex plane and the resolvent operators are compact. In fact the spectrum of $\Af$ consists of two sequences of real eigenvalues except finitely many complex eigenvalues and the sequence of real eigenvalues diverges  to $-\infty$. 
 It is shown that $(\Af+\omega \If, \Bf)$, for any $-\omega<0$,  is stabilizable in $\Hf$, by checking the Hautus condition. Next, the standard results give that the system can be stabilized in $\Hf$ using a feedback control which can be constructed by solving an algebraic Riccati equation in the whole space. We finally obtain that the coupled system is feedback stabilizable with any exponential dacay $-\omega<0$, using only one control acting in one parabolic equation. \\[2.mm]
\noindent
Next part of our work is to give the numerical approximation and error estimates of the trajectories and the feedback controls. To do it, we closely follow the technique introduced in 
\cite{Lasiecka1}. However, we execute the method in our case for coupled parabolic equations giving the explicit approximations and error estimates. 
A family of discrete operators $\Af_h,$ $\Bf_h$ on a finite dimensional space $\Hf_h$ from the finite element method lead to a family of discrete system approximating  \eqref{eqn:main_system_PCE} on $\Hf_h$. For all $h>0$, it is shown that the uniform coercivity of the bilinear form associated to $\Af_h$ hold and thus the spectrum of $\Af_h$, for $h>0$ lies in a uniform sector $\Sigma(-\nuh;\theta_0):=\{-\nuh+re^{\pm i\theta}\,|\, r> 0,\, \theta\in (-\pi, \pi], \, |\theta| \ge \theta_0\}$
in the complex plane, for  $\nuh>0$ and $\frac{\pi}{2}<\theta_0<\pi,$ and a uniform resolvent estimate holds in the complement of $\Sigma(-\nuh;\theta_0)$ for all $\Af_h$, for $h>0$. These estimates finally give that the family $\{\Af_h\}_{h>0}$ generates a uniform analytic semigroup on $\Hf_h$. It is also shown that the eigenvalues of $\Af_h$ converge to the corresponding eigenvalues of $\Af$ with a quadratic rate of convergence.
Moreover, the estimate between the resolvent operators for $\Af$ and $\Af_h$, $\|R(-\nuh,\Af)-R(-\nuh,\Af_h)\Pi_h\|_{\Lc(\Hf)}\le Ch^2$ is established for this system. 
This along with the uniform analyticity of $\{\Af_h\}_{h>0}$ is the crucial hypothesis given in \cite{Lasiecka1} to carry out further analysis. We explicitly derive them for this system. 
Next, using some perturbation results from operator theory, we study the feedback stabilization of the approximated system and obtain that if the continuous system is stabilizable with decay rate $-\gamma,$ then the approximated system is stabilizable with decay $-\wh\gamma,$ for any $\wh\gamma\in (0,\gamma)$.  Also, it is obtained that the stabilization of the discrete operators is uniform in $h$, for sufficiently small $h$. It is one of the challenging parts of this theory. Finally, the stabilizing control is determined by solving an algebraic Riccati equation in the discrete space $\Hf_h$. The error estimates of the stabilized solutions as well as the feedback controls are studied. The theoretical results are validated by numerical results. 

\medskip
\noindent The novelty of this work is that this article provides a complete analysis of the feedback stabilization of  a parabolic coupled system using only one control acting in one equation, and a thorough study of its numerical aspects. The approximations, error estimates, construction of feedback controls are given explicitly for this model with a rigorous spectral analysis. The dependency of the constants in estimates on the coefficients of the principal part of the system is also tracked. This may be helpful when we study the stabilization of the degenerate system, for example, $\beta_0=0$ in the second equation of \eqref{PCE_mod}, that is, a parabolic-ODE coupled system.

\subsection{Organisation.} The article is organized as follows. The main results of this article are stated in Section \ref{sec:mainresults-st}. 
Section \ref{sec:cont dyn} deals with the  spectral analysis of $\Af$, the analytic semigroup generated by $\Af$ and the feedback stabilizability of $(\Af,\Bf)$ with exponential decay $-\omega<0,$ for any $\omega>0.$ Section \ref{sec:lapl prop} presents some finite element approximation results while Section \ref{sec:Appopp} presents the uniform analyticity of semigroup generated by the discrete approximation operators $\Af_h$. Some intermediate convergence results needed for the rest of the article are given in Section \ref{sec:convres-unctrl}. Moreover, in this subsection, the error estimates for solutions of the system and its approximated system are established, when there is no control acting in the system. Section \ref{sec:pert} deals with the uniform analyticity and stability in a general set up, for a certain family of operators under a suitable perturbation. Section \ref{sec:disRiccati} contributes to the existence of the solution of a discrete algebraic Riccati equation. Further, the uniform stabilizability of the approximated system is proved using a feedback operator obtained by the discrete Riccati operator. In Section \ref{sec:results}, the proof of the main results on error estimates are established. Finally, the results of numerical experiments that validate the theoretical results are provided in Section \ref{sec:NI}.

 \subsection{Notations.} \label{subsec:notn} Throughout the paper, we denote the inner product and norm in $\Lt$ by $\langle \phi, \psi \rangle:=\int_\Omega \phi \overline{\psi}\, dx$ and $\|\phi\|:=\left( \int_\Omega |\phi|^2\,dx\right)^{1/2}$ for all $\phi,\psi\in \Lt.$  The space $\Hf:=\Lt\times\Lt$ is equipped with complex inner product $\left\langle \begin{pmatrix}u\\v\end{pmatrix},\begin{pmatrix}\phi\\ \psi\end{pmatrix}\right\rangle:=\langle u, \phi\rangle+\langle v,\psi \rangle$ and norm $\left\|\begin{pmatrix}u\\v\end{pmatrix}\right\|:=\left(\|u\|^2+\| v\|^2\right)^{1/2},$ for all $\begin{pmatrix}u\\v\end{pmatrix},\begin{pmatrix}\phi\\ \psi\end{pmatrix}\in\Hf$ and $\Uf:=\Lt$ is equipped with the usual complex inner product and norm. The notation $|\cdot|$ refers to the absolute value of a real number or the modulus of a complex number depending on the context. In the sequel, $\Re(\mu)$ denotes the real part of the complex number $\mu,$ for any operator $\textbf{T},$ $\rho(\textbf{T}):=\left\lbrace \mu\in \Cb\,|\, \mu I-\textbf{T}\text{ is invertible and the inverse is bounded}\right\rbrace$ denotes the resolvent set of $\textbf{T},$ for any $\mu\in \rho(\textbf{T})$, $R(\mu,\textbf{T}):=(\mu I -\textbf{T})^{-1}$ denotes the resolvent operator and $\sigma(\textbf{T})$ denotes the spectrum of $\textbf{T}$. The positive constant $C$ is generic and independent of the discretization parameter $h.$

\section{Main results}\label{sec:mainresults-st}
\noindent The unbounded operator $(\Af, D(\Af))$ on $\Hf=\Lt\times \Lt$ associated to \eqref{PCE_mod} is defined by
\begin{equation}  \label{eqdef-A_PCE}
			\Af:=\left( \begin{matrix}
				\eta_0 \Delta -\nu_0 I & -\eta_1 I \\ I & \beta_0\Delta -(\kappa+\nu_0) I 
			\end{matrix} \right)\text{ and } D(\Af):= \left(H^2(\Omega)\cap H^1_0(\Omega) \right)^2,
\end{equation}
where $I:\Lt\rightarrow \Lt$ is the identity operator. Further, the control operator $\Bf \in \Lc( L^2(\Omega), \Hf)$ is defined by 
\begin{equation}\label{eqcontrol_PCE} 
			\Bf f:=\left( \begin{matrix}
				f\chi_{\mathcal{O}}\\0
			\end{matrix} \right) \text{ for all } f\in L^2(\Omega).
\end{equation}

\noindent We denote $(\Af^*,D(\Af^*))$ and $\Bf^*$ as the adjoint operators corresponding to the  operators $(\Af,D(\Af))$ and $\Bf,$ respectively. 

\medskip
\noindent For any given $\omega>0,$ to study the stabilizability of \eqref{eqn:main_system_PCE} with exponential decay $-\omega<0,$ set $\Yt(t):=e^{\omega t}\Yf(t)$ and $\ut(t):=e^{\omega t}u(t)$. Then $(\Yt(t),\ut(t))$ satisfy 
\begin{equation}\label{eqn: main shifted_PCE}
			\Yt'(t)=\Af_\omega\Yt(t)+\Bf \ut(t)\;\text{ for all }t>0,\quad \Yt(0)=\Yf_0,
\end{equation}
where
\begin{equation} \label{Aw}
	\begin{array}{l}
		\Aw:=\Af+\omega \If \text{ with }D(\Aw)=D(\Af) \text{ and }
		\Aw^*:= \Af^*+\omega \If \text{ with }D(\Aw^*)=D(\Af^*),
	\end{array}
\end{equation}
$\If:\Hf\rightarrow\Hf$  being the identity operator. 

\noindent As studied in \cite{WKR}, if \eqref{eqn: main shifted_PCE} is stabilizable by a control $\ut(t)=K \wt \Yf(t),$ for some $K\in \Lc(\Hf,\Lt),$ then \eqref{eqn:main_system_PCE} is stabilizable with decay $-\omega<0$ by the control $u(t)=e^{-\omega t} \ut (t).$ Therefore, to study the stabilizability of \eqref{eqn:main_system_PCE} with decay $-\omega<0,$ it is enough to study the exponential stabilizability of \eqref{eqn: main shifted_PCE}. Often, the feedback operator $K$ is obtained by studying an optimization problem and by using a Riccati equation. 
To obtain the feedback operator, consider the optimal control problem:
\begin{align}\label{eqoptinf_PCE}
\min_{\ut\in E_{\Yf_0}} J( \Yt,\ut) \text{ subject to \eqref{eqn: main shifted_PCE}},
\end{align}
where
\begin{equation}
J( \Yt,\ut):=\int_0^\infty\big( \|\Yt(t)\|^2+\|\ut(t)\|^2\big) \, dt,
\end{equation}
and $E_{\Yf_0}:=\{ \ut\in L^2(0,\infty; L^2(\Omega))\mid \Yt \text{ solution of \eqref{eqn: main shifted_PCE} with control }\ut \text{ such that } J(\Yt,\ut)<\infty\}.$ 

\medskip
\noindent The next theorem yields the minimizer of \eqref{eqoptinf_PCE} as well as the stabilizing control in the feedback form. The proof of the theorem is provided in Section \ref{subsec:Proof main 1}. 
\begin{Theorem}[stabilization for the continuous case]\label{th:stb cnt}
Let $\omega>0$ be  any real number. Let $\normalfont\Aw$ (resp. $\normalfont\Bf$) be as defined in \eqref{Aw} (resp. \eqref{eqcontrol_PCE}). Then the following results hold: 
	\begin{enumerate}
	\item[(a)] There exists a unique operator $\normalfont\Pf\in \mathcal{L}(\Hf)$ that satisfies the non-degenerate Riccati equation 
				\begin{equation}\label{eqn:ARE}
					\begin{array}{l}
						\normalfont\Aw^*\Pf+\Pf\Aw-\Pf\Bf\Bf^*\Pf+\If=0,\quad \Pf=\Pf^* \geq 0 \text{ on }\Hf.
					\end{array}
				\end{equation}
	\item[(b)] For any $\Yf_0\in \Hf$, there exists a unique optimal pair $\normalfont(\Yf^\sharp,u^\sharp)$ for \eqref{eqoptinf_PCE}, where for all $t>0$, $\n\Yf^\sharp(t)$ satisfies the closed loop system 
		\begin{equation}\label{eqcl-loop}
			\normalfont\Yf{^\sharp}'(t)=(\Aw-\Bf\Bf^*\Pf)\Yf^\sharp(t),\;\; \Yf^\sharp(0)=\Yf_0,
		\end{equation}
		$u^\sharp(t)$ can be expressed in the feedback form as
		\begin{equation}\label{eqoptcntrl}
			\normalfont u^\sharp(t)=-\Bf^*\Pf\Yf^\sharp(t), %=-\Bf^*\Pf e^{t\Af_{\omega,\Pf}}\Yf_0.
		\end{equation}
		and
		$\displaystyle\min_{\ut\in E_{{\Yf}_0}}\normalfont J( \Yt,\ut)=J(\Yf^\sharp,u^\sharp)=\langle \Pf\Yf_0,\Yf_0\rangle.$
	\item[(c)] The feedback control in \eqref{eqoptcntrl} stabilizes \eqref{eqn: main shifted_PCE}. In particular, let us denote the operator $\normalfont\Af_{\omega,\Pf}:=\Aw-\Bf\Bf^*\Pf,$ with $\normalfont D(\Af_{\omega,\Pf})= D(\Af)$. The semigroup $\n \{e^{t\Af_{\omega,\Pf}}\}_{t\ge 0}$, generated by $(\normalfont\Af_{\omega,\Pf}, \normalfont D(\Af_{\omega,\Pf})),$ on $\n\Hf$ is analytic and exponentially stable, that is, there exist $\gamma>0$ and $M>0$ such that 
	$$\n \|e^{t\Aw{_{,\Pf}}}\|_{\Lc(\Hf)}\leq Me^{-\gamma t} \text{ for all }  t>0.$$
	\end{enumerate}
	\end{Theorem}
 \noindent In particular, the above theorem gives that $K=-\Bf^*\Pf$, where $\Pf$ is the solution of \eqref{eqn:ARE}, is a feedback operator such that 
 $\Aw+\Bf K$ is stable in $\Hf$.

\noindent Consider the finite dimensional subspace $\Hf_h$ of $\Hf$, projection operator $\Pi_h:\Hf\longrightarrow \Hf_h,$ the discrete operator $\Af_h:\Hf_h\longrightarrow\Hf_h$ that corresponds to $\Af$  and the discrete operator $\Af_{\omega_h}:\Hf_h\longrightarrow\Hf_h$ that corresponds to $\Aw$ defined by 
\begin{equation}\label{app_w}
\Af_{\omega_h}=\Af_h+\omega \If_h,
\end{equation} where $\If_h:\Hf_h\rightarrow \Hf_h$ is the identity operator. Also, define the discrete operator $\Bf_h:\Uf\longrightarrow \Hf_h$ as
\begin{align}\label{eqn:app B}
    \Bf_h=\Pi_h\Bf.
\end{align}
Denote $\Af_{\omega_h}^*$ and $\Bf_h^*$ as the adjoint of the operators $\Af_{\omega_h}$ and $\Bf_h,$ respectively.\\[2mm]
\noindent The approximating system for \eqref{eqn:main_system_PCE} (resp. \eqref{eqn: main shifted_PCE}) is 
		\begin{equation}\label{eq:appY'AYBu}
			\Yf_h'(t)=\Af_h\Yf_h(t)+\Bf_h u_h(t) \text{ for all }t>0,\quad \Yf_h(0)=\Yf_{0_h},
		\end{equation}
		\begin{align} \label{eqn:d-state-2}
			\Big(\text{resp. }\Yt{_h}'(t)=\Af_{\omega_h}\Yt{_h}(t)+\Bf_h\ut_h(t)\text{ for all }t>0, \quad \Yt{_h}(0)=\Yf_{0_h}\Big),
		\end{align}
\noindent where $\Yf_{0_h}\in \Hf_h$ is an approximation of $\Yf_0.$ The details of finite dimensional approximation are presented in Section \ref{sec:Appopp}. To obtain the feedback stabilizing control, an optimal control problem is considered. Define
\begin{align}
J_h(\Yt{_h},\ut_h):=\int_0^\infty\big( \|\Yt{_h}(t)\|^2+\|\ut_h(t)\|^2\big) \, dt,
\end{align}
$$E_{h{\Yf}_{0_h}}:=\{ \ut_h\in L^2(0,\infty; L^2(\Omega))\mid J_h(\Yt_h,\ut_h)<\infty,\text{ where }\Yt_h \text{ is solution of }\eqref{eqn:d-state-2} \},$$
and consider the discrete optimal control problem:
\begin{equation}\label{eqn:d-ocfun-2}
\min_{\ut_h\in E_{h{\Yf}_{0_h}}}J_h(\Yt{_h},\ut_h) \text{ subject to }\eqref{eqn:d-state-2}.
\end{equation}
In the next result, it is established that for each $h,$ the optimal control problem \eqref{eqn:d-ocfun-2} has a unique minimizer and the minimizing control is obtained in feedback form by solving a discrete algebraic Riccati equation posed on $\Hf_h.$ The proof is provided in Section \ref{sec:disRiccati}.

\begin{Theorem}[uniform stabilizability and discrete Riccati operator]\label{th:dro}
Let $\Af_{\omega_h}$ and $\Bf_h$ be as defined in \eqref{app_w} and \eqref{eqn:app B}, respectively. Then there exists $h_0>0,$ such that  for all $0<h<h_0$, the results stated below hold:
 \begin{itemize}
\item[(a)] There exists a unique, non-negative, self-adjoint  Riccati operator $\n\Pf_h\in \Lc(\Hf_h)$ associated with \eqref{eqn:d-state-2} that satisfies the discrete Riccati equation
\begin{align} \label{eqn:d-ARE_PCE}
\n\Af^*_{\omega_h}\Pf{_h}+\Pf{_h}\Af_{\omega_h}-\Pf{_h}\Bf_h\Bf^*_h\Pf{_h}+\If_h=0,\; \Pf{_h}=\Pf{_h}^*\geq 0 \text{ on }\Hf_h.
 \end{align}
 %where $\n\Pi_h^*\in \Lc(\Hf_h,\Hf_h^*).$
 \item[(b)] For any $\Yf_{0_h}\in \Hf_h$, there exists a unique optimal pair $(\Yf^\sharp_h,u^\sharp_h)$ for \eqref{eqn:d-ocfun-2},  
 where $\n\Yf^\sharp_h(t)$ is the solution of the corresponding closed loop system
 \begin{equation}\label{eq-d-cl-loop}
 \n\Yf{^\sharp_h}'(t)=(\Af_{\omega_h}-\Bf_h\Bf_h^*\Pf{_h})\Yf^\sharp_h(t) \text{ for all } t>0,\quad \Yf^\sharp_h(0)=\Yf_{0_h},
 \end{equation}
$u^\sharp_h(t)$ can be expressed in the feedback form as 
\begin{align}\label{eqn:uht}
   u^\sharp_h(t)=-\Bf_h^*\Pf_h\Yf^\sharp_h(t),
\end{align}
and 
 \begin{align} \label{eq:minvalfun-d}
 \displaystyle\n\min_{\ut_h\in E_{h{\Yf}_{0_h}}}J_h(\Yt{_h},\ut_h)=J_h(\Yf^\sharp_h, u^\sharp_h)=\left\langle \Pf_h \Yf_{0_h},\Yf_{0_h}\right\rangle.
\end{align} 
\item[(c)] The operator $\Af_{\omega_h,\Pf_h}:=\Af_{\omega_h}-\Bf_h\Bf_h^*\Pf_h$ generates a uniformly analytic semigroup $\{e^{t\Af_{\omega_h,\Pf_h}}\}_{t\ge 0}$ on $\Hf_h$ satisfying
$$\normalfont\|e^{t\Af_{{\omega_h,\Pf_h}}}\|_{\Lc(\Hf_h)}\leq
M_Pe^{-\omega_P t} \text{ for all } t>0, $$
for some positive constants $\normalfont\omega_P$ and $\normalfont M_P$ independent of $h$. 
\end{itemize}
\end{Theorem}

\noindent The main results of the paper on convergence and error estimates are stated next. The proofs are presented in Section \ref{sec:results}.

% \begin{Theorem}[error estimates for Riccati and cost functional]\label{th:main-conv-P}
% Let $\n\Pf\in \Lc(\Hf)$ (resp. $\n\Pf_h\in\Lc(\Hf_h)$) solve the Riccati equation \eqref{eqn:ARE} (resp. \eqref{eqn:d-ARE_PCE}). Then for any given $0<\epsilon<1,$ the estimates below hold:
% \begin{itemize}
% \item[(a)]  $\n\;\|\Pf-\Pf_h\Pi_h\|_{\Lc(\Hf)}\le C h^{2(1-\epsilon)} ,$   {\n (b) } $\n \left\vert  J(\Yf^\sharp,u^\sharp) - J_h(\Yf^\sharp_h,u^\sharp_h)\right\vert\le C h^{2(1-\epsilon)}, $
% \item[(c)] $
% \|\Bf^*\Pf - \Bf_h^*\Pf_h\Pi_h\|_{\Lc(\Hf,\Uf)}   \le C h^{2(1-\epsilon)} $ and {\n (d)} $\|\Bf^*\Pf-\Bf_h^*\Pf_h\|_{\Lc(\Hf_h,\Uf)} \le  Ch^{2(1-\epsilon)}.$
% \end{itemize}
% Here, the constant $C>0$ is independent of $h$ but depends on $\epsilon,\gamma$ and $\omega_P,$ where $\gamma$ and $\omega_P$ are as in Theorems \ref{th:stb cnt} and \ref{th:dro}, respectively.
% \end{Theorem}
\begin{Theorem}[error estimates for Riccati and cost functional]\label{th:main-conv-P}
Let $\n\Pf$ and $(\Yf^\sharp$, $u^\sharp)$, for any $\Yf_0\in \Hf$, be as obtained in Theorem \ref{th:stb cnt}. Let $h_0$, $\Pf_h$, and $(\Yf_h^\sharp$, $u^\sharp_h)$, for $\Yf_{0_h}=\Pi_h \Yf_0$, be as  in Theorem \ref{th:dro}. 
Then there exists $\tilde{h}_0\in (0, h_0)$ such that for any given $0<\epsilon<1,$ and for all $0<h<\tilde{h}_0$ the estimates below hold:
\begin{itemize}
\item[(a)]  $\n\;\|\Pf-\Pf_h\Pi_h\|_{\Lc(\Hf)}\le C h^{2(1-\epsilon)} ,$   {\n (b) } $\n \left\vert  J(\Yf^\sharp,u^\sharp) - J_h(\Yf^\sharp_h,u^\sharp_h)\right\vert\le C h^{2(1-\epsilon)}, $
\item[(c)] $
\|\Bf^*\Pf - \Bf_h^*\Pf_h\Pi_h\|_{\Lc(\Hf,\Uf)}   \le C h^{2(1-\epsilon)} $, and {\n (d)} $\|\Bf^*\Pf-\Bf_h^*\Pf_h\|_{\Lc(\Hf_h,\Uf)} \le  Ch^{2(1-\epsilon)}.$
\end{itemize}
Here, the constant $C>0$ is independent of $h$ but depends on $\epsilon$. 
\end{Theorem}

\begin{Theorem}[error estimates for stabilized solutions and stabilizing control]\label{th:main-conv-new}
Let $\gamma$, and $(\Yf^\sharp$, $u^\sharp)$, for any $\Yf_0\in \Hf$, be as obtained in Theorem \ref{th:stb cnt}. Let $h_0$, $\omega_P,$ and $(\Yf_h^\sharp$, $u^\sharp_h)$, for $\Yf_{0_h}=\Pi_h \Yf_0$, be as obtained in Theorem \ref{th:dro}. For any $\tilde{\gamma}$ satisfying $0<\tilde{\gamma}<\mathrm{min}\{\gamma, \omega_p\}$, there exists $\tilde{h}_0\in (0, h_0)$
such that for any $0<\epsilon<1$ and for all $0<h<\tilde{h}_0$, the following estimates hold: 
\begin{itemize}
\item[(a)] $\n\|\Yf^\sharp(t) - \Yf^\sharp_h(t)\|\le C h^{2(1-\epsilon)} \frac{e^{-\wt\gamma t}}{t} \|\Yf_0\|  $ for all $t>0,$  {\n (b)} $\n\|\Yf^\sharp(\cdot) - \Yf^\sharp_h(\cdot)\|_{L^2(0,\infty;\Hf)}\le C h^{1-\epsilon} ,$
\item[(c)] $\n\|u^\sharp(t) - u^\sharp_h(t)\|\le C h^{2(1-\epsilon)}  \frac{e^{-\wt\gamma t}}{t} \|\Yf_0\| $ for all $t>0,$ and {\n (d)} $\n\|u^\sharp(\cdot) - u^\sharp_h(\cdot)\|_{L^2(0,\infty;\Uf)}\le C h^{1-\epsilon} .$ 
\end{itemize}
Here, the constant $C>0$ is independent of $h$ but depends on $\epsilon,\gamma$, and $\omega_P.$
\end{Theorem}

\section{Continuous dynamics}\label{sec:cont dyn}
\noindent  In this section, we study the wellposedness and the exponential stabilizability of \eqref{eqn:main_system_PCE}. Recall $\Hf=\Lt\times \Lt$, and $ (\Af,D(\Af))$ from \eqref{eqdef-A_PCE}. The section starts with some preliminaries. In Subsection \ref{sub:sgp&pr}, analytic semigroup and well-posedness of \eqref{eqn:main_system_PCE} is studied while Subsection \ref{subsec:specana} describes the spectral analysis of the operator $\Af$ in $\Hf.$ Finally, the proof of Theorem \ref{th:stb cnt} is presented in Subsection \ref{subsec:Proof main 1}. \\[.2mm]

\noindent\textbf{Poincar\'{e} inequality (\cite{Kes}).} Let $\Omega$ be a bounded open set in $\mathbb{R}^n$, $n\in\mathbb{N}.$  Then there exists a positive constant $C_p=C_p(\Omega)$ such that
		\begin{equation}\label{poincare ineq}
			\|u\|\le C_p \|\nabla u\|\text{ for every }u\in \Hio.
		\end{equation}
\noindent In the next lemma, some results from operators on Banach space that will be used in the later analysis are stated.		
		
\begin{Lemma}[\n \cite{Lasiecka1}]\label{lem:id-3}
(a) Let $\mathfrak{T}$, $\mathfrak{T}_1$ and $\mathfrak{T}_2$ be bounded operators on a Banach space $\Hf$ such that $\|\mathfrak{T}\|_{\Lc(\Hf)}<1,$ $(\If-\mathfrak{T}_1)^{-1}$ and  $(\If-\mathfrak{T}_2)^{-1}$ exist in $\Lc(\Hf).$ Then the following results hold:
\begin{align}\label{id-1}
			\|(\If-\mathfrak{T})^{-1}\|_{\Lc(\Hf)}\le \left(1-\|\mathfrak{T}\|_{\Lc(\Hf)}\right)^{-1}\text{ and } (\If-\mathfrak{T}_1)^{-1}-(\If-\mathfrak{T}_2)^{-1}=(\If-\mathfrak{T}_1)^{-1}(\mathfrak{T}_1-\mathfrak{T}_2)(\If-\mathfrak{T}_2)^{-1}.
\end{align}
(b) Let $\mathfrak{T}_1$ and $\mathfrak{T}_2$ be two closed operators on a Hilbert space $\Hf$ such that 
$\mathfrak{T}_1$ is invertible with $\mathfrak{T}_1^{-1}\in \Lc(\Hf)$ and $\|\mathfrak{T}_1^{-1}(\mathfrak{T}_2-\mathfrak{T}_1)\|_{\Lc(\Hf)}<1.$ Then $\mathfrak{T}_2$ is invertible and $\mathfrak{T}_2^{-1}\in \Lc(\Hf).$
Furthermore, for two invertible operators $\mathfrak{T}_1$ and $\mathfrak{T}_2,$
\begin{align}\label{eqn:id-3}
    \mathfrak{T}_2^{-1}=\left(\If+\mathfrak{T}_1^{-1}(\mathfrak{T}_2-\mathfrak{T}_1) \right)^{-1}\mathfrak{T}_1^{-1}.
\end{align}
\end{Lemma}

\subsection{Analytic semigroup and its properties} \label{sub:sgp&pr}
\noindent Introduce $\Vf:=H^1_0(\Omega)\times H^1_0(\Omega)$ endowed with the norm $\left\|  \begin{pmatrix} \phi \\ \psi\end{pmatrix}\right\|_\Vf:= $ $\left(\|\nabla \phi\|^2 +\|\nabla\psi\|^2\right)^{1/2}.$ The weak formulation corresponding to \eqref{PCE_mod} seeks $\begin{pmatrix} y(\cdot)\\ z(\cdot)\end{pmatrix}\in\Vf$ such that 
\begin{align*}
    & \left\langle \frac{d}{dt}\begin{pmatrix} y(t)\\ z(t)\end{pmatrix} , \begin{pmatrix} \phi\\ \psi\end{pmatrix}\right\rangle+ a\left(\begin{pmatrix}y(t)\\z(t) \end{pmatrix},\begin{pmatrix}\phi\\ \psi \end{pmatrix}\right)= \left\langle \begin{pmatrix} u(t)\chi_{\mathcal{O}}\\ 0\end{pmatrix}, \begin{pmatrix} \phi\\ \psi\end{pmatrix}\right\rangle \text{ for all } \begin{pmatrix} \phi\\ \psi\end{pmatrix} \in \Vf, \text{ for all }t>0,\\
    &  \left\langle\begin{pmatrix} y(0)\\ z(0)\end{pmatrix} , \begin{pmatrix} \phi\\ \psi\end{pmatrix}\right\rangle=  \left\langle\begin{pmatrix} y_0\\ z_0\end{pmatrix} , \begin{pmatrix} \phi\\ \psi\end{pmatrix}\right\rangle,
\end{align*}
where the sesquilinear form $a(\cdot,\cdot)$ for all $\begin{pmatrix}y\\z \end{pmatrix}\in\Vf$ and $\begin{pmatrix}\phi\\ \psi \end{pmatrix}\in\Vf$ is defined by 
\begin{equation}\label{bilinear_a}
\begin{aligned}
			a\left(\begin{pmatrix}y\\z \end{pmatrix},\begin{pmatrix}\phi\\ \psi \end{pmatrix}\right):= & \eta_0 \langle \nabla y, \nabla \phi\rangle +\eta_1 \langle z, \phi\rangle +\nu_0\langle y, \phi\rangle +\beta_0 \langle\nabla z ,\nabla \psi\rangle +(\kappa+\nu_0)\langle z, \psi\rangle - \langle y, \psi\rangle .
\end{aligned} 
\end{equation}
		
\noindent		
We assume that the coefficients $\nu_0$ and $\eta_1$ in \eqref{PCE_mod} are such that
\begin{align}\label{eqdef:nuhat-nu0}
\nuh:=\nu_0-\frac{|\eta_1|+1}{2}>0.
\end{align}

%  \noindent For the choice of $\nu_0$ as in \eqref{eqdef:nuhat-nu0},  define
% \begin{align}\label{eqdef:nuhat}
% \nuh:=\nu_0-\frac{|\eta_1|+1}{2}>0.
% \end{align} 
\noindent Choose $\begin{pmatrix} y \\ z\end{pmatrix}=\begin{pmatrix} \phi \\ \psi\end{pmatrix}$ in \eqref{bilinear_a}, apply a Cauchy-Schwarz inequality and use \eqref{eqdef:nuhat-nu0} to obtain 
 \begin{align}\label{bilinear_acor}
			\Re\left(a\left(\begin{pmatrix}\phi\\ \psi \end{pmatrix},\begin{pmatrix}\phi\\ \psi \end{pmatrix}\right)\right) -\nuh\left\langle\begin{pmatrix}\phi \\ \psi\end{pmatrix}, \begin{pmatrix}\phi \\ \psi\end{pmatrix}\right\rangle\ge \alpha_0 \left(\|\nabla \phi\|^2 +\|\nabla\psi\|^2\right), \text{ for all } \begin{pmatrix}\phi \\\psi\end{pmatrix}\in \Vf,
\end{align}
where $\alpha_0=\min\{ \eta_0,\beta_0\}>0,$ and consequently $a(\cdot, \cdot)$ defined in \eqref{bilinear_a}  is coercive. \\[1mm]
\noindent Note that $a(\cdot,\cdot):\Vf\times \Vf\rightarrow \Cb$ is continuous. That is, there exists $\alpha_1>0$ such that
\begin{align} \label{bilinear_a_bdd}
\left\vert a\left(\begin{pmatrix}y\\ z \end{pmatrix},\begin{pmatrix}\phi\\ \psi \end{pmatrix}\right)\right\vert \le \alpha_1 \left(\|\nabla  y\|^2 +\|\nabla z\|^2\right)^{1/2}\left(\|\nabla \phi\|^2 +\|\nabla\psi\|^2\right)^{1/2}.
\end{align}

\noindent Further, observe that $(\Af, D(\Af))$ defined in \eqref{eqdef-A_PCE} satisfies
\begin{equation}\label{equw1}
			\begin{array}{l}
				D(\Af)=\left\lbrace\begin{pmatrix}y\\ z \end{pmatrix}\in \Vf \mid  \begin{pmatrix}\phi\\ \psi \end{pmatrix} \mapsto a\left(\begin{pmatrix}y\\ z \end{pmatrix},\begin{pmatrix}\phi\\ \psi \end{pmatrix} \right) \text{ is continuous in }  \Hf\right\rbrace, \\[2.mm]
				\Big\langle -\Af\begin{pmatrix}y\\ z \end{pmatrix} ,\begin{pmatrix}\phi\\ \psi \end{pmatrix}\Big\rangle = a\left(\begin{pmatrix}y\\ z \end{pmatrix},\begin{pmatrix}\phi\\ \psi \end{pmatrix}\right), \, \begin{pmatrix}y\\ z \end{pmatrix}\in D(\Af), \, \begin{pmatrix}\phi\\ \psi \end{pmatrix}\in \Vf .
			\end{array}
\end{equation}

\begin{Remark} \label{rem:chc-nu_0}
The condition \eqref{eqdef:nuhat-nu0} is not restrictive. We assume this condition to have the second term in the left hand side of \eqref{bilinear_acor} with a negative coefficient so that the associated operator $\Af$ is stable. 

\noindent
If \eqref{eqdef:nuhat-nu0} is not satisfied, that is $\nuh\le 0$, choosing $\nu'>-\nuh$, and defining the bilinear form 
$$a_{\nu'}\left(\begin{pmatrix}y\\z \end{pmatrix},\begin{pmatrix}\phi\\ \psi \end{pmatrix}\right):=a\left(\begin{pmatrix}y\\z \end{pmatrix},\begin{pmatrix}\phi\\ \psi \end{pmatrix}\right)+\nu'\left\langle\begin{pmatrix}y \\ z\end{pmatrix}, \begin{pmatrix}\phi \\ \psi\end{pmatrix}\right\rangle \text{ for all } \begin{pmatrix}y\\z \end{pmatrix}\in\Vf, \, \begin{pmatrix}\phi\\ \psi \end{pmatrix}\in\Vf,$$
where $a(\cdot, \cdot)$ is as in \eqref{bilinear_a}, we can have 
\begin{align}
			\Re\left(a_{\nu'}\left(\begin{pmatrix}\phi\\ \psi \end{pmatrix},\begin{pmatrix}\phi\\ \psi \end{pmatrix}\right)\right) -(\nu'+\nuh)\left\langle\begin{pmatrix}\phi \\ \psi\end{pmatrix}, \begin{pmatrix}\phi \\ \psi\end{pmatrix}\right\rangle\ge \alpha_0 \left(\|\nabla \phi\|^2 +\|\nabla\psi\|^2\right), \text{ for all } \begin{pmatrix}\phi \\\psi\end{pmatrix}\in \Vf.
\end{align}
Here $\nu'+\nuh>0$ because of the choice of $\nu'$. The operator associated to the bilinear form $a_{\nu'}(\cdot, \cdot)$ is $\widehat{\Af}:=\Af-\nu'\If$ with $D(\widehat{\Af})=D(\Af)$ on $\Hf$ and it can be shown that $\widehat{\Af}$ is stable. 
The case, when $\nuh\le 0$, can be handled by shifting \eqref{PCE_mod} in a way such that the corresponding linear operator is $\widehat{\Af}$
and thus all results of the article can be proved analogously. 
\end{Remark}

\noindent Recall from \cite{BDDM} the definition of an analytic semigroup on a Hilbert space. See {\n\cite[Definition 2.3, Chapter 1, Part II]{BDDM}}.
		
\begin{Definition}[analytic semigroup]\label{def:ana-sgp}
			Let $\n\{e^{t\Af}\}_{t\ge 0}$ be a strongly continuous semigroup on $\n\Hf$, with infinitesimal generator $\n\Af$. The semigroup $\n\{e^{t\Af}\}_{t\ge 0}$ is analytic if 
\begin{itemize}
\item[(i)] there exist $a\in \mathbb{R}$ and a  sector $\n\Sigma (a;\theta_a):=\{ a+re^{i\theta} \in \Cb\,|\, r>0,\, \theta\in (-\pi, \pi], \, |\theta|\ge \theta_a\}$ for some $\frac{\pi}{2}<\theta_a<\pi,$ such that the complement of the  sector $\n\Sigma(a;\theta_a),$ that is, $\n\Sigma^c (a;\theta_a)\subset \rho(\Af)$, and 
\item[(ii)] for all $\n\mu \in \Sigma^c(a;\theta_a),$ $\mu\neq a,$ $\n\|R(\mu,\Af)\|_{\Lc(\Hf)}\le \frac{C}{|\mu - a|},$ for some $C>0$ independent of $\mu.$
\end{itemize} 
\end{Definition}
\noindent In the next theorem, it is shown that $(\Af, D(\Af))$ defined in \eqref{eqdef-A_PCE} generates an analytic semigroup on $\Hf$.

\begin{Theorem}[resolvent estimate and analytic semigroup]\label{thm:ua-sem}
Let $\normalfont (\Af,D(\Af))$ be as defined in \eqref{eqdef-A_PCE} and $\nuh$ be as introduced in \eqref{eqdef:nuhat-nu0}. Then the following results hold.
\begin{itemize}
\item[(a)] There exists $\frac{\pi}{2}<\theta_0<\pi$ such that  $\n\Sigma^c(-\nuh;\theta_0)$  $\n\subset \rho(\Af)$ and
	\begin{equation}\label{eq:resolestA}
					\n\|R(\mu, \Af)\|_{\mathcal{L}(\Hf)} \le \frac{C }{|\mu+\nuh|} \text{ for all }\mu\in \Sigma^c(-\nuh; \theta_0), \, \mu\neq -\nuh,
	\end{equation} 
	for some $C>0$ independent of $\mu$.
\begin{figure}[ht!]
\begin{center}
\begin{tikzpicture}
\path[fill=black!15] (-1,0)--(-7,3.6)--(-7,-3.6)--cycle;
\draw[gray,dashed,thick,<->] (-7,0)--(3,0);
\draw[gray,dashed,thick,<->] (0,-4)--(0,4);
\draw[thick,->] (-1,0)--(-6,3);
\draw[thick,->] (-1,0)--(-6,-3);
\draw[thick,->] (-0.5,0) arc (0:150:0.5cm);
\node[] at (-0.9,0.6) {$\theta_0$};
\node[] at (0.25,-0.25){$O$};
\draw[thick,->] (-5,0) arc (180:150:4cm);
\draw[thick,->] (-5,0) arc (180:210:4cm);
\node[] at (-4.8,0.5){$\n\Sigma(-\nuh; \theta_0)$};
\node[] at (-1,0){$\bullet$};
\node[] at (-1,-0.3){$-\nuh$};
%\node[purple] at (0,-1){$r_\delta$};

\draw[thick,red, dashed,->] (-1,0.85)--(-5,4);
%\node[red] at (-1,-1.7){$-\widetilde{\nu}$};
%\node[red] at (-1,1.7){$-\widetilde{\nu}$};
\draw[thick,red, dashed,->] (-5,-4)--(-1,-0.85);
\draw[thick,red,dashed,thick,->] (-1,-0.85) arc (270:360:0.85cm);
\draw[thick,red,dashed] (-1,0.85) arc (90:0:0.85cm);
%\node[red] at (0.5,-1.7){$\Gamma$};
\node[red] at (-3.5,-2.7){$\Gamma_-$};
%\node[red] at (1,-0.3){$\Gamma_0$};
\node[red] at (-3.5,2.7){$\Gamma_+$};
\node[red] at (0,0.4){$\Gamma_0$};
%\node[red] at (-3,-3.7){$\Gamma_1$};
%\node[red] at (1,-0.3){$\Gamma_2$};
%\node[red] at (-3,3.7){$\Gamma_3$};
%\draw[thick,blue,dashed] (-1,-2) arc (270:180:2cm);
%\draw[thick,blue,dashed,->] (-1,2) arc (90:180:2cm);
%\draw[yellow] (-1,-3)--(-1,3);
%\draw[purple,->] (-1,0)--(0.6,-1.2);
\end{tikzpicture}
\end{center}
\caption{$\n\Sigma(-\nuh; \theta_0)$ and $\color{red}{\Gamma}=\Gamma_+\cup
\Gamma_-\cup\Gamma_0$} \label{fig:spec-t}
\end{figure}

\item[(b)] The operator $\normalfont (\Af,D(\Af))$ generates an analytic semigroup $\normalfont \{e^{t\Af}\}_{t\ge 0}$ on $\n\Hf$ and the operator $\n e^{t\Af}$ in $\n\mathcal{L}(\Hf)$ can be represented by 
	\begin{equation}\label{eqrepsemi}
					\n e^{t\Af}=\frac{1}{2\pi i}\int_\Gamma e^{\mu t} R(\mu ,\Af) \, d\mu \text{ for all } t>0,
	\end{equation}
where $\Gamma$ is any curve from $-\infty$ to $\infty$ and is entirely in $\n\Sigma^c(-\nuh; \theta_0)$.
\item[(c)] Furthermore, the semigroup $\normalfont \{e^{t\Af}\}_{t\ge 0}$ on $\n\Hf$ satisfies 
	\begin{equation} \label{eq-grthcond}
	\n \|e^{t \Af}\|_{\mathcal{L}(\Hf)}\le Ce^{-\nuh t} \text{ for all } t>0,
	\end{equation}
for some $C>0$.
\end{itemize}				
				
\end{Theorem} 

\begin{proof}
We identify a sector $\n\Sigma(-\nuh; \theta_0):=\{-\nuh+re^{i\theta}\,|\, r>0,\, \theta\in (-\pi, \pi], \, |\theta|\ge\theta_0\}$ (see Figure \ref{fig:spec-t}), for a suitable $\theta_0$ so that (a) holds. 
% The choice of $\theta_0$ comes  from the bounds for the real and imaginary parts of $\Big\langle -\Af\begin{pmatrix}\phi\\ \psi\end{pmatrix},\begin{pmatrix}\phi\\ \psi\end{pmatrix}\Big\rangle$ for $\begin{pmatrix}\phi\\ \psi\end{pmatrix}\in D(\Af).$ 
Set 
$ \theta_0 = \pi- \tan^{-1}\left( \frac{\alpha_1}{\alpha_0}\right),$
where $\alpha_0$ and $ \alpha_1$ are the positive constants from \eqref{bilinear_acor} and \eqref{bilinear_a_bdd}, respectively. Note that $\theta_0\in (\frac{\pi}{2},\pi)$, as $0<\tan^{-1}\left( \frac{\alpha_1}{\alpha_0}\right)<\frac{\pi}{2}$.\\
 \noindent (a) The proof is established in three steps. In the first step it is shown that $\{\mu\in \Cb\, |\, \Re(\mu)\ge -\nuh\}\subset \rho(\Af).$ The resolvent estimate in $\{\mu\in \Cb\, |\, \Re(\mu)\ge -\nuh\}$ is derived in Step 2. In Step 3, the estimates in Steps 1 \& 2 are extended in $\Sigma^c(-\nuh;\theta_0).$

\medskip
\noindent {\bf  Step 1. Resolvent for $ \Re (\mu) \ge -\nuh$, $\mu\neq -\nuh$.} Let $\mu\in \Cb$ with $\Re\, \mu\ge -\nuh$. 
First we show that $(\mu\If-\Af):D(\Af)\longrightarrow \Hf$ is invertible, that is, for any given $\begin{pmatrix}f\\g\end{pmatrix}\in \Hf$ there exists a unique $\begin{pmatrix}y_\mu \\ z_\mu\end{pmatrix}\in D(\Af)$ such that 
\begin{equation} \label{eqn:invert_PCE}
(\mu \If-\Af)\begin{pmatrix}y_\mu\\z_\mu\end{pmatrix}=\begin{pmatrix}f\\g\end{pmatrix} \text{ for all }\mu\in \Cb \text{ with } \Re(\mu)\ge -\nuh.
\end{equation}
The weak formulation that corresponds to \eqref{eqn:invert_PCE} is 
\begin{align} \label{eqWeakform}
a\left( \begin{pmatrix} y_\mu \\ z_\mu \end{pmatrix} , \begin{pmatrix} \phi \\ \psi \end{pmatrix}\right)+\mu \left\langle \begin{pmatrix}  y_\mu \\ z_\mu\end{pmatrix}, \begin{pmatrix} \phi \\ \psi \end{pmatrix}\right\rangle = \left\langle \begin{pmatrix}  f \\ g\end{pmatrix}, \begin{pmatrix} \phi \\ \psi \end{pmatrix}\right\rangle \text{ for all } \begin{pmatrix} \phi \\ \psi \end{pmatrix}\in \Vf.
\end{align}
Since, $\Re(\mu)\ge -\nuh,$ \eqref{bilinear_acor} and Lax-Milgram theorem \cite{RDJLL} imply the existence of a unique $\begin{pmatrix} y_\mu \\ z_\mu \end{pmatrix} \in \Vf$ such that \eqref{eqWeakform} holds.
% Note that, \eqref{eqn:invert_PCE} can be re-written as
% \begin{subequations} \label{eq:re-write w6}
% \begin{align}  
% 	& -\eta_0\Delta  y_\mu+(\mu +\nu_0)y_\mu+\eta_1 z_\mu  =f\text{ in }\Omega,\,  y_\mu=0 \text{ on }\partial\Omega,\\  
% 	& -\beta_0 \Delta z_\mu +(\kappa+\nu_0+\mu)z_\mu-y_\mu=g \text{ in }\Omega, \,  z_\mu=0 \text{ on }\partial\Omega, \label{eq:re-write w6b}
% \end{align}
% \end{subequations}	
From the $H^2$-regularity result for elliptic equations, it follows that $\begin{pmatrix} y_\mu \\ z_\mu \end{pmatrix} \in D(\Af).$

\noindent 
Now, for $\mu$ with $\Re(\mu)\ge -\nuh,$ the choice $\begin{pmatrix}
\phi\\ \psi\end{pmatrix}=\begin{pmatrix}
y_\mu \\ z_\mu \end{pmatrix}$ in \eqref{eqWeakform}, \eqref{bilinear_acor}, and \eqref{poincare ineq} show
\begin{equation} \label{eqValreltn-yzfg}
(\|y_\mu\|^2+\|z_\mu\|^2)^{1/2} \le C \left( \|f\|^2+\|g\|^2\right)^{1/2},
\end{equation}
for some positive constant $C=C(\alpha_0,C_p).$

\medskip
\noindent 
{\bf Step 2. Resolvent estimate for $\Re\mu \ge -\nuh, \, \mu\neq -\nuh$.}
Let $\mu=-\nuh+\rho e^{i\theta},$ $(\rho\neq 0)$ where $-\frac{\pi}{2}\le \theta\le \frac{ \pi}{2}$  and choose $\begin{pmatrix} \phi \\ \psi \end{pmatrix}= e^{i\frac{\theta}{2}} \begin{pmatrix} y_\mu \\ z_\mu \end{pmatrix}$ in \eqref{eqWeakform} to obtain
\begin{align} \label{eqn:bil_con}
a\left( \begin{pmatrix} y_\mu \\ z_\mu \end{pmatrix} , e^{i\frac{\theta}{2}} \begin{pmatrix} y_\mu \\ z_\mu \end{pmatrix}\right)+(-\nuh+\rho e^{i\theta}) \left\langle \begin{pmatrix}  y_\mu \\ z_\mu\end{pmatrix}, e^{i\frac{\theta}{2}} \begin{pmatrix} y_\mu \\ z_\mu \end{pmatrix}\right\rangle = \left\langle \begin{pmatrix}  f \\ g\end{pmatrix}, e^{i\frac{\theta}{2}} \begin{pmatrix} y_\mu \\ z_\mu \end{pmatrix}\right\rangle.
\end{align}
Note that the definition of $a(\cdot,\cdot)$ from \eqref{bilinear_a} shows
\begin{align*}
   &  \Re\left( a\left( \begin{pmatrix} y_\mu \\ z_\mu \end{pmatrix} , e^{i\frac{\theta}{2}} \begin{pmatrix} y_\mu \\ z_\mu \end{pmatrix}\right)+(-\nuh+\rho e^{i\theta}) \left\langle \begin{pmatrix}  y_\mu \\ z_\mu\end{pmatrix}, e^{i\frac{\theta}{2}} \begin{pmatrix} y_\mu \\ z_\mu \end{pmatrix}\right\rangle \right)\\
   & \qquad = \cos(\theta/2) \left( \eta_0\|\nabla y_\mu\|^2+(\nu_0-\nuh+\rho)\|y_\mu\|^2+\beta_0\|\nabla z_\mu\|^2 +(\kappa+\nu_0-\nuh+\rho)\|z_\mu\|^2\right) \\
   & \qquad \qquad +\Re\left(  e^{-i\frac{\theta}{2}} (\eta_1 \langle z_\mu, y_\mu\rangle -\langle y_\mu, z_\mu\rangle)\right).
\end{align*}
This, \eqref{eqn:bil_con}, and a Cauchy-Schwarz inequality followed by a use of \eqref{eqValreltn-yzfg} lead to 
\begin{align*}
     \cos(\theta/2) \left( (\nu_0-\nuh+\rho)\|y_\mu\|^2 +(\kappa+\nu_0-\nuh+\rho)\|z_\mu\|^2\right) %\\
    % & \qquad  \le \left\vert \Re\left( \left\langle \begin{pmatrix}  f \\ g\end{pmatrix}, e^{i\frac{\theta}{2}} \begin{pmatrix} y_\mu \\ z_\mu \end{pmatrix}\right\rangle_{\Hf} \right)\right\vert +\left\vert \Re\left(  e^{-i\frac{\theta}{2}} (\eta_1 \langle z_\mu, y_\mu\rangle -\langle y_\mu, z_\mu\rangle)\right)\right\vert\\
    %& \qquad
    \le C (\|f\|^2+\|g\|^2)^{1/2}(\|y_\mu\|^2+\|z_\mu\|^2)^{1/2},%+\frac{|\eta_1|+1}{2}\cos(\theta/2)\left(\|y_\mu\|^2+\|z_\mu\|^2\right),
\end{align*}
for some $C=(\eta_1, \alpha_0, C_p)>0.$ Since $\nu_0>\nuh$ (see \eqref{eqdef:nuhat-nu0}) and $\cos(\theta/2)\ge \cos(\pi/4)>0$ for all $-\frac{\pi}{2}\le \theta\le \frac{\pi}{2}$,
\begin{align*}
     (\|y_\mu\|^2+\|z_\mu\|^2)^{1/2} \le \frac{C}{\rho \cos(\pi/4)} (\|f\|^2+\|g\|^2)^{1/2},
\end{align*}
holds and thus, noting $\rho=|\mu+\nuh|,$ it follows that for all $\Re\mu \ge -\nuh, \, \mu\neq -\nuh$,
\begin{align*}
    \left\|R(\mu,\Af)\begin{pmatrix} f \\ g\end{pmatrix} \right\|=(\|y_\mu\|^2+\|z_\mu\|^2) ^{1/2} \le \frac{C}{|\mu+\nuh|} (\|f\|^2+\|g\|^2)^{1/2}.
\end{align*}

\medskip
\noindent {\bf Step 3. Case of any $\mu\in\Sigma^c(-\nuh;\theta_0)$ with $\Re \mu<-\nuh$.} %Let $\mu=|\mu|e^{i\theta_\mu}$ with $\frac{\pi}{2}<\theta_\mu <\theta_0$ and $\mu_0=i|\mu|=|\mu|e^{i\frac{\pi}{2}}.$
		Using the fact that $\displaystyle\lim_{\phi\rightarrow \frac{\pi}{2}}e^{i (\phi-\frac{\pi}{2})}=1$, there exists a $\delta_1>0$ such that 
		\begin{equation}\label{eq:resol4}
			\vert1-e^{i(\phi-\frac{\pi}{2})}\vert\le \frac{1}{2C} \text{ for all } \frac{\pi}{2}\le \phi\le \frac{\pi}{2}+\delta_1,
		\end{equation}
		where $C$ is the constant obtained in Step 2. For any $\mu$ such that  $\mu+\nuh=|\mu+\nuh|e^{i\theta}$ with $\frac{\pi}{2}<\theta \le \frac{\pi}{2}+\delta_1,$
		choosing $\mu_0+\nuh=i|\mu+\nuh|$, we obtain from Step 2 that $\mu_0\in \rho(\Af)$ and 
		$\|R(\mu_0,\Af)\|_{\Lc(\Hf)}\le \frac{C}{|\mu_0+\nuh|}=\frac{C}{|\mu+\nuh|}$.
		
		\noindent Note that $|\mu-\mu_0|=|\mu+\nuh||1-e^{i(\theta-\frac{\pi}{2})}|$ and  
		\begin{align*}
			\|(\mu-\mu_0)R(\mu_0,\Af)\|_{\Lc(\Hf)}\le |\mu-\mu_0|\frac{C}{|\mu_0+\nuh|}=C|1-e^{i(\theta_\mu-\frac{\pi}{2})}|\le \frac{1}{2}.
		\end{align*}
		Now, choosing $\mathfrak{T}_1=\mu_0\If-\Af$ and $\mathfrak{T}_2=\mu\If-\Af$, Lemma \ref{lem:id-3}(b) yields that $R(\mu,\Af)$ exists in $\Lc(\Hf)$ for each $\mu=-\nuh+r e^{\pm i\theta},$ for any $r>0$ and $\frac{\pi}{2}<\theta\le \frac{\pi}{2}+\delta_1,$ and \eqref{id-1} with $\mathfrak{T}=(\mu-\mu_0)R(\mu_0,\Af)$ leads to
		$$\|R(\mu,\Af)\|_{\Lc(\Hf)}\le \frac{1}{1-\|(\mu-\mu_0)R(\mu_0,\Af)\|_{\Lc(\Hf)}} \left\|R(\mu_0,\Af)\right\|_{\Lc(\Hf)}\le \frac{2C}{|\mu+\nuh|}.$$

\noindent Let $\theta_0=\frac{\pi}{2}+\delta_0$ and $n_0$ be the largest natural number such that $n_0\delta_1\le \delta_0< (n_0+1)\delta_1$. Now repeating the same argument for $\mu \in \Sigma^c(-\nuh;\theta_0)$ with 
$\mu=-\nuh+r e^{\pm i\theta}$ for $r>0$ and $\frac{\pi}{2}+n \delta_1 < \theta \le \frac{\pi}{2}+(n+1)\delta_1$, for each $n=1, \cdots n_0-1$, we obtain $\|R(\mu,\Af)\|_{\Lc(\Hf)}\le  \frac{2^{n}C}{|\mu+\nuh|}$ and  finally for  $\mu=-\nuh+r e^{\pm i\theta}$ with $\frac{\pi}{2}+n_0 \delta_1 < \theta \le \frac{\pi}{2}+\delta_0$,  
$\|R(\mu,\Af)\|_{\Lc(\Hf)}\le  \frac{2^{n_0}C}{|\mu+\nuh|}$ holds.
 Thus, for $C_0:=\max\{C, 2^n C, n=1, \cdots, n_0\}$, where $C$ is the positive constant obtained in Step 2, for all $\mu \in \Sigma^c(-\nuh; \theta_0)$, the estimate \eqref{eq:resolestA} holds: 
$$ \|R(\mu,\Af)\|_{\Lc(\Hf)}\le  \frac{C_0}{|\mu+\nuh|}, \quad \mu \in \Sigma^c(-\nuh; \theta_0), \quad \mu\neq -\nuh.$$

\noindent (b) Using the fact that 
$-\Delta:D(-\Delta)\subset L^2(\Omega)\longrightarrow L^2(\Omega)$, where $D(-\Delta)=H^2(\Omega)\cap H^1_0(\Omega),$ is a closed and densely defined operator in $L^2(\Omega)$, it can be obtained that  $(\Af,D(\Af))$ is densely defined and closed operator on $\Hf$. Thus, using this along with part $(a)$, it follows that $(\Af, D(\Af))$ generates an analytic semigroup $\{e^{t\Af}\}_{t\ge 0}$ on $\Hf$ with the representation \eqref{eqrepsemi} \cite[Theorem 2.10, Chapter 1, Part II]{BDDM}. 

\medskip
\noindent (c) Choose $\Gamma=\Gamma_{\pm}\cup\Gamma_0$ with $\Gamma_{\pm}=\{-\nuh+re^{\pm i\phi_0}\,| \, r\ge r_0\} $ and $\Gamma_0=\{-\nuh+ r_0 e^{\pm i \vartheta}\,|\, |\vartheta|\le \phi_0\}$ for some $r_0>0$ and $\frac{\pi}{2}<\phi_0<\theta_0 < \pi.$ 
% and a detailed proof is added in Appendix \ref{pf-thm:ua-sem-c}.
Then \eqref{eqrepsemi} yields
\begin{align*}
e^{t\Af}=\frac{1}{2\pi i}\int_\Gamma e^{\mu t} R(\mu,\Af) \, d\mu = \frac{1}{2\pi  i}\int_{t\Gamma^1}e^{\mu_1-\nuh t}R(\frac{\mu_1}{t}-\nuh,\Af)\, \frac{d\mu_1}{t},
\end{align*}
where the last equality is obtained using the substitution $\mu_1=(\mu+\nuh)t$ and $\Gamma^1=\Gamma^1_\pm\cup\Gamma^1_0$ with $\Gamma^1_\pm=\{re^{\pm i\phi_0}\,| \, r\ge r_0\}$
 and $\Gamma^1_0=\{r_0 e^{i\vartheta},\, |\vartheta|\le \phi_0\}.$ Utilizing the fact that the above integral is independent of such path $\Gamma$, we obtain
\begin{equation}\label{eqint}
 \begin{aligned}
e^{t\Af}  =\frac{e^{-\nuh t}}{2\pi t i}\int_{\Gamma^1}e^{\mu_1}R(\frac{\mu_1}{t}-\nuh,\Af)\, d\mu_1 
\end{aligned}.
\end{equation}
We first estimate it over $\Gamma_+^1$ (then similarly on $\Gamma_-^1$) using (a) and observe that $\cos{\phi_0}<0$ to obtain
\begin{align*}
   \left\| \int_{\Gamma^1_+}e^{\mu_1}R(\frac{\mu_1}{t}-\nuh,\Af)\, d\mu_1 \right\|_{\Lc(\Hf)}  = \left\| \int_{r_0}^\infty e^{re^{\pm i \phi_0}}R(\frac{re^{\pm i \phi_0}}{t}-\nuh,\Af)\, e^{\pm i \phi_0} \, dr  \right\|_{\Lc(\Hf)} \le 
   %\frac{t}{r_0}\int_{r_0}^\infty e^{r\cos{\phi_0}} dr=
   -\frac{ Ct e^{r_0\cos{\phi_0}}}{r_0\cos(\phi_0)}. 
\end{align*}
To estimate over $\Gamma^1_0,$ where $\mu_1=r_0e^{i\vartheta},$ again use (a) to obtain
\begin{align*}
   \left\| \int_{\Gamma^1_0}e^{\mu_1}R(\frac{\mu_1}{t}-\nuh,\Af)\, d\mu_1 \right\|_{\Lc(\Hf)}  = \left\| \int_{-\phi_0}^{\phi_0} e^{r_0e^{i\vartheta}}R(\frac{r_0e^{i\vartheta}}{t}-\nuh,\Af)\, r_0e^{ i\vartheta} i \,d\vartheta  \right\|_{\Lc(\Hf)} \le %\frac{t}{r_0}\int_{-\phi_0}^{\phi_0} e^{r_0\cos{\vartheta}} d\vartheta \le 
   \frac{2C t \phi_0 e^{r_0}}{r_0}. 
\end{align*}
Use the last two inequalities in \eqref{eqint} to deduce \eqref{eq-grthcond} and to conclude the proof.
\end{proof}

\noindent Since, $(\Af,D(\Af))$ generates an analytic semigroup $\{e^{t\Af}\}_{t\ge 0}$ of negative type on $\Hf,$ the well-posedness of \eqref{eqn:main_system_PCE} follows. 
 
\begin{Lemma}[well-posedness]\label{lem:cont-1} For any given $\n\Yf_0\in \Hf$ and $\n F\in L^2(0,\infty; \Hf)$, the system  $\n\Yf'(t)=\Af\Yf(t)+F(t)$ for all $t>0,$ $\Yf(0)=\Yf_0$
			admits  a unique solution  $\n\Yf(\cdot)\in C([0,\infty); \Hf)$ with the representation $\n\Yf(t)=e^{t\Af}\Yf_0+\int_0^te^{(t-s)\Af}F(s)\,ds\text{ for all } t>0.$
\end{Lemma}
\noindent The proof is standard, for example, see \cite[Prop. 3.1, Ch-1, Part-II]{BDDM}.\\[1mm]
\noindent In the next remark, a regularity result for $R(-\nuh,\Af)$ is studied. 
\begin{Remark}[regularity result] \label{rem:RegResult}
For $\mu=-\nuh,$ from Step 1 of the proof of Theorem \ref{thm:ua-sem} and the $H^2$-regularity result for elliptic equations, we have $R(-\nuh,\Af)\in\Lc(\Hf,  D(\Af)) $ and 
\begin{align}\label{eqRegEsr-yzinH2}
\left\|R(-\nuh,\Af)\begin{pmatrix}f\\ g\end{pmatrix}\right\|_{H^2(\Omega)\times H^2(\Omega)} \le C (\|f\|+\|g\|),
\end{align}
for some $C>0$ and for all $\begin{pmatrix}f\\ g\end{pmatrix}\in \Hf.$
\end{Remark}

\noindent The adjoint operator $\n (\Af^*,D(\Af^*))$ corresponding to $\n (\Af, D(\Af))$ is defined as 
\begin{equation}\label{eqadjopt_PCE}
			\Af^*:=\left( \begin{matrix}
				\eta_0 \Delta -\nu_0 I & I \\ -\eta_1 I  & \beta_0 \Delta-(\kappa+\nu_0) I 
			\end{matrix} \right) \text{ and } D(\Af^*):=\left( H^2(\Omega)\cap \Hio\right)^2.
\end{equation}

\noindent Note that $(\Af^*,D(\Af^*))$ generates a strongly continuous semigroup, $\rho(\Af)=\rho(\Af^*)$ and $\|R(\mu,\Af)\|_{\Lc(\Hf)}=\|R(\overline{\mu},\Af^*)\|_{\Lc(\Hf)}$ for all $\mu\in \rho(\Af^*)$ \cite[Proposition 2.4, Ch. 1, Part II]{BDDM}. Therefore, using Theorem \ref{thm:ua-sem}, the next lemma shows that the adjoint operator $(\Af^*,D(\Af^*))$ generates an analytic semigroup on $\Hf.$

\begin{Lemma}[analytic semigroup by $(\Af^*,D(\Af^*))$] \label{lem:RegResult-Af*}
Let $\n (\Af^*,D(\Af^*))$ be as defined in \eqref{eqadjopt_PCE}. Then the results below hold:
\begin{itemize}
\item[(a)] The set $\n \Sigma^c(-\nuh;\theta_0)$ is contained in the resolvent set $\n \rho(\Af^*),$ and for all $\n\mu\in \Sigma^c(-\nuh;\theta_0),$ the resolvent satisfies
$
\n\|R(\mu,\Af^*)\|_{\Lc(\Hf)} \le \frac{C}{|\mu+\nuh|}, \, \mu\neq -\nuh,
$
for some $C>0$ independent of $\mu$. 
Therefore, $\n\Af^*$ generates an analytic semigroup $\n\{e^{t\Af^*}\}_{t\ge 0}$ on $\n\Hf$ satisfying $\|\n e^{t\Af^*}\|_{\Lc(\Hf)}\le C e^{-\nuh t}$ for all $t>0,$ for some $C>0$. 
\item[(b)] For $\mu=-\nuh,$ $R(-\nuh,\Af^*)\in \Lc(\Hf,D(\Af^*))$ and for any  $\n\begin{pmatrix} p \\q \end{pmatrix}\in \Hf,$ there exists $C>0$ such that 
\begin{align*}
\left\| R(-\nuh,\Af^*)\begin{pmatrix} p \\q \end{pmatrix}\right\|_{H^2(\Omega)\times H^2(\Omega)} \le C \left( \|p\| + \|q\|\right).
\end{align*}
\end{itemize}
\end{Lemma}

\noindent Using the regularity result for an analytic semigroup with negative type (for details, see \cite[Proposition 3.13, Section 3.6, Chapter I,Part II]{BDDM}), we have the next result. It justifies the equivalency of the semigroup formulation and weak formulation of \eqref{eqn:main_system_PCE}. The weak formulation will be used in Section \ref{sec:Appopp} to define the finite dimensional approximation of the system. 

\begin{Lemma}[solution regularity] \label{lem:semidiscre-C}
Let $u\in L^2(0,\infty;\Lt)$ and $\n\Yf_0=\begin{pmatrix}y_0\\z_0\end{pmatrix}\in \Hf$ be given. Then the solution $\n\Yf(t)=\begin{pmatrix}y(t)\\z(t)\end{pmatrix}$ of \eqref{eqn:main_system_PCE} obtained in Lemma \ref{lem:cont-1} belongs to $\n C([0,\infty);\Hf)\cap H^1(\epsilon,\infty;\Hf)\cap  L^2(\epsilon,\infty; D(\Af))$ for all $\epsilon>0$ and for all $\n \begin{pmatrix} \phi \\ \psi \end{pmatrix}\in \Vf,$ $\Yf(t)$ satisfies
\begin{equation}\label{eq-weak-form}
    \begin{aligned}
        \frac{d}{dt}\langle y(t), \phi\rangle & = -\eta_0 \langle \nabla y(t),\nabla \phi\rangle - \eta_1\langle  z(t), \phi\rangle - \nu_0\langle y(t), \phi\rangle +  \langle u(t)\chi_{\mathcal{O}},\phi\rangle \text{ a.e. }t\in(0,\infty)\\
        \frac{d}{dt}\langle z(t), \psi\rangle  & = -\beta_0\langle \nabla z(t), \nabla \psi\rangle - (\kappa+\nu_0) \langle z(t), \psi\rangle + \langle y(t), \psi\rangle \text{ a.e. }t\in(0,\infty)\\
        \langle y(0),\phi \rangle & =\langle y_0,\phi\rangle, \quad \langle z(0),\psi \rangle = \langle  z_0, \psi\rangle.
    \end{aligned}
\end{equation}
\end{Lemma}

\subsection{Spectral analysis} \label{subsec:specana}
In this subsection, the spectral analysis of the operator $\Af$ on $\Hf$ is discussed. Note that Theorem \ref{thm:ua-sem}(a) implies that $\sigma(\Af)$, the spectrum of $\Af$, is a subset of $\Sigma(-\nuh;\theta_0).$ Moreover, Remark \ref{rem:RegResult} gives that $(-\nuh\If-\Af)^{-1}\in \mathcal{L}(\Hf, D(\Af))$, is a linear, bounded, compact operator in $\Hf$. Thus, using 
\cite[Theorem 6.26 and Theorem 6.29, Chapter 3]{Kato}, we obtain the next result. 
\begin{Theorem}[properties of spectrum of $\Af$] \label{th:spec Af}
Let $\n (\Af,D(\Af))$ be as defined in \eqref{eqdef-A_PCE} and $\Sigma(-\nuh;\theta_0)$ be as in Theorem \ref{thm:ua-sem}. Then
\begin{itemize}
    \item[(a)] the spectrum of $\Af,$ $\sigma(\Af) \subset \Sigma(-\nuh;\theta_0),$ 
    \item[(b)] the set $\sigma(\Af)$ contains only isolated eigenvalues of $\Af$ and if there exists a convergent sequence $\{\Lambda_k\}_{k\in\Nb}\subset \sigma(\Af),$ then $\Lambda_k\rightarrow -\infty$ as $n\rightarrow \infty.$
\end{itemize}
\end{Theorem}

\noindent We mention that $\sigma(\Af^*)$, the spectrum of the adjoint operator $\Af^*$, is the same as $\sigma(\Af)$.

\medskip
\noindent In the next proposition, we provide the expression of eigenvalues and eigenvectors of the operators $(\Af,D(\Af))$ and $(\Af^*,D(\Af^*))$ utilizing the eigenvalue problem for the Laplace operator \cite{Kes}:\\
 There exists an orthonormal basis $\{\phi_k\}_{k\in\Nb}$ of $\Lt$ and a sequence of positive real numbers $\{\lambda_k\}_{k\in\Nb}$ with $\lambda_k \rightarrow \infty $ as $k\rightarrow \infty$ such that
\begin{equation}\label{eigvalLapl}
\begin{aligned}
& 0<\lambda_1\le \lambda_2\le \cdots \le \lambda_k\le \cdots, \\
& -\Delta\phi_k =\lambda_k \phi_k \text{ in }\Omega,\\
& \phi_k \in \Hio\cap C^\infty(\Omega).
\end{aligned}
\end{equation} 
Note that $\Hf=\mathrm{span}_{k\in \mathbb{N}}\,\left\lbrace \begin{pmatrix}\phi_k \\ 0\end{pmatrix}, \begin{pmatrix}0 \\ \phi_k
\end{pmatrix} \right\rbrace$, and for all $k\in \Nb,$  $\mathrm{span}\,\left\lbrace \begin{pmatrix}\phi_k \\ 0\end{pmatrix}, \begin{pmatrix}0 \\ \phi_k
\end{pmatrix}\right\rbrace$ is invariant under $\Af.$ Restricting $\Af$ on $\mathrm{span}\,\left\lbrace \begin{pmatrix}\phi_k \\ 0\end{pmatrix}, \begin{pmatrix}0 \\ \phi_k
\end{pmatrix}\right\rbrace,$ for each ${k\in \Nb}$, we derive the characterstic polynomial of $\Af$ as 
\begin{align*}
\Lambda_k^2+\left( \eta_0\lambda_k +\nu_0+\beta_0\lambda_k+(\kappa+\nu_0) \right)\Lambda_k + \left( \eta_0\lambda_k+\nu_0\right)\left( \beta_0\lambda_k+(\kappa+\nu_0)\right) +\eta_1=0.
\end{align*}
Denoting the roots of the above equation by $\Lambda_k^\pm,$ the eigenvalues of $\Af$ are obtained.
 \begin{Proposition}[eigenpairs for $\Af$ and $\Af^*${\n\cite[Proposition 3.3]{WKR}}] \label{pps:spec Af}
  Let $\n (\Af,D(\Af))$ and $(\Af^*,D(\Af^*))$ be as defined in \eqref{eqdef-A_PCE} and \eqref{eqadjopt_PCE}, respectively. Then the results below hold:
\begin{itemize}
    \item[(a)] The eigenvalues of $\n\Af$ consist of two sequences $\{\Lambda_k^+\}_{k\in \Nb}$ and $\{\Lambda_k^-\}_{k\in \Nb}$ with the expressions
    \begin{align} \label{eqExpEigA}
    \Lambda_k^\pm = -\frac{1}{2}\left((\eta_0+\beta_0)\lambda_k+\kappa+2\nu_0 \right)\pm \frac{1}{2} \sqrt{\left( (\beta_0-\eta_0)\lambda_k+\kappa  \right)^2-4\eta_1},
    \end{align}
    where $\{\lambda_k\}_{k\in\Nb}$ is the family of eigenvalues of $-\Delta$ shown in  \eqref{eigvalLapl}. The eigenvalues of $\Af^*$ also consist of two sequences $\{\overline{\Lambda_k^+}\}_{k\in \Nb}$ and $\{\overline{\Lambda_k^-}\}_{k\in \Nb}$ with $\Lambda_k^\pm$ from \eqref{eqExpEigA}.
    \item[(b)] For the case of simple eigenvalues and $\Lambda_k^+\neq \Lambda_k^-;$ the eigenfunctions denoted by $\xi_k^\pm$ corresponding to $\Lambda_k^\pm$ of $\Af$ and $\xi_k^{\pm*}$ corresponding to $\overline{\Lambda_k^\pm}$ of $\Af^*$ are 
    \begin{align} \label{eqEigFunAf*}
        \xi_k^\pm=\begin{pmatrix} 1 \\ \frac{1}{\Lambda_k^\pm+\beta_0\lambda_k+\kappa+\nu_0} \end{pmatrix} \phi_k \text{ and } \xi_k^{\pm*}=\begin{pmatrix} 1 \\-\frac{\eta_1}{\overline{\Lambda_k^\pm}+\beta_0\lambda_k+\kappa+\nu_0} \end{pmatrix} \phi_k \text{ for all }k\in \Nb.
    \end{align}
    \item[(c)] For the case of multiple eigenvalues and $\Lambda_k^+= \Lambda_k^-=\Lambda_k;$ the eigenfunctions denoted by $\xi_k^\pm$ corresponding to $\Lambda_k^+=\Lambda_k^-$ of $\Af$ are 
    \begin{align*}
         \xi_k^+=\begin{pmatrix} 1 \\ \frac{1}{\Lambda_k+\beta_0\lambda_k+\kappa+\nu_0} \end{pmatrix} \phi_k, \text{ and } \xi_k^-=\begin{pmatrix} 1 \\ 0 \end{pmatrix} \phi_k ,
    \end{align*}
    and the eigenfunctions denoted by $\xi_k^{\pm*}$ corresponding to $\overline{\Lambda_k^+}=\overline{\Lambda_k^-}$ of $\Af^*$ are 
    \begin{align} \label{eqEigFunAf*-m}
         \xi_k^{+*}=\begin{pmatrix} 1 \\-\frac{\eta_1}{\overline{\Lambda_k}+\beta_0\lambda_k+\kappa+\nu_0} \end{pmatrix} \phi_k \text{ and } \xi_k^{-*}=\begin{pmatrix} 1 \\0 \end{pmatrix} \phi_k.
    \end{align}
\end{itemize}
 \end{Proposition}

\subsection{Proof of Theorem \ref{th:stb cnt}}  \label{subsec:Proof main 1}
% Theorem \ref{th:stb cnt} follows from \cite[Theorem 3.1, Remark 3.1 and Corollary 4.2, Part-V, Ch-1]{BDDM}, provided Theorem \ref{th:stbopen} stated below holds.
%$(\Aw, \Bf)$ is stabilizable in $\Hf$. 
%The stabilizability of $(\Aw, \Bf)$ is proved in Theorem \ref{th:stbopen}. 
Let $\omega>0$ be any given number. Let $\Aw$ and $\Bf$ be as defined in \eqref{Aw} and \eqref{eqcontrol_PCE}, respectively. The pair $(\Aw,\Bf)$  is said to be {\it open loop stabilizable} if there exists $\ut\in L^2(0,\infty;\Uf)$ such that the corresponding solution $\Yt(t)$ of \eqref{eqn: main shifted_PCE} with $\ut$  satisfies $\Yt\in L^2(0,\infty;\Hf).$ To prove the open loop stabilizability of $(\Aw,\Bf),$ it is enough to prove Hautus condition given in  \eqref{eq:Hautus:PCE} \cite[Proposition 3.1, Ch. 1, Part - V]{BDDM}. This result is instrumental in the proof of Theorem \ref{th:stb cnt}. 

\medskip
\noindent The spectrum of $\Aw,$ $\sigma(\Aw):= \{\Lambda_n^\pm+\omega\, |\, \Lambda_n^\pm \in \sigma(\Af),\, n\in \Nb\}.$ Since $\sigma(\Af)\subset \Sigma(-\nuh;\theta_0)$, for any large $\omega>0,$ the spectrum of $\Aw$ has only finitely many eigenvalues with positive real part. Hence there exists $n_\omega\in \Nb$ such that 
\begin{align} \label{eq-counteigval}
\Re(\Lambda_n^\pm+\omega )>0 \text{ for all }1\le n \le n_\omega \text{ and } \Re(\Lambda_n^\pm+\omega) <0 \text{ for all }n> n_\omega.
\end{align}
Denote the set of positive elements in $\sigma(\Aw)$ by $
\sigma_+(\Aw)=\{\Lambda_n^\pm+\omega\, |\, 1\le n\le n_\omega\}$
and set of negative elements by $\sigma_-(\Aw)=\sigma(\Aw)\diagdown \sigma_+(\Aw).$ Let $\pi_s$ be the projector on $\sigma_-(\Aw)$ defined by 
 \begin{align*}
 \pi_s=\frac{1}{2\pi i}\int_{\Gamma_s} R(\mu,\Aw) \, d\mu,
 \end{align*}
 where $\Gamma_s$ is a simple Jordan curve around $\sigma_-(\Aw).$ The adjoint operator $\Bf^*\in \Lc(\Hf,\Uf)$ corresponding to $\Bf\in \Lc(\Uf,\Hf)$ is defined by 
\begin{align} \label{eqadjopt-B*}
\Bf^* \begin{pmatrix} \phi \\ \psi \end{pmatrix}=\phi \chi_{\mathcal{O}} \text{ for all }\begin{pmatrix} \phi \\ \psi \end{pmatrix}\in\Hf.
\end{align}

\noindent The next theorem shows that the pair $(\Aw, \Bf),$ equivalently, \eqref{eqn: main shifted_PCE} is open loop stabilizable in $\Hf.$

\begin{Theorem}[open loop stabilizability of $(\Aw,\Bf)$] \label{th:stbopen}
Let $\omega>0$ be arbitrary and $\n (\Aw, D(\Aw))$ be as defined in \eqref{Aw}. Let $\Bf$ be as defined in \eqref{eqcontrol_PCE}. Then $\n (\Aw,\Bf)$ is open loop stabilizable in $\n\Hf.$
\end{Theorem}
\begin{proof}
For a given $\omega>0$, 
\begin{itemize}
    \item[(a)] utilizing Theorem \ref{thm:ua-sem}, $(\Aw , D(\Aw))$ generates an analytic semigroup $\{e^{ t\Aw }\}_{t\ge 0}$ on $\Hf$ with the control operator $\Bf\in \Lc(\Uf,\Hf),$
    \item[(b)] $\Aw$ has only finitely many eigenvalues with non-negative real part, as stated in \eqref{eq-counteigval},
    \item[(c)] there exist $M>0,\, \epsilon>0$ such that
\begin{align*}
\sup_{\Lambda\in \sigma_-(\Aw)} \Re(\Lambda)<-\epsilon
 \text{ and } \|e^{t\Aw}\pi_s\|_{\Lc(\Hf)} \le M e^{-\epsilon t} \text{ for all } t>0.
 \end{align*}
\end{itemize}
 
\noindent To show the stabilizability of $(\Aw, \Bf),$ it is enough to show the Hautus condition \cite[Proposition 3.3, Ch. 1, Part-V]{BDDM}
 \begin{align}\label{eq:Hautus:PCE}
 \text{Ker }(\Lambda \If-\Aw^*)\cap \text{Ker }(\Bf^*)=\{0\} \text{ for all }\Lambda \in \sigma(\Aw^*) \text{ with } \Re(\Lambda)\ge 0.
 \end{align}
 For $\xi \in \text{Ker }(\Lambda \If-\Aw^*)\cap \text{Ker }(\Bf^*),$ $\Aw^*\xi=\Lambda \xi$ implies that $\xi$ is an eigenfunction of $\Aw^*$ corresponding to the eigenvalue $\Lambda.$ Therefore, $\xi$ is of the form $\xi=C_k\xi_k^{+*}$ or $\xi=C_k\xi_k^{-*}$ for some $k\in \{ 1,2,\cdots, n_\omega\},$ where $C_k$ is any scalar constant and $\xi_k^{+*}$ and $\xi_k^{-*}$ are eigenfunctions of $\Aw^*$ for eigenvalues $\overline{\Lambda_k^+}+\omega$ and $\overline{\Lambda_k^-}+\omega,$ respectively.  Also, $\xi\in \text{Ker }\Bf^*$. Thus, \eqref{eqEigFunAf*} and \eqref{eqEigFunAf*-m} imply $C_k\phi_k\chi_{\mathcal{O}}=0$ for all $k\in \{1,\cdots, n_\omega\},$ where $\phi_k$ is an eigenfunction of $-\Delta$ for eigenvalue $\lambda_k$ in $\Lt.$ Since, $\phi_k$ is an analytic function in $\Omega,$ an open connected domain in $\mathbb{R}^d,$ $\phi_k$ cannot vanish in $\mathcal{O}.$  Therefore, $C_k=0$ for all $k\in \{1,\cdots, n_\omega\},$ and hence $\xi=0.$  Thus the Hautus condition holds and $(\Aw, \Bf)$ is open loop stabilizable in $\Hf.$
\end{proof}

 \noindent {\it Proof of Theorem \ref{th:stb cnt}.} Since Theorem \ref{th:stbopen} holds, the existence of a solution of Riccati equation \eqref{eqn:ARE} in $(a)$
 and the result in $(b)$ follow from \cite[Proposition 2.3, Theorem 3.1, Part-V, Ch-1]{BDDM}. 
 Next, since $-\Bf\Bf^*\Pf\in \Lc(\Hf)$ and $(\Aw, D(\Aw))$ generates an analytic semigroup in $\Hf$, \cite[Theorem 12.37]{RROG} gives that $(\Aw{_{,\Pf}}, D(\Aw{_{,\Pf}}))$ also generates an analytic semigroup on $\Hf$. Further, the exponential stability of $\Aw{_{,\Pf}}$ follows from \cite[Remark 3.1, Part-V, Ch-1]{BDDM} and hence the results in $(c)$ is proved. Finally, \cite[Corollary 4.2, Part-V, Ch-1]{BDDM} gives the uniqueness of the solution of \eqref{eqn:ARE} and thus we conclude the proof of the theorem. \qed

% \noindent {\it Proof of Theorem \ref{th:stb cnt}.} Since Theorem \ref{th:stbopen} holds, the result follows from \cite[Theorem 3.1, Remark 3.1 and Corollary 4.2, Part-V, Ch-1]{BDDM}. \qed

\medskip
\noindent Since $(\Aw{_{,\Pf}}, D(\Aw{_{,\Pf}}))$ generates an analytic and exponentially stable semigroup on $\Hf$ as obtained in Theorem \ref{th:stb cnt}, \cite[Theorems 12.31]{RROG} leads to the next proposition.
% \noindent Since $-\Bf\Bf^*\Pf\in \Lc(\Hf)$ and $(\Aw, D(\Aw))$ generates an analytic semigroup in $\Hf$, $(\Aw{_{,\Pf}}, D(\Aw{_{,\Pf}}))$ also generates an analytic semigroup on $\Hf$ \cite[Theorem 12.37]{RROG}. This, the exponential stabilizability of  $(\Aw,\Bf)$ as obtained in Theorem \ref{th:stb cnt}, and \cite[Theorems 12.31]{RROG} lead to the next proposition.

\begin{Proposition}\label{propcontstab}
	% The operator $\n\Aw{_{,\Pf}}:=\Aw-\Bf\Bf^*\Pf$ with $\n D(\Aw{_{,\Pf}})=D(\Aw)$ generates an analytic semigroup $\n\lbrace e^{t\Aw{_{,\Pf}}}\rbrace_{t\geq 0}$ on $\n\Hf$. Furthermore, 
The operator $\n\Aw{_{,\Pf}}$ has the spectrum in $\n\Sigma(-\gamma;\theta_\Pf)=\{-\gamma+re^{i\theta}\,|\, r> 0,\, |\theta|\ge \theta_\Pf\}$ for some $\n\theta_\Pf\in (\frac{\pi}{2},\pi)$ (see Figure \ref{fig:spec-3}).
\end{Proposition}

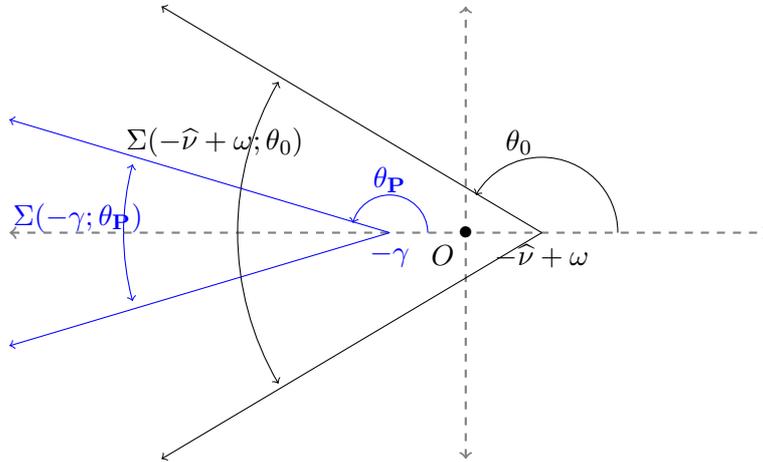
\begin{figure} [ht!]
	\begin{center}
		\begin{tikzpicture}
			\draw[gray,dashed,thick,<->] (-7,0)--(3,0);
			\draw[gray,dashed,thick,<->] (-1,-3)--(-1,3);
			\draw[->] (0,0)--(-5,3);
			\draw[->] (0,0)--(-5,-3);
			\draw[->] (1,0) arc (0:150:1cm);
			\node[] at (-0.3,1.2) {$\theta_0$};
			\node[] at (0,-0.3){$-\nuh+\omega$};
			\draw[->] (-4,0) arc (180:150:4cm);
			\draw[->] (-4,0) arc (180:210:4cm);
			\node[] at (-4.3,1.2){$\Sigma(-\nuh+\omega;\theta_0)$};
			\node[] at (-1,0){$\bullet$};
			\node[] at (-1.3,-0.3){$O$};
					
			\draw[blue,->] (-2,0)--(-7,1.5);
			\draw[blue,->] (-2,0)--(-7,-1.5);
			\draw[blue,->] (-1.5,0) arc (0:165:0.5cm);
			\node[blue] at (-2,0.7) {$\theta_\Pf$};
			\node[blue] at (-2,-0.3){$-\gamma$};
			\draw[blue,->] (-5.5,0) arc (180:165:3.5cm);
			\draw[blue,->] (-5.5,0) arc (180:195:3.5cm);
			\node[blue] at (-6.1,0.2){$\Sigma(-\gamma;\theta_\Pf)$};
		\end{tikzpicture}
	\end{center}
	\caption{$\Sigma(-\nuh+\omega;\theta_0)$ and $\color{blue}{\Sigma(-\gamma;\theta_\Pf)}$ in $\mathbb{C}$} \label{fig:spec-3}
\end{figure}

		%--------------------------------------------------------
		%~~~~~~~~~~~~~~~~~~~~~~~~~~~~~~~~~~~~~~~~~~~~~~~~~~~~~~~~
\section{Approximation of continuous dynamics}\label{sec:approx}
\noindent A framework to study the approximation of the stabilization problem \eqref{eqn: main shifted_PCE} and error estimate are presented. 

\medskip
\noindent Let $\T_h$ be a shape regular quasi-uniform triangulation of $\overline{\Omega}$ \cite{Thomee} into closed triangles with discretization parameter $\displaystyle h:=\max_{ T\in \T_h}\text{diam}(T)$. Let $V_h=\{v_h\in C^0(\overline{\Omega})\,: v_h|_T\in P_1(T)\, \text{ for all } T\in \T_h,\, v_h|_\Gamma=0\}\subset H^1_0(\Omega)$ be a finite dimensional subspace of $\Lt$, with complex field and $\Hf_h:=V_{h}\times V_{h}$ be a finite dimensional subspace of $\Hf$, with complex field, with the inner-product $\langle \cdot,\cdot\rangle$ and norm $\|\cdot\|$ as per notations in Section \ref{subsec:notn}.

\subsection{Projection operators and their approximation properties} \label{sec:lapl prop}
		
Let $\{\phi_j\}_{j=1}^{n_h}$ denote the canonical nodal basis functions for $V_h$, formed by pyramid functions that take value 1 at the interior vertices $P_j$ of triangulation $\T_h$ and vanishes at the boundary. Here $n_h$ denotes the cardinality of the interior nodes of $\T_h$. A given smooth function $v$ on $\Omega$ that vanishes on $\partial \Omega$ may be approximated by $\mathcal{I}_hv(x)=\sum_{j=1}^{n_h}v(P_j)\phi_j(x),$
%\begin{align} \label{I_h}
%\mathcal{I}_hv(x)=\sum_{j=1}^{N_h}v(P_j)\phi_j(x),
%\end{align}
where $P_j$'s denote the interior vertices of $\T_h$. 
For all $v\in\Lt$, let $\pi_h: L^2(\Omega)\longrightarrow V_h$ be the orthogonal projection defined by
\begin{align} \label{eqn:def of pi_h}
    \langle \pi_h v,\phi_h\rangle & =\langle v,\phi_h\rangle\; \text{ for all } \phi_h\in V_h.
\end{align}
		
\medskip
\noindent For any $v\in\Lt$, the definition of $\pi_h$ in \eqref{eqn:def of pi_h} yields
\begin{align} \label{def:proj}
	\|\pi_h v-v\|=\inf_{\phi_h\in V_h}\|v-\phi_h\|.
\end{align}
\noindent
Next we define the discrete operator corresponding to $(\Delta, H^2(\Omega)\cap H_0^1(\Omega))$ on $L^2(\Omega)$. 

\begin{Definition}[discrete Laplace operator {\n \cite{Thomee}}] \label{def:Delta_h}
For each $h>0$, the discrete operator $\Delta_h$ on $V_h$ corresponding to $(\Delta, H^2(\Omega)\cap H_0^1(\Omega))$ on $L^2(\Omega)$  is defined by $\langle \Delta_h u_h,v_h\rangle =-\langle \nabla u_h, \nabla v_h\rangle \text{ for all } u_h, v_h\in V_h.$
\end{Definition}

\begin{Lemma}[interpolation estimates {\n\cite{Thomee}}]\label{lem:intpolerr}
    For any $v\in H^2(\Omega)\cap H_0^1(\Omega)$, it holds that $
    (a)\;  \|\mathcal{I}_hv-v\|\le C h^2\|v\|_{H^2(\Omega)} \text{ and } (b)\; \|\nabla(\mathcal{I}_hv-v)\|\le C h\|v\|_{H^2(\Omega)} $ for some $C>0$ independent of $h.$
\end{Lemma}
	
\begin{Lemma}[inverse inequality {\n\cite{Crlt}}]\label{lem:inv-ineq}
	For any $v_h\in V_h$, it holds that $\|\nabla v_h\|\le C h^{-1}\|v_h\|$ for some positive $C$ independent of $h.$
	%\begin{align} \label{eqn:inv-ineq}
	%\|\nabla v_h\|_{\Lt}\lesssim h^{-1}\|v_h\|_{\Lt},
	%\end{align}
	%where the constant absorbed in $'\lesssim'$ independent of $h$.
\end{Lemma}

\begin{Lemma}[properties of orthogonal projection] \label{lem:projerr}
	Let $\pi_h$ be the orthogonal projection from $\Lt$  onto $V_h$ as defined in \eqref{eqn:def of pi_h}. Then for some $C>0$ independent of $h,$ the estimates below hold:
	\begin{itemize}
	 \item[(a)] $ \|\pi_h v\|\le \|v\| \text{ for all } v\in \Lt, \,    \, \pi_h^2=\pi_h \text{ and } \, \pi_h (I-\pi_h)=(I-\pi_h)\pi_h=0, $
	\item[(b)] for $v\in H^2(\Omega)\cap H^1_0(\Omega)$, $\|\pi_hv-v\|\le C h^2\|v\|_{H^2(\Omega)} \text{ and }
	\|\nabla(v-\pi_h v)\|\le C h \|v\|_{H^2(\Omega)},$
	\item[(c)] for $v\in\Lt$, $\displaystyle \lim_{h\rightarrow 0}\|\pi_h v-v\|_{\Lt}\longrightarrow 0.$
	\end{itemize}
\end{Lemma}

\begin{proof}
(a) Since $\pi_h$ is an orthogonal projection, the estimate follows using \eqref{eqn:def of pi_h}, \eqref{def:proj} and \cite[Corollary 7.1.3, Remark 7.1.8, and Example 7.2.1]{KesFun}.

\medskip
\noindent (b) As $\pi_h$ is an orthogonal projection of $\Lt$ on $V_h$ and $\mathcal{I}_h v\in V_h,$  for any $v\in H^2(\Omega),$ Lemma \ref{lem:intpolerr}(a)  yields 
\begin{align}\label{eq:intest1}
  \|\pi_h v-v\|  =\inf_{\phi_h\in V_h}\|v-\phi_h\| \le \|v-\mathcal{I}_hv\|\le C h^2 \|v\|_{H^2(\Omega)}.
\end{align}
Add and subtract $\nabla(\mathcal{I}_hv)$, utilize Lemma \ref{lem:intpolerr}(b) and  Lemma \ref{lem:inv-ineq}  to obtain
\begin{align*}
	\|\nabla(v-\pi_h v)\|  \le \|\nabla(\mathcal{I}_hv-v)\| +\|\nabla(\pi_h v-\mathcal{I}_hv)\| \le C h\|v\|_{H^2(\Omega)}+h^{-1}\|\pi_h v-\mathcal{I}_hv\|.
\end{align*}
A triangle inequality followed by Lemma \ref{lem:intpolerr}(a) and \eqref{eq:intest1} leads to 
$$ \|\mathcal{I}_hv-\pi_h v\|\le \|\mathcal{I}_hv- v\|+\|v-\pi_h v\|\le C h^2\|v\|_{H^2(\Omega)}.$$
A combination of the last two inequalities concludes the proof.
		
\medskip
\noindent (c) Since $H^2(\Omega)\cap H_0^1(\Omega)$ is dense in $\Lt$, for any $v\in\Lt$ and for any given $\epsilon>0$, there exists $w\in H^2(\Omega)\cap H_0^1(\Omega)$ such that
\begin{equation} \label{eqn:denseee}
	\|w-v\|<\epsilon.
\end{equation} 
Note that $\pi_h v-v=\pi_h w-w+(\pi_h(v-w)-(v-w))$. A triangle inequality with Lemma \ref{lem:projerr}(a)-(b), \eqref{def:proj}, and \eqref{eqn:denseee} concludes the proof.
\end{proof}

\noindent Using $\pi_h$  defined in \eqref{eqn:def of pi_h}, for each $h>0,$ let the  projection operator 
\begin{align}\label{eqdef:Pi_h}
\Pi_h:\Hf\longrightarrow \Hf \text{ with Range }(\Pi_h)=\Hf_h \text{ be defined by } \Pi_h:=\begin{pmatrix}
			\pi_h & 0\\
			0 & \pi_h\end{pmatrix}.
\end{align}
Utilizing Lemma \ref{lem:projerr}(a)-(b), for each $h>0,$  $\Pi_h$ satisfies the properties stated in the lemma below.
\begin{Lemma}[properties of $\Pi_h$]\label{lem:prpty Pi_h}
Let for each $h>0,$ $\Pi_h$ be as defined in \eqref{eqdef:Pi_h}. Then $\Pi_h$ satisfies
\begin{itemize}
    \item[(a)] $\Pi_h^2=\Pi_h$ and $\Pi_h(\If-\Pi_h)=0=(\If-\Pi_h)\Pi_h ,$
    \item[(b)] $\Pi_h$ is self adjoint, that is, $\Pi_h^*=\Pi_h,$
    \item[(c)] $\|\Pi_h \xi\|\le \|\xi\| $ and $\displaystyle \lim_{h\rightarrow 0}\|\xi-\Pi_h \xi\|_{\Hf}\longrightarrow 0$ for all $\xi\in \Hf,$
    \item[(d)] $\|\If-\Pi_h\|_{\Lc(D(\Af),\Hf)}\le Ch^2$ for some $C>0$ independent of $h.$
\end{itemize}
\end{Lemma}

\subsection{Approximation operators and their properties}\label{sec:Appopp}
		
\noindent  This subsection is devoted to construct the approximation operators on finite dimensional space $\Hf_h$ corresponding to the operators $\Af$ and $\Bf.$ Also, it is established that the approximated operator $\Af_h$ has similar properties to $\Af$ in the context of spectral analysis and analytic semigroup.

\medskip
\noindent The discrete operator $\Af_h:\Hf_h\longrightarrow\Hf_h$ that corresponds to $\Af$ is defined by 

\begin{equation}
\begin{aligned} \label{eqn:app A}
\left\langle - \Af_h \begin{pmatrix}y_h \\ z_h\end{pmatrix}, \begin{pmatrix}\phi_h \\ \psi_h\end{pmatrix}\right \rangle:= a\left(\begin{pmatrix}y_h \\ z_h\end{pmatrix}, \begin{pmatrix}\phi_h \\ \psi_h\end{pmatrix}\right) & =\eta_0 \langle \nabla y_h, \nabla \phi_h\rangle +\eta_1 \langle z_h, \phi_h\rangle +\nu_0\langle y_h, \phi_h\rangle +\beta_0\langle\nabla z_h ,\nabla \psi_h\rangle\\
    & \quad +(\kappa+\nu_0)\langle z_h, \psi_h\rangle - \langle y_h, \psi_h\rangle  \text{ for all } \begin{pmatrix}y_h \\ z_h\end{pmatrix}, \begin{pmatrix}\phi_h \\ \psi_h\end{pmatrix} \in \Hf_h,
\end{aligned}
\end{equation}
where the sesquilinear form $a(\cdot, \cdot)$ is introduced in \eqref{bilinear_a}. The adjoint operator $\Af_h^*:\Hf_h\longrightarrow\Hf_h$ is defined by
    $$\displaystyle\left\langle\Af_h^* \begin{pmatrix}
	\varphi_h \\ \varkappa_h
	\end{pmatrix},\begin{pmatrix} \phi_h \\ \psi_h \end{pmatrix}\right\rangle  =\left\langle \begin{pmatrix}
	\varphi_h \\ \varkappa_h \end{pmatrix},\Af_h  \begin{pmatrix} \phi_h \\ \psi_h
	\end{pmatrix} \right\rangle \text{ for all }  \begin{pmatrix}	\varphi_h \\ \varkappa_h
	\end{pmatrix},\; \begin{pmatrix} \phi_h \\ \psi_h
	\end{pmatrix}\in \Hf_h .$$

\noindent From the definition, it is clear that for all $h>0$, $\Af_h$ generates an analytic semigroup on $\Hf_h$. However, for our analysis, it is needed that for all $h>0$, $\Af_h$ generates a uniformly (with respect to $h$) analytic semigroup $\{e^{t \Af_h}\}_{t\ge 0}$ on $\Hf_h$, that is,  all constants and parameters in Definition \ref{def:ana-sgp} are independent of $h$, for all $h>0$. We show it in the following theorem. 
% In particular, we show that for all $h>0$, 
% $$ \sigma(\Af_h)\subset \Sigma(-\nuh; \theta_0), \quad \n\|R(\mu, \Af_h)\|_{\mathcal{L}(\Hf_h)}\leq \frac{C}{|\mu+\nuh|} \text{ for all } \mu \in \Sigma^c(-\nuh; \theta_0),\,  \mu\neq -\nuh,$$ 
% for some positive constant $C$, independent of $h$, where $\n\Sigma(-\nuh; \theta_0)$ is given in Theorem \ref{thm:ua-sem}.

\begin{Theorem}[uniform analyticity and resolvent estimate]\label{th-unianaly-approx}
Let the finite dimensional operator $\n\Af_h$ on $\n\Hf_h$ be as defined in \eqref{eqn:app A}. Then for all $h>0$, the results below hold.
    \begin{itemize}
		\item[(a)] The sector $\Sigma(-\nuh;\theta_0)$ as in $(a)$ of Theorem \ref{thm:ua-sem} contains the spectrum $\sigma(\Af_h)$, and for all $\mu\in \Sigma^c(-\nuh; \theta_0),$ the resolvent operator $\n R(\mu,\Af_h):= (\mu \If_h-\Af_h)^{-1}$ satisfies 
			\begin{equation}\label{resolest-approx}
			\n \|R(\mu,\Af_h)\|_{\Lc(\Hf_h)} \le \frac{C}{|\mu+\nuh|}, \quad \mu \neq -\nuh,
			\end{equation}
		for some $C>0$ independent of $\mu$ and $h$. For $\mu=-\nuh,$ the uniform bound below holds
		\begin{align} \label{eq-uniboundresol}
		    \|\n R(-\nuh,\Af_h)\|_{\Lc(\Hf)}\le C,
		\end{align}
		for some $C>0$ independent of $h$. 

        \item[(b)] The operator $\n e^{t\Af_h}\in \Lc(\Hf_h)$ can be represented by 
			\begin{align} \label{eqrepsemiaopprox}
			\n e^{t\Af_h}=\frac{1}{2\pi i}\int_\Gamma e^{\mu t} R(\mu,\Af_h)\, d\mu , \, \text{ for all }t>0,
			\end{align}
		where $\Gamma$ is any curve from $-\infty$ to $\infty$ and is entirely in $\n\Sigma^c(-\nuh;\theta_0).$
		
		\item[(c)] The operator $\n\Af_h$ generates a uniformly (in $h$) analytic semigroup $\normalfont \{e^{t\Af_h}\}_{t\ge 0}$ on $\n\Hf_h$ satisfying 
			\begin{equation*} 
			\n \|e^{t \Af_h}\|_{\mathcal{L}(\Hf_h)}\le Ce^{-\nuh t} \text{ for all } t>0,
			\end{equation*}
			for some $C>0$ independent of $h$.
	\end{itemize}
\end{Theorem}

\begin{proof}
Let $\begin{pmatrix}f_h \\ g_h\end{pmatrix}\in\Hf_h$ be arbitrary. Then our first aim is to find a unique $\begin{pmatrix}y_h \\ z_h\end{pmatrix}\in \Hf_h$ such that for all  $\begin{pmatrix}\phi_h \\ \psi_h\end{pmatrix} \in \Hf_h,$
$$\left\langle (\mu \If_h- \Af_h)\begin{pmatrix}y_h \\ z_h\end{pmatrix}, \begin{pmatrix}\phi_h \\ \psi_h\end{pmatrix} \right\rangle=a\left( \begin{pmatrix}y_h \\ z_h\end{pmatrix}, \begin{pmatrix}\phi_h \\ \psi_h\end{pmatrix}\right)+\mu \left\langle \begin{pmatrix}y_h \\ z_h\end{pmatrix}, \begin{pmatrix}\phi_h \\ \psi_h\end{pmatrix} \right\rangle=\left\langle\begin{pmatrix}f_h \\ g_h\end{pmatrix}, \begin{pmatrix}\phi_h \\ \psi_h\end{pmatrix}\right\rangle,$$   
where $a(\cdot,\cdot)$ is defined in \eqref{bilinear_a}. For all $h>0,$ we have the coercivity and boundedness of $a(\cdot,\cdot)+\mu\langle \cdot, \cdot \rangle$ with constants $\alpha_0$ and $\alpha_1$ (both independent of $h$), respectively (see \eqref{bilinear_acor} and \eqref{bilinear_a_bdd}). Therefore, as in Theorem \ref{thm:ua-sem}(a), for all $h>0,$  $\sigma(\Af_h)\subset \Sigma(-\nuh;\theta_0)$ and there exists a unique $\begin{pmatrix}y_h \\ z_h\end{pmatrix}\in \Hf_h$ such that the last displayed equality holds. Now, proceed as in the proof of Theorem \ref{thm:ua-sem}(a) to obtain
\begin{align*}
\left\| R(\mu,\Af_h)\begin{pmatrix} f_h \\ g_h\end{pmatrix}\right\|=(\|y_h\|^2+\|z_h\|^2)^{1/2} \le \frac{C}{|\mu+\nuh|}\left( \|f_h\|^2+\|g_h\|^2\right)^{1/2},
\end{align*} 
for some  $C=C(\alpha_1,\alpha_0, C_p)$ independent of $\mu$ and $h.$ An analogous argument to establish \eqref{eqValreltn-yzfg} leads to \eqref{eq-uniboundresol}.

\medskip
\noindent (b) Since, for all $h>0,$ $\sigma(\Af_h)$ is contained in the uniform (in $h$) sector $\Sigma(-\nuh;\theta_0)$ and the constants appearing in \eqref{resolest-approx} are independent of $h,$ a similar argument as in Theorem \ref{thm:ua-sem}(b) concludes that for all $h>0$, $\n\Af_h$ generates a uniformly (in $h$) analytic semigroup $\normalfont \{e^{t\Af_h}\}_{t\ge 0}$ on $\n\Hf_h$ with the representation \eqref{eqrepsemiaopprox}. 

\medskip
\noindent (c) The proof is analogous to the proof of Theorem \ref{thm:ua-sem}(c).
\end{proof}
			
\begin{figure} [ht!]
	\begin{center}
		\begin{tikzpicture}
			\draw[gray,dashed,thick,<->] (-7,0)--(3,0);
			\draw[gray,dashed,thick,<->] (0,-3.5)--(0,3.5);
			\draw[->] (0,0)--(-3,3);
			\draw[->] (0,0)--(-3,-3);
			\draw[->] (0,1) arc (55:150:0.5cm);
			\node[] at (-0.6,1.3) {$\delta_0$};
			\node[] at (0,-0.3){O};
			%\draw[->] (-4,0) arc (180:150:4cm);
	    	%\draw[->] (-4,0) arc (180:210:4cm);
					%\node[] at (-4.3,1.2){$\Sigma(A)$};
			\node[] at (-1,0){$\bullet$};
			\draw[] (0,0)--(0,1.5);
			\node[blue] at (-2,1.5){$\bullet$};
			\draw[red,dashed](-2,1.5) circle(0.3cm);
			\node[blue] at (-2,-1.5){$\bullet$};
			\draw[red,dashed](-2,-1.5) circle(0.3cm);
					
			\node[blue] at (-1.5,0.5){$\bullet$};
			\draw[red,dashed](-1.5,0.5) circle(0.3cm);
			\node[blue] at (-1.5,-0.5){$\bullet$};
			\draw[red,dashed](-1.5,-0.5) circle(0.3cm);
					
			\node[blue] at (-4,0.8){$\bullet$};
			\draw[red,dashed](-4,0.8) circle(0.3cm);
			\node[blue] at (-4,-0.8){$\bullet$};
			\draw[red,dashed](-4,-0.8) circle(0.3cm);
					
			\node[blue] at (-3.8,1.1){$\bullet$};
			\draw[red,dashed](-3.8,1.1) circle(0.3cm);
			\node[blue] at (-3.8,-1.1){$\bullet$};
			\draw[red,dashed](-3.8,-1.1) circle(0.3cm);
					
			\node[blue] at (-3,2.5){$\bullet$};
			\draw[red,dashed](-3,2.5) circle(0.3cm);
			\node[blue] at (-3,-2.5){$\bullet$};
			\draw[red,dashed](-3,-2.5) circle(0.3cm);
			
			\node[blue] at (-1.8,0){$\bullet$};
			\draw[red,dashed](-1.8,0) circle(0.3cm);
			
			\node[blue] at (-2.8,0){$\bullet$};
			\draw[red,dashed](-2.8,0) circle(0.3cm);
			
			\node[blue] at (-4.5,0){$\bullet$};
			\draw[red,dashed](-4.5,0) circle(0.3cm);
			
			\node[blue] at (-4.7,0){$\bullet$};
			\draw[red,dashed](-4.7,0) circle(0.3cm);
			
			\node[blue] at (-5.5,0){$\bullet$};
			\draw[red,dashed](-5.5,0) circle(0.3cm);
					
			\end{tikzpicture}
		\end{center}
		\caption{Finitely many complex eigenvalues of $\Af$ plotted with blue dots} \label{fig:ua}
\end{figure}
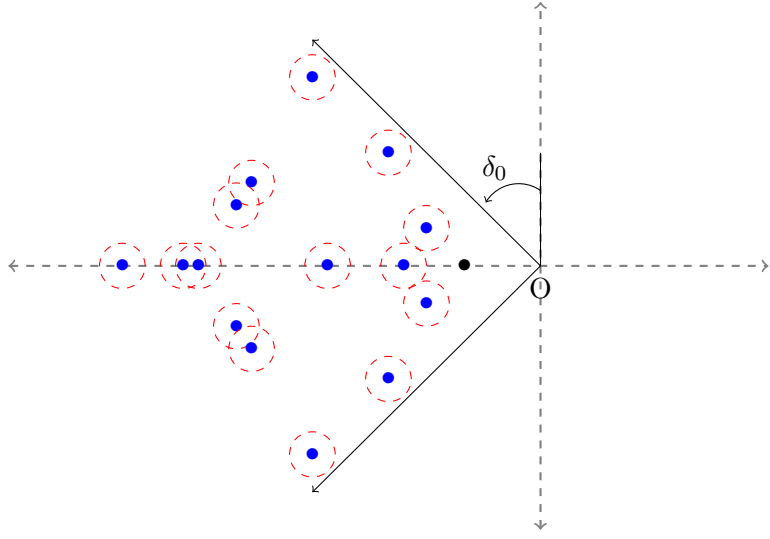
\noindent In the next lemma, it is established that the eigenvalues of $\n\Af_h$ converge to the corresponding eigenvalues of $\n\Af$ with quadratic rate of convergence and this is verified numerically in Section \ref{sec:NI} (see Table \ref{tab:roc_eig_2}).		
		
\begin{Lemma}[convergence of eigenvalues]\label{lem:eig-conv}
Let $\{\Lambda_k^\pm\}_{k\in \Nb}$ be the eigenvalues of $\Af$ as in Proposition \ref{pps:spec Af}. Let $\{\Lambda_{k,h}^\pm\}_{k=1}^{n_h}$ be the corresponding eigenvalues of $\Af_h,$ where $2n_h$ denotes the dimension of $\Hf_h.$ Then for all $k=1,\ldots,n_h$ and for all $h>0,$ there exist positive constants $C(k)$ such that 
\begin{align*}
    |\Lambda_{k,h}^\pm-\Lambda_k^\pm|\le C(k) \left(\beta_0+\eta_0+|\beta_0-\eta_0| \right) h^2.
\end{align*}
\end{Lemma}
		
\begin{proof}
Let us first recall the convergence of eigenvalues of discrete Lapalacian to those of the Laplacian operator $-\Delta_h.$ For any $h>0$, let $\mathrm{dim}\,(V_h)= n_h\in \Nb$ and let $\{\lambda_{k,h}\mid k=1, \cdots, n_h\}$ be the set of eigenvalues of $-\Delta_h$ defined on $V_h$. It is well-known that for all $h>0,$ and for all $k=1, \ldots, n_h$, $\lambda_k\leq \lambda_{k,h},$ 
and $\lambda_{k,h}\le \lambda_k+ C (k) h^2$, for some $C(k)>0$ independent of $h$ \cite{Bof}. Combining the above results, for each $k=1, \ldots, n_h$, $|\lambda_{k,h}- \lambda_k|\le C(k)h^2\rightarrow 0$ as $h\downarrow 0$. 

\noindent Note that for $h>0$, the dimension of the Hilbert space, $\Hf_h= 2n_h$ and the discrete linear operator $\Af{_h}$ defined on $\Hf_h$ is of order $2n_h$. It can be checked that the eigenvalues of $\Af{_h}$ are 
\begin{equation}\label{eqapproxev}
\Lambda_{k,h}^\pm= -\frac{1}{2}\left((\eta_0+\beta_0)\lambda_{k,h}+\kappa+2\nu_0 \right)\pm \frac{1}{2} \sqrt{\left( (\beta_0-\eta_0)\lambda_{k,h}+\kappa  \right)^2-4\eta_1} \text{ for all } k=1, \cdots, n_h,
\end{equation}
(see Figure~\ref{fig:ua}). Using the expression of $\Lambda_k^\pm$ given in Proposition \ref{pps:spec Af}, \eqref{eqapproxev} and $|\lambda_{k,h}- \lambda_k|\le C(k)h^2$, we deduce that $ | \Lambda_{k,h}^\pm - \Lambda_k^\pm |  \leq C(k) \left(\beta_0+\eta_0+|\beta_0-\eta_0| \right) h^2$ for  any positive integer $k=1,\ldots,n_h.$ This concludes the proof.
\end{proof}

\noindent Since, for all $h>0,$ $\Af_h$ generates uniformly analytic semigroup $\{e^{t\Af_h}\}_{t\ge 0}$ of negative type, the well-posedness of \eqref{eq:appY'AYBu} follows. The proof is standard, for example, see \cite[Prop. 3.1, Ch-1, Part-II]{BDDM}.

\begin{Lemma}\label{cns:dis-1}
For any $\n\Yf_{0_h}\in \Hf_h$ and any $\n F_h\in L^2(0,\infty; \Hf_h)$,  system  $\n\Yf_h'(t)=\Af_h\Yf_h(t)+F_h(t),\;t>0, \;  \Yf_h(0)=\Yf_{0_h}$
admits  a unique solution  $\n\Yf_h(\cdot)\in C([0,\infty); \Hf_h)$ with the representation $\n\Yf_h(t)=e^{t\Af_h}\Yf_{0_h}+\int_0^te^{(t-s)\Af_h}F_h(s) \,ds \text{ for all } t>0.$
\end{Lemma}

\subsection{Convergence results for approximation operators}\label{sec:convres-unctrl}
In this subsection, an error estimate for the system without control is established. This result is essential to establish the convergence results for the stabilized system.

\medskip
\noindent Consider the continuous dynamics 
\begin{align}\label{eq:unctrl}
\Yf'(t)=\Af\Yf(t)\text{ for all }t>0,\,\Yf(0)=\Yf_0 \in \Hf,
\end{align}
and its approximation dynamics 
\begin{align}\label{eq:d-unctrl}
\Yf_h'(t)=\Af_h\Yf_h(t) \text{ for all }t>0,\, \Yf_h(0)=\Pi_h\Yf_0. 
\end{align}
Note that the solution of \eqref{eq:unctrl} and \eqref{eq:d-unctrl}, respectively, can be represented by 
$$ \Yf(t)=\begin{pmatrix} y(t) \\ z(t)\end{pmatrix}=e^{t\Af}\Yf_0, \quad \Yf_h(t)=\begin{pmatrix} y_h(t)\\z_h(t)\end{pmatrix}=e^{t\Af_h}\Pi_h\Yf_0 \text{ for all } t>0.$$
Now to study the convergence of $\Yf_h(t)$ to $ \Yf(t)$ in a suitable norm, we use the expression of the semigroups \eqref{eqrepsemi} and \eqref{eqrepsemiaopprox} and hence we need suitable estimates of the resolvent operators that are established in the next lemma.

\begin{Lemma}[error in resolvent]\label{lem:1st compresest}
Let $\n\Af$ and $\n\Af_h$ be as defined in \eqref{eqdef-A_PCE} and \eqref{eqn:app A}, respectively. Let $\n\Sigma^c(-\nuh;\theta_0)$ be as introduced in Theorem \ref{thm:ua-sem}. Then for some $C>0,$ independent of $h,$ the resolvent operators satisfy

\begin{itemize}
\item[(a)] $\n\|R(-\nuh,\Af)-R(-\nuh,\Af_h)\Pi_h\|_{\Lc(\Hf)}\le Ch^2$ for all $h>0$,
\item[(b)] $\n \displaystyle\sup_{\mu\in \Sigma^c(-\nuh;\theta_0)} \|R(\mu,\Af)-R(\mu,\Af_h)\Pi_h\|_{\Lc(\Hf)}\le Ch^2$ for all $h>0$. 
\end{itemize}

 \end{Lemma}

\begin{proof}
(a) For any $\begin{pmatrix}f\\g\end{pmatrix}\in \Hf,$ let $\begin{pmatrix}u\\ v\end{pmatrix}\in D(\Af)$ and $\begin{pmatrix}u_h\\ v_h\end{pmatrix}\in \Hf_h$ be such that $R(-\nuh,\Af)\begin{pmatrix}f\\g\end{pmatrix}=\begin{pmatrix}u\\ v\end{pmatrix}$ and $R(-\nuh,\Af_h)\Pi_h\begin{pmatrix}f\\g\end{pmatrix}=\begin{pmatrix}u_h\\v_h\end{pmatrix}.$ That is,
\begin{align}\label{eqn:cont-wk}
a\left( \begin{pmatrix}u\\ v\end{pmatrix}, \begin{pmatrix} \phi \\ \psi\end{pmatrix} \right)-\nuh \left\langle \begin{pmatrix}u\\ v\end{pmatrix}, \begin{pmatrix} \phi \\ \psi\end{pmatrix}\right\rangle = \left\langle \begin{pmatrix}f\\g\end{pmatrix}, \begin{pmatrix} \phi \\ \psi\end{pmatrix} \right\rangle \text{ for all } \begin{pmatrix} \phi \\ \psi\end{pmatrix} \in \Vf, \text{ and }
\end{align}
\begin{align}\label{eqn:disc-wk}
 a\left( \begin{pmatrix}u_h\\ v_h\end{pmatrix}, \begin{pmatrix} \phi_h \\ \psi_h\end{pmatrix} \right)-\nuh \left\langle \begin{pmatrix}u_h\\ v_h\end{pmatrix}, \begin{pmatrix} \phi_h \\ \psi_h\end{pmatrix}\right\rangle= \left\langle \Pi_h\begin{pmatrix}f\\g\end{pmatrix}, \begin{pmatrix} \phi_h \\ \psi_h\end{pmatrix} \right\rangle \text{ for all } \begin{pmatrix} \phi_h \\ \psi_h\end{pmatrix} \in \Hf_h.
\end{align}
Subtract \eqref{eqn:disc-wk} from \eqref{eqn:cont-wk} and use \eqref{eqn:def of pi_h}, that is, $\left\langle \begin{pmatrix}f-\pi_h f \\ g-\pi_h g \end{pmatrix}, \begin{pmatrix} \phi_h \\ \psi_h\end{pmatrix} \right\rangle =0,$  for all $\begin{pmatrix} \phi_h \\ \psi_h\end{pmatrix} \in \Hf_h\subset \Vf$ to obtain 
\begin{align} \label{eqOrthogoality_a}
a\left( \begin{pmatrix}u-u_h\\ v-v_h\end{pmatrix}, \begin{pmatrix} \phi_h \\ \psi_h\end{pmatrix} \right)-\nuh\left\langle \begin{pmatrix}u-u_h\\ v-v_h\end{pmatrix}, \begin{pmatrix} \phi_h \\ \psi_h\end{pmatrix} \right\rangle  =0.
\end{align}
This implies
$$
\begin{array}{l}
a\left( \begin{pmatrix}u-u_h\\ v-v_h\end{pmatrix}, \begin{pmatrix} u-u_h \\ v-v_h\end{pmatrix} \right)  -\nuh \left\langle \begin{pmatrix}u-u_h\\ v-v_h\end{pmatrix}, \begin{pmatrix} u- u_h\\ v- v_h\end{pmatrix} \right\rangle  = \\
\hspace{3cm} a\left( \begin{pmatrix}u-u_h\\ v-v_h\end{pmatrix}, \begin{pmatrix} u-\pi_h u  \\ v-\pi_h v \end{pmatrix} \right)
-\nuh \left\langle \begin{pmatrix}u-u_h\\ v-v_h\end{pmatrix}, \begin{pmatrix} u - \pi_h u  \\ v - \pi_h v \end{pmatrix} \right\rangle .
\end{array}
$$
The continuity from \eqref{bilinear_a_bdd}, coercivity from \eqref{bilinear_acor} followed by Lemma \ref{lem:projerr}(b) and \eqref{eqRegEsr-yzinH2}, and the last displayed equality lead to 
\begin{equation} \label{eqEnergEstu-uh v-vh}
\begin{aligned}
\left(\|\nabla(u-u_h)\|^2 +\|\nabla(v-v_h)\|^2 \right)^{1/2} & \le \frac{\alpha_1}{\alpha_0} \left(\|\nabla(u-\pi_h u)\|^2 +\|\nabla(v-\pi_h v)\|^2 \right)^{1/2} \\
&\le \frac{\alpha_1}{\alpha_0} h \left( \|u\|_{H^2(\Omega)}^2+\|v\|_{H^2(\Omega)}^2\right)^{1/2}  \le C\frac{\alpha_1}{\alpha_0} h \left( \|f\|^2+\|g\|^2\right)^{1/2}. 
\end{aligned}
\end{equation}
 
\noindent To employ a duality argument, consider a dual problem: for given $\begin{pmatrix}p\\ q\end{pmatrix}\in \Hf,$ seek $\begin{pmatrix} \Phi \\ \Psi\end{pmatrix}\in D(\Af^*)$ such that
\begin{align}\label{eqDualProb}
(-\nuh \If-\Af)^*\begin{pmatrix} \Phi \\ \Psi\end{pmatrix}=\begin{pmatrix}p\\ q\end{pmatrix} \text{ in }\Omega, \quad \begin{pmatrix} \Phi \\ \Psi\end{pmatrix}=0 \text{ on }\partial\Omega.
\end{align}
Then from Lemma \ref{lem:RegResult-Af*}(b), we have the existence of $\begin{pmatrix} \Phi \\ \Psi\end{pmatrix} \in D(\Af^*)$ and with   \eqref{eqOrthogoality_a}, we obtain
\begin{align*}
\left\langle \begin{pmatrix}u- u_h \\ v-v_h\end{pmatrix}, \begin{pmatrix}p\\ q\end{pmatrix}\right\rangle & =\left\langle \begin{pmatrix}u- u_h \\ v-v_h\end{pmatrix}, (-\nuh \If-\Af)^*\begin{pmatrix} \Phi \\ \Psi\end{pmatrix}\right\rangle =a\left(  \begin{pmatrix}u- u_h \\ v-v_h\end{pmatrix}, \begin{pmatrix} \Phi \\ \Psi\end{pmatrix}  \right) - \nuh \left\langle  \begin{pmatrix}u- u_h \\ v-v_h\end{pmatrix}, \begin{pmatrix} \Phi \\ \Psi\end{pmatrix}  \right\rangle\\
& = a\left(  \begin{pmatrix}u- u_h \\ v-v_h\end{pmatrix}, \begin{pmatrix} \Phi-\pi_h \Phi \\ \Psi-\pi_h \Psi\end{pmatrix}  \right) - \nuh \left\langle  \begin{pmatrix}u- u_h \\ v-v_h\end{pmatrix}, \begin{pmatrix} \Phi-\pi_h \Phi \\ \Psi-\pi_h \Psi\end{pmatrix}  \right\rangle.
\end{align*}
Thus, a use of \eqref{bilinear_a_bdd} in above equality followed by \eqref{eqEnergEstu-uh v-vh}, Lemma \ref{lem:projerr}(b), and Lemma \ref{lem:RegResult-Af*}(b) leads to
\begin{align*}
\left\vert \left\langle \begin{pmatrix}u- u_h \\ v-v_h\end{pmatrix}, \begin{pmatrix}p\\ q\end{pmatrix}\right\rangle \right\vert &\le \alpha_1 \left(\|\nabla(u- u_h)\|^2+\|\nabla( v-v_h)\|^2\right)^{1/2} \left( \|\nabla( \Phi- \pi_h\Phi)\|^2+\|\nabla(\Psi-\pi_h\Psi)\|^2\right)^{1/2} \\
%&\le C\frac{\alpha_1^2}{\alpha_0} h \left( \|f\|^2+\|g\|^2\right)^{1/2}\left(\|\nabla(\Phi-\pi_h\Phi)\|^2+\|\nabla(\Psi-\pi_h\Psi)\|^2\right)^{1/2}  \\
& \le C\frac{\alpha_1^2}{\alpha_0} h^2 \left( \|f\|^2+\|g\|^2\right)^{1/2}\left(\|\Phi\|_{H^2(\Omega)}^2+\|\Psi\|_{H^2(\Omega)}^2\right)^{1/2}\\
& \le C\frac{\alpha_1^2}{\alpha_0} h^2 \left( \|f\|^2+\|g\|^2\right)^{1/2} \left(\|p\|^2+\|q\|^2\right)^{1/2}.
\end{align*}
Choose $\begin{pmatrix}p\\ q\end{pmatrix}=\begin{pmatrix}u-u_h\\ v-v_h\end{pmatrix}$ in the last displayed inequality to obtain
\begin{align*}
\left(\|u-u_h\|^2 +\|v-v_h\|^2 \right)^{1/2}\le C\frac{\alpha_1^2}{\alpha_0} h^2 \left( \|f\|^2+\|g\|^2\right)^{1/2}, 
\end{align*}
and thus 
\begin{align*}
\left\|\left( R(-\nuh,\Af) - R(-\nuh,\Af_h)\Pi_h\right) \begin{pmatrix}f\\g\end{pmatrix} \right\| & =\left(\|u-u_h\|^2 +\|v-v_h\|^2\right)^{1/2} \le C h^2\left( \|f\|^2+\|g\|^2\right)^{1/2}.
%= Ch^2 \left\|\begin{pmatrix}f\\g\end{pmatrix} \right\|_{\Hf}. 
\end{align*}
This completes the proof of (a).

\medskip
\noindent (b) We first derive some useful identities which will be used to obtain the estimates. 
The definition of the resolvent operator shows that for any $\mu\in \rho(\Af)$, $(\mu \If-\Af)[\If-(\mu+\nuh) R(\mu,\Af)]R(-\nuh,\Af)=\If$ and 
$(\mu \If_h-\Af_h)[\If_h-(\mu+\nuh) R(\mu,\Af_h)]R(-\nuh,\Af_h)=\If_h$, and thus 
\begin{equation*}%\label{new res exp-1}
\begin{aligned}
&R(\mu,\Af)=R(-\nuh,\Af)-(\mu+\nuh) R(\mu,\Af)R(-\nuh,\Af)
\text{ and } \\
&R(\mu,\Af_h)\Pi_h=R(-\nuh,\Af_h)\Pi_h-(\mu+\nuh) R(\mu,\Af_h)R(-\nuh,\Af_h)\Pi_h.
\end{aligned}
\end{equation*}
An addition and subtraction of $(\mu+\nuh) R(\mu,\Af) R(-\nuh,\Af_h)\Pi_h$ after subtracting the two identities above and elementary algebra leads to
% \begin{align*}
% R(\mu,\Af)-R(\mu,\Af_h)\Pi_h & =\left(R(-\nuh,\Af) - R(-\nuh,\Af_h)\Pi_h\right) -(\mu+\nuh) \Big(R(\mu, \Af)-R(\mu,\Af_h)\Pi_h\Big)R(-\nuh,\Af_h)\Pi_h\\
% & \quad -(\mu+\nuh) R(\mu,\Af)\left(R(-\nuh,\Af) - R(-\nuh,\Af_h)\Pi_h\right),
% \end{align*}
% and it yields 
\begin{equation}\label{eq:diff resol}
\begin{aligned}
\big(R(\mu,\Af)- & R(\mu,\Af_h)\Pi_h\big)  \left(\If+(\mu+\nuh)R(-\nuh,\Af_h)\Pi_h\right)\\
&=\left(\If-(\mu+\nuh) R(\mu,\Af)\right)\left(R(-\nuh,\Af)-R(-\nuh,\Af_h)\Pi_h\right). 
\end{aligned}
\end{equation}
Elementary algebra shows
\begin{align*}
    & \left(\If+(\mu+\nuh)R(-\nuh,\Af_h)\Pi_h\right)\left(\If-\Pi_h+(-\nuh\If_h-\Af_h)R(\mu,\Af_h)\Pi_h\right) ={}\\ & \hspace{3cm} =\If-\Pi_h+(\mu\If_h+\nuh\If_h-\nuh\If_h-\Af_h)R(\mu,\Af_h)\Pi_h=\If \text{ and } \\
    & \If-\Pi_h+(-\nuh\If_h-\Af_h)R(\mu,\Af_h)\Pi_h=\If-(\mu+\nuh)R(\mu,\Af_h)\Pi_h.
\end{align*}
The last displayed estimates and \eqref{eq:diff resol} lead to 
\begin{equation}\label{estar}
\begin{aligned}
R(\mu,\Af)-R(\mu,\Af_h)\Pi_h=\left(\If-(\mu+\nuh) R(\mu,\Af)\right)\left(R(-\nuh,\Af) - R(-\nuh,\Af_h)\Pi_h\right)\left(\If-(\mu+\nuh) R(\mu, \Af_h)\Pi_h\right). 
\end{aligned}
\end{equation}
Utilize \eqref{eq:resolestA}, \eqref{resolest-approx} and (a) in \eqref{estar} to obtain
\begin{align*}
\left\|R(\mu,\Af)-R(\mu,\Af_h)\Pi_h\right\|_{\Lc(\Hf)} 
% & \le \left\|\If-(\mu+\nuh) R(\mu,\Af)\right\|_{\Lc(\Hf)}\left\|R(-\nuh,\Af) - R(-\nuh,\Af_h)\Pi_h\right\|_{\Lc(\Hf)} \\
% & \qquad \times \left\|\If-(\mu+\nuh) R(\mu, \Af_h)\Pi_h\right\|_{\Lc(\Hf)} \\
% & 
\le C h^2, \quad \forall\, \mu\in \Sigma^c(-\nuh;\theta_0),
\end{align*}
where the positive constant $C$ is independent of $\mu$ and $h$. 

\noindent  This concludes the proof.
\end{proof}

\ncom{\yk}{\mathfrak{y}}
\ncom{\zk}{\mathfrak{z}}

\begin{Theorem}[error estimate for the system without control] \label{pps-estSGest}
For any $\n\Yf_0\in \Hf,$ the operators $\n\Af$ and $\n\Af_h$ defined in \eqref{eqdef-A_PCE} and \eqref{eqn:app A}, respectively, satisfy 
\begin{itemize}
    \item[(a)]  $\left\|(e^{t\Af}-e^{t\Af_h}\Pi_h)\Yf_0\right\| \le  C h^2 \frac{e^{-\nuh t}}{t}\|\Yf_0\|$ for all $t>0$ and for all $h>0$,
\item[(b)] $\n \left\|(e^{t\Af}-e^{t\Af_h}\Pi_h)\Yf_0\right\|_{L^2(0,\infty;\Hf)} \le C_\theta h^{2\theta}\|\Yf_0\|$ for all $h>0$ and for any $0<\theta<\frac{1}{2},$ 
\end{itemize}
for some positive constants $C$ and $C_\theta$ independent of $h.$  
% Here, $\n \{e^{t\Af}\}_{t\ge 0}$ and $\n \{e^{t\Af_h}\}_{t\ge 0}$  are the semigroups generated by $\n\Af$ and $\n\Af_h$, respectively.
\end{Theorem}

\begin{figure}[ht!]
\begin{center}
\begin{tikzpicture}
\path[fill=black!15] (-1,0)--(-7,3.6)--(-7,-3.6)--cycle;
\draw[gray,dashed,thick,<->] (-7,0)--(3,0);
\draw[gray,dashed,thick,<->] (0,-4)--(0,4);
\draw[thick,->] (-1,0)--(-6,3);
\draw[thick,->] (-1,0)--(-6,-3);
\draw[thick,->] (-0.5,0) arc (0:150:0.5cm);
\node[] at (-0.9,0.6) {$\theta_0$};
\node[] at (0.25,-0.25){$O$};
\draw[thick,->] (-5,0) arc (180:150:4cm);
\draw[thick,->] (-5,0) arc (180:210:4cm);
\node[] at (-4.8,0.5){$\Sigma(-\nuh; \theta_0)$};
\node[] at (-1,0){$\bullet$};
\node[] at (-1,-0.3){$-\nuh$};
%\node[purple] at (0,-1){$r_\delta$};

\draw[thick,red, dashed,->] (-1,0.85)--(-5,4);
%\node[red] at (-1,-1.7){$-\widetilde{\nu}$};
%\node[red] at (-1,1.7){$-\widetilde{\nu}$};
\draw[thick,red, dashed,->] (-5,-4)--(-1,-0.85);
\draw[thick,red,dashed,thick,->] (-1,-0.85) arc (270:360:0.85cm);
\draw[thick,red,dashed] (-1,0.85) arc (90:0:0.85cm);
%\node[red] at (0.5,-1.7){$\Gamma$};
\node[red] at (-3.5,-2.7){$\Gamma_-$};
%\node[red] at (1,-0.3){$\Gamma_0$};
\node[red] at (-3.5,2.7){$\Gamma_+$};
\node[red] at (0,0.4){$\Gamma_0$};
%\node[red] at (-3,-3.7){$\Gamma_1$};
%\node[red] at (1,-0.3){$\Gamma_2$};
%\node[red] at (-3,3.7){$\Gamma_3$};
%\draw[thick,blue,dashed] (-1,-2) arc (270:180:2cm);
%\draw[thick,blue,dashed,->] (-1,2) arc (90:180:2cm);
%\draw[yellow] (-1,-3)--(-1,3);
%\draw[purple,->] (-1,0)--(0.6,-1.2);
\end{tikzpicture}
\end{center}
\caption{$\Sigma(-\nuh; \theta_0)$ and $\color{red}{\Gamma}=\Gamma_+\cup
\Gamma_-\cup\Gamma_0$} \label{fig:spec-x}
\end{figure}

\begin{proof}
\noindent (a) Let $\Gamma$ be a path (refer Figure \ref{fig:spec-x}) in $\Sigma^c(-\nuh; \theta_0)$ such that $\Gamma = \Gamma_\pm\cup\Gamma_0,$ where $\Gamma_\pm=\{-\nuh+re^{\pm i\phi_0}, r\ge r_0 \}$ and $\Gamma_0=\{-\nuh+r_0 e^{i\vartheta},\, |\vartheta|\le \phi_0\}$ for some $\frac{\pi}{2}<\phi_0<\theta_0 <\pi$ and for some $r_0>0$. Theorems \ref{thm:ua-sem}(b) and \ref{th-unianaly-approx}(b) show
\begin{align*}%\label{eqintdiff}
		e^{t\Af}\Yf_0 - e^{t\Af_h}\Pi_h\Yf_0 & =\frac{1}{2\pi i}\int_{\Gamma}e^{\mu t}( R(\mu,\Af) - R(\mu,\Af_h)\Pi_h)\Yf_0\, d\mu \\
		&=\frac{1}{2\pi i}\int_{t\Gamma^1}e^{\mu_1-\nuh t}\left( R\left(\frac{\mu_1}{t}-\nuh,\Af\right) - R\left(\frac{\mu_1}{t}-\nuh,\Af_h\right)\Pi_h \right)\Yf_0\, \frac{d\mu_1}{t},
\end{align*}
where the last equality is obtained using the change of variable $\mu_1=(\mu+\nuh)t$ and $\Gamma^1=\Gamma^1_\pm\cup\Gamma^1_0,$ with $\Gamma^1_\pm=\{re^{\pm i\phi_0}, r\ge r_0 \}$ and $\Gamma^1_0=\{r_0 e^{i\vartheta},\, |\vartheta|\le \phi_0\}.$ Since, the above displayed integral is independent of such paths $\Gamma$, we obtain
\begin{equation}\label{eqintdiff}
 \begin{aligned}
e^{t\Af}\Yf_0 - e^{t\Af_h}\Pi_h\Yf_0  =\frac{e^{-\nuh t}}{2\pi t i}\int_{\Gamma^1}e^{\mu_1}\left( R\left(\frac{\mu_1}{t}-\nuh,\Af\right) - R\left(\frac{\mu_1}{t}-\nuh,\Af_h\right)\Pi_h \right)\Yf_0\, d\mu_1. 
\end{aligned}
\end{equation}

\noindent Evaluate the integral over $\Gamma^1_+$ (similarly over $\Gamma_1^-$) with $\mu_1=re^{+ i\phi_0}$ and note that $-\cos(\phi_0)>0$ to obtain
\begin{align*}
  & \left\| \int_{\Gamma^1_\pm}e^{\mu_1}\left( R\left(\frac{\mu_1}{t}-\nuh,\Af\right) - R\left(\frac{\mu_1}{t}-\nuh,\Af_h\right)\Pi_h\right)\Yf_0\, d\mu_1 \right\| \\
%   &  \qquad = \left\| \int_{r_0}^\infty e^{re^{\pm i\phi_0}}\left(R\left(\frac{re^{\pm i\phi_0}}{t}-\nuh,\Af_h\right)\Pi_h- R\left(\frac{re^{\pm i\phi_0}}{t}-\nuh,\Af\right)\right)\Yf_0\, e^{\pm i\phi_0} dr \right\|_{\Hf} \\
  & \qquad \le \sup_{r\ge r_0}\left\|\left( R\left(\frac{re^{\pm i\phi_0}}{t}-\nuh,\Af\right) - R\left(\frac{re^{\pm i\phi_0}}{t}-\nuh,\Af_h\right)\Pi_h \right)\Yf_0\right\| \int_{r_0}^\infty e^{r \cos \phi_0} \, dr \le Ch^2\|\Yf_0\| \frac{e^{r_0}}{(-\cos\phi_0)},
\end{align*}
where in the last inequality, Lemma \ref{lem:1st compresest}(b) is used.
To estimate the integral in \eqref{eqintdiff} over $\Gamma^1_0$, a substitution of $\mu_1=r_0e^{\pm i\phi}$ for $-\phi_0\le \phi\le \phi_0$ leads to
\begin{align*}
  & \left\| \int_{\Gamma^1_0}e^{\mu_1}\left(R\left(\frac{\mu_1}{t}-\nuh,\Af\right) - R\left(\frac{\mu_1}{t}-\nuh,\Af_h\right)\Pi_h \right)\Yf_0\, d\mu_1 \right\| \\
%   &  \qquad = \left\| \int_{-\phi_0}^{\phi_0} e^{r_0e^{\pm i\phi}}\left(R\left(\frac{r_0e^{\pm i\phi}}{t}-\nuh,\Af_h\right)\Pi_h- R\left(\frac{r_0e^{\pm i\phi}}{t}-\nuh,\Af\right)\right)\Yf_0\,r_0e^{\pm i\phi} (\pm i) d\phi \right\|_{\Hf} \\
  & \qquad \le \sup_{-\phi_0 \le \phi\le \phi_0}\left\|\left(R\left(\frac{r_0e^{\pm i\phi}}{t}-\nuh,\Af\right) - R\left(\frac{r_0e^{\pm i\phi}}{t}-\nuh,\Af_h\right)\Pi_h \right)\Yf_0\right\| r_0\int_{-\phi_0}^{\phi_0}  e^{r_0 \cos \phi}\, d\phi\\
  & \qquad \le Ch^2\|\Yf_0\| 2r_0e^{r_0}\phi_0,
\end{align*}
where the last inequality is obtained utilizing Lemma \ref{lem:1st compresest}(b). A combination of the last two estimates and \eqref{eqintdiff} conclude the proof.
% \begin{align*}
%     \|(e^{t\Af_h}\Pi_h- e^{t\Af})\Yf_0\|  \le C h^2\frac{e^{-\nuh t}}{ t} \|\Yf_0\|
% \end{align*}
% \noindent (b) The proof follows from (a), utilizing the fact that $\frac{1}{t}\le \frac{1}{T}=C(T)$ for all $t>T.$ \\[1mm]

\medskip
\noindent (b) For all $t>0,$ Theorems \ref{thm:ua-sem}(c) and \ref{th-unianaly-approx}(c) show
\begin{align*}
    \left\|(e^{t\Af} - e^{t\Af_h}\Pi_h)\Yf_0\right\|\le 2C e^{-\nuh t}\|\Yf_0\|.
\end{align*}
For any $0<\theta<\frac{1}{2}$, taking interpolation between the above inequality and (a), we obtain
\begin{align*}
   \left\|(e^{t\Af} - e^{t\Af_h}\Pi_h)\Yf_0\right\| \le \widetilde{C}_\theta \frac{h^{2\theta} e^{-\nuh t}}{t^\theta}\|\Yf_0\|,
\end{align*}
for some positive constant $\widetilde{C}_\theta$, depending on $\theta$.
Integrate the above relation over $[0,\infty)$ with respect to $t$ using the fact that $\theta\in (0, \frac{1}{2})$ to obtain
\begin{align*}
    \left\|(e^{t\Af} - e^{t\Af_h}\Pi_h)\Yf_0\right\|_{L^2(0,\infty;\Hf)}^2 \le C_\theta h^{4\theta}\|\Yf_0\|^2,
\end{align*}
for some constant $C_\theta>0$ depending on $\theta$ and $\nuh$.
\end{proof}

\noindent The above result gives that for any $T>0$, the error estimate between the trajectories corresponding to the continuous system and the discrete system holds uniformly in $t$ for all $t\ge T$. In the next theorem, for small time $t\in [0,T]$, the convergence result is obtained. 
\begin{Theorem}[Convergence result for system without control in finite time interval] \label{th:res-sg-est}
Let for any $\n\Yf_0=\begin{pmatrix}y_0 \\ z_0\end{pmatrix}\in\Hf$, $\n\Yf(t)=e^{t\Af}\Yf_0$ (resp $\n\Yf_h(t)=e^{t\Af_h}\Pi_h\Yf_0$) solve \eqref{eq:unctrl} (resp. \eqref{eq:d-unctrl}). Then 
for any fixed $T>0,$ 
 $\displaystyle \n \sup_{[0,T]}\left\|\left(e^{t\Af} - e^{t\Af_h}\Pi_h \right)\Yf_0\right\| \longrightarrow 0 \text{ as }h\downarrow 0.$
\end{Theorem}
\begin{proof} Let $T>0$ be fixed. The proof relies on an adaptation of the Trotter-Kato's theorem. Using Lemma \ref{lem:1st compresest}(b), the result follows from \cite[Theorem 4.2, Section 3.4, Chapter 3]{Pazy}.
\end{proof}

\section{Feedback operator and perturbation results} \label{sec:pert}
\noindent In this section, we analyze the sufficient conditions on the perturbation under which uniform analyticity and uniform stability properties of a linear operator still hold. These results are key to carry out our analysis further. We closely follow the approach given in \n\cite[Section 4.4]{Lasiecka1}. Here we get the estimates and track the dependency of the constants in our set up.

\medskip
\noindent In view of \eqref{eqcontrol_PCE} and \eqref{eqn:app B} and Theorems \ref{thm:ua-sem}, \ref{th-unianaly-approx}, we summarize the properties satisfied by the control operators and linear operators introduced in \eqref{Aw} and \eqref{app_w} in $(\mathcal{A}_1)$:\\[1mm]
\textbf{Property $(\mathcal{A}_1).$}
\begin{itemize}
\item[(a)] The control operators $\Bf$ and $\Bf_h$ are given in \eqref{eqcontrol_PCE} and \eqref{eqn:app B} satisfy
\begin{align*}
\|\Bf\|_{\Lc(\Uf,\Hf)}\le C_B \text{ and } \|\Bf_h\|_{\Lc(\Uf,\Hf_h)}\le C_B,
\end{align*}
for some positive constant $C_B$ independent of $h$.
\item[(b)] The operator $(\Aw,D(\Aw))$ defined in \eqref{Aw} generates an analytic semigroup on $\Hf$ with $\Sigma^c(-\nuh+\omega;\theta_0)\subset \rho(\Af_{\omega})$ and satisfies the resolvent estimate 
\begin{align*}
\|R(\mu,\Aw)\|_{\Lc(\Hf)} \le \frac{C_1}{|\mu+\nuh-\omega|} \text{ for all }\mu\in \Sigma^c(-\nuh+\omega;\theta_0),\, \mu\neq -\nuh+\omega,
\end{align*}
for some positive constant $C_1$ independent of $\mu$.
\item[(c)] For all $h>0,$ $\Af_{\omega_h}$ defined in \eqref{app_w} generates a uniformly analytic semigroup on $\Hf_h$ with $\Sigma^c(-\nuh+\omega;\theta_0)\subset \cap_{h>0}\rho(\Af_{\omega_h})$ and satisfies the resolvent estimate
\begin{align*}
\|R(\mu,\Af_{\omega_h})\|_{\Lc(\Hf_h)} \le \frac{C_1}{|\mu+\nuh-\omega|} \; \text{ for all } \mu \in \Sigma^c(-\nuh+\omega;\theta_0), \, \mu \neq -\nuh+\omega,
\end{align*}
 for some positive constant $C_1$ independent of $\mu$ and $h.$ 
\item[(d)] The operators $(\Aw,D(\Aw))$ and $\Af_{\omega_h}$ defined in \eqref{Aw} and  \eqref{app_w}, respectively, satisfy 
$$\displaystyle \sup_{\mu \in \Sigma^c(-\nuh+\omega;\theta_0)} \| R(\mu,\Aw)-R(\mu,\Af_{\omega_h})\Pi_h\|_{\Lc(\Hf)} \le Ch^2,$$
for some positive $C$ independent of $\mu$ and $h.$ 
\end{itemize} 
In $(\mathcal{A}_1);$ (b), (c), and (d) hold as consequences of Theorems \ref{thm:ua-sem}, \ref{th-unianaly-approx}, and Lemma \ref{lem:1st compresest}, respectively.\\
Next, we assume uniform boundedness of perturbed operators $\wt\Ff_h\in \Lc(\Hf,\Uf)$ and $\Ff_h\in\Lc(\Hf_h,\Uf) $ in   $(\mathcal{A}_2)-(\mathcal{A}_4)$ below.

\medskip 

\noindent \textbf{Assumptions.}
\begin{itemize}
\item[($\mathcal{A}_2$).] For all $h>0$, let $\wt \Ff_h\in \Lc(\Hf,\Uf)$ be a family of operators such that $\|\wt \Ff_h\|_{\Lc(\Hf,\Uf)} \le C_2$ for some positive constant $C_2$ independent of $h.$
\item[($\mathcal{A}_3$).] For all $h>0,$ let $\Ff_h \in \Lc(\Hf_h,\Uf)$ be a family of operators such that $\|\Ff_h\|_{\Lc(\Hf_h,\Uf)} \le C_3$ for some positive constant $C_3$ independent of $h.$
\item[$(\mathcal{A}_4)$.] Let $\Ff\in \Lc(\Hf,\Uf)$ and for all $h>0$, the operator $\Ff_h\in \Lc(\Hf_h,\Uf)$ satisfy $\|\Ff-\Ff_h\|_{\Lc(\Hf_h,\Uf)}\le C h^s,$ for some $s\in (0, 2]$ and $C>0$ independent of $h.$
\end{itemize}
These assumptions will be verified in Sections \ref{sec:disRiccati}-\ref{sec:results} in our set up.
\medskip
 \noindent 
%Let $\wt \Ff_h\in \Lc(\Hf,\Uf)$ and $\Ff_h\in \Lc(\Hf_h,\Uf)$ satisfy %$(\mathcal{A}_2)$ and $(\mathcal{A}_3)$. 
For all $h>0$, set 
\begin{align}\label{eqAwpAwhph}
\Af_{\omega,\wt\Ff_h}:=\Aw+\Bf\wt \Ff_h \text{ and } \Af_{\omega_h,\Ff_h}:=\Af_{\omega_h}+\Bf_h\Ff_h,
\end{align}
where $\Aw, \Af_{\omega_h}, \Bf, \Bf_h$ satisfy $(\mathcal{A}_1)$.
The next subsections establish the uniform analyticity and then uniform stability of the above perturbed operators provided $\wt\Ff_h$ and $\Ff_h$ satisfy suitable conditions.

\subsection{Uniform analyticity of perturbed operators}\label{sec:unianalytpert}

% In this subsection, it is shown that under $(\mathcal{A}_2)$ and $(\mathcal{A}_3)$, $\Af_{\omega,\wt\Ff_h}$ and $\Af_{\omega_h,\Ff_h}$, defined in \eqref{eqAwpAwhph}, generate uniformly analytic semigroups on $\Hf$ and $\Hf_h$, respectively. 
% To prove this, we show that there exist $\omega'\in \mathbb{R}$ and $\theta_0\in (\frac{\pi}{2}, \pi)$, uniform in $h$, such that 
% $\Sigma^c(\omega'; \theta_0)$ belongs to the resolvent set of $\Af_{\omega,\wt\Ff_h}$ and $\Af_{\omega_h,\Ff_h},$ and a uniform resolvent estimate holds on $\Sigma^c(\omega'; \theta_0)$. 

\noindent Let $\theta_0\in (\frac{\pi}{2},\pi)$ be as in Theorem \ref{thm:ua-sem} and 
set $\omega'\in \mathbb{R}$ such that 
\begin{align} \label{eqdef:omega'}
\omega'>-\nuh+\omega, \quad \mathrm{and}\quad |\omega'+\nuh-\omega|> \max \left\lbrace \frac{C_1 C_2C_B}{\sin(\theta_0)},  \frac{C_1C_3 C_B}{\sin(\theta_0)} \right\rbrace,
\end{align}
where the constants $C_B, C_1, C_2, C_3$ appear in $(\mathcal{A}_1)-(\mathcal{A}_3)$ (see Figure~\ref{fig:d-spec region}). 

\medskip
\noindent
Since $\omega'>-\nuh+\omega$ and $\theta_0\in (\frac{\pi}{2},\pi)$,  $(-\nuh+\omega)\notin \Sigma^c(\omega';\theta_0)$ and the definition of distance between the set $\Sigma^c(\omega';\theta_0)$ from $-\nuh+\omega$ implies
\begin{align}\label{eq:trigres}
 |\mu+\nuh-\omega| \ge \sin(\theta_0)|\omega'+\nuh-\omega| \text{ for all }\mu\in \Sigma^c(\omega';\theta_0).
\end{align}
The definition of distance from point $\mu$ to the real axis and the angle of the vector joining points $\mu$ and $\omega'$ with the real axis lead to
\begin{align}\label{eq:trigres-2}
|\mu+\nuh-\omega| \ge \sin(\theta_0)|\mu-\omega'| \text{ for all }\mu\in \Sigma^c(\omega';\theta_0).
\end{align}

% \end{Lemma}

\begin{Lemma}[uniform analyticity of $e^{t\Af{_{\omega_h,\Ff_h}}}$] \label{lem:uah}
Let (a) and (c) of $(\mathcal{A}_1)$ hold. Let $\omega'$ be as given in \eqref{eqdef:omega'} and $\theta_0$ be as introduced in Theorem \ref{thm:ua-sem}. Let $ \normalfont \Ff_h\in \Lc(\Hf_h,\Uf)$ be such that $(\mathcal{A}_3)$ holds and $\n \Af{_{\omega_h,\Ff_h}}$ be as defined in \eqref{eqAwpAwhph}. Then for all $h>0$, the following holds:
\begin{itemize}
\item[(a)] $\n\Sigma^c(\omega'; \theta_0)\subset \rho(\Af_{\omega_h,\Ff_h}),$ and
\item[(b)] for some $C>0$ independent of $h,$ $
\n \|R(\mu,\Af_{\omega_h,\Ff_h})\|_{\Lc(\Hf_h)}\le \frac{C}{|\mu-\omega'|}\text{ for all }\mu\in \Sigma^c(\omega'; \theta_0),\, \mu\neq \omega'.$
\end{itemize}
The operator $\n\Af_{\omega_h,\Ff_h}$ generates a uniformly analytic semigroup $\n\{e^{t\Af_{\omega_h,\Ff_h}}\}_{t\ge 0}$ on $\n\Hf_h$ satisfying 
\begin{align*}
\|  e^{t\Af_{\omega_h,\Ff_h}}\|_{\Lc(\Hf_h)}\le C e^{\omega ' t} \text{ for all } t>0, \, h>0.
\end{align*}
\end{Lemma}

\begin{proof}
Note that $\Sigma^c(\omega';\theta_0)\subset \Sigma^c(-\nuh+\omega;\theta_0)\subset \rho(\Af_{\omega_h})$ for all $h>0$ (see Figure \ref{fig:d-spec region}). Thus, for any $\mu  \in \Sigma^c(\omega';\theta_0)$, using $(a),$ $(c)$ of $(\mathcal{A}_1)$, $(\mathcal{A}_3)$, \eqref{eq:trigres}, and  \eqref{eqdef:omega'}, we have
\begin{equation} \label{eqn:estR(m,Awhkh)BhKh}
\begin{aligned}
\|R(\mu,\Af_{\omega_h})\Bf_h\Ff_h\|_{\Lc(\Hf_h)} & \leq \|R(\mu,\Af_{\omega_h})\|_{\Lc(\Hf_h)} \|\Bf_h\|_{\Lc(\Uf,\Hf_h)} \|\Ff_h\|_{\Lc(\Hf_h,\Uf)} \\
& \leq \frac{C_1 C_BC_3}{\vert \mu+\nuh-\omega\vert} \le \frac{C_1  C_B C_3}{\sin(\theta_0)\vert \omega'+\nuh-\omega\vert}=:\delta_0'<1\; \text{ for all } \mu\in \Sigma^c(\omega';\theta_0).
\end{aligned}
\end{equation}
Therefore, for all $\mu\in \Sigma^c(\omega';\theta_0),$ with $\mathfrak{T}_1=\mu\If_h-\Af_{\omega_h}$ and $\mathfrak{T}_2=(\mu\If_h-\Af_{\omega_h,\Ff_h}),$ Lemma \ref{lem:id-3}(b) leads to the existence of $R(\mu,\Af_{\omega_h,\Ff_h})$ in $\Lc(\Hf_h)$ with
\begin{align*}
R(\mu,\Af_{\omega_h,\Ff_h})=(\If_h -  R(\mu,\Af_{\omega_h})\Bf_h\Ff_h)^{-1} R(\mu,\Af_{\omega_h}),
\end{align*}
and hence $ \Sigma^c(\omega';\theta_0)\subset \rho(\Af_{\omega_h,\Ff_h})$ for all $h>0.$ Applying \eqref{id-1} for $\mathfrak{T} = R(\mu,\Af_{\omega_h})\Bf_h\Ff_h$, $(c)$ of $(\mathcal{A}_1)$, \eqref{eqn:estR(m,Awhkh)BhKh}, and \eqref{eq:trigres-2}, for all $h>0,$ we obtain
\begin{align*}
\|R(\mu,\Af{_{\omega_h,\Ff_h}})\|_{\Lc(\Hf_h)} & \leq \frac{1}{1-\|R(\mu,\Af_{\omega_h})\Bf_h\Ff_h\|_{\Lc(\Hf_h)}} \|R(\mu,\Af_{\omega_h})\|_{\Lc(\Hf_h)} \\
& \leq \frac{1}{1-\delta_0'} \frac{C_1}{|\mu+\nuh-\omega|} \le \frac{1}{1-\delta_0'} \frac{C_1}{\sin(\theta_0)|\mu-\omega'|} \text{ for all } \mu \;(\neq \omega') \in \Sigma^c(\omega';\theta_0).
\end{align*}
Since, the constants appearing in the last displayed estimates are independent of $\mu$ and $h,$ Definition \ref{def:ana-sgp} and arguments in the proof of Theorem \ref{thm:ua-sem}(b) conclude the proof.
\end{proof}

%------------------------------------------------------

\begin{figure}[ht!]
	\begin{center}
		\begin{tikzpicture}
		    \path[draw, fill=gray!20]         (-2.8,-5)--(-1.7,0)--(-2.8,5)--(6,0)--cycle;
		    
		    \draw[thick,->,dashed] (-2.8,-5)--(-3.3,-5.3);
			\draw[thick,->,dashed] (-2.8,5)--(-3.3,5.3);
		    
		    \draw[thick,<->,dashed] (-7,0)--(7.5,0);
			\draw[thick,<->,dashed] (0,-5.5)--(0,5.5);
			
			\draw[thick,dashed,blue,->] (2,0)--(-4,1.8);
			\draw[thick,dashed,blue,->] (2,0)--(-4,-1.8);
			
			\draw[thick,dashed,blue,->] (6,0)--(-3.5,4.3);
			\draw[thick,dashed,blue,->] (6,0)--(-3.5,-4.3);
			
			\draw[thick,dashed,orange,->] (-2.5,0)--(-6,1.2);
			\draw[thick,dashed,orange,->] (-2.5,0)--(-6,-1.2);
			\draw[orange] (-2,0) arc (0:153:0.5cm);
			\node[orange] at (-2.2,0.7){$\theta_P$};
			
		%	\draw[orange] (-2.5,-2.9)--(-2.5,2.9);
			
		%	\draw[thick,<->] (-2.5,1.2)--(-1.7,1.2);
		%	\node[] at (-2.2,1){$\epsilon$};
			
			\draw[<->] (-2.8,3) arc (135:230:4cm);
			\node[] at (-2.1,-1.3){{\huge$\Sigma(\omega';\theta_0)$}};
			
			\draw[dashed,blue] (2.5,0) arc (0:150:0.5cm);
			\node[blue] at (2.6,0.4){$\theta_0$};			
			\draw[thick,->] (4.5,0)--(-4,3.5);
			\draw[thick,,->] (4.5,0)--(-4,-3.5);
			\node[] at (4.5,-0.2){$\omega'$};
			\node[] at (6,-0.3){$\omega'+\ell$};
			
			\draw[] (5,0) arc (0:153:0.5cm);
			\node[] at (5.1,0.4){$\theta_0$};
			
			\draw[dashed,blue] (6.5,0) arc (0:153:0.5cm);
			\node[blue] at (6.6,0.4){$\theta_0$};
			
			\draw[] (7,0) arc (0:148:1cm);
			\node[] at (6.8,1){$\wt\theta_0$};
			
			\draw[] (-1,0) arc (0:133:0.5cm);
			\node[] at (-0.9,0.4){$\phi_0$};
			
			\node[] at (-1.7,0.2){$-\wh\gamma$};
			\node[] at (-1.7,0){$\bullet$};
			\node[red] at (-2.5,-0.3){$-\gamma$};
			\node[red] at (-2.5,0){$\bullet$};
			
			\node[] at (-0.2,-0.2){O};
			\node[] at (-2.4,3.1){$P$};
			\node[] at (-2.4,-3.1){$Q$};
			\node[] at (-1.6,-0.2){$R$};
			\node[blue] at (2.2,-0.2){$-\nuh+\omega$};
			
			\node[] at (-0.3,2){{\Huge$\mathcal{K}_{\wh\gamma}$}};
\end{tikzpicture}
\end{center}
\caption{Spectrum region for $\Aw{_{,\Ff}}$ and $\Af{_{\omega_h,\Ff_h}}$}  \label{fig:d-spec region}
\end{figure}

%------------------------------------------------------

%Next, for the operator $\Af_{\omega,\wt \Ff_h}$ as defined in \eqref{eqAwpAwhph},  the result below is analogous to the above lemma. 

\begin{Lemma}[uniform analyticity of $e^{t\Af{_{\omega,\wt \Ff_h}}}$]\label{lem:uah-AwKh}
Let (a) and (b) of $(\mathcal{A}_1)$ hold. Let $\omega'$ be as given in \eqref{eqdef:omega'} and $\theta_0$ be as in Theorem \ref{thm:ua-sem}. Let  $ \wt \Ff_h\in \Lc(\Hf,\Uf)$ be a family of bounded operators such that $(\mathcal{A}_2)$ holds  and  $\Af_{\omega,\wt \Ff_h}$ be as defined in \eqref{eqAwpAwhph}.  Then for all $h>0,$ the following holds:
\begin{itemize}
    \item[(a)]  $\Sigma^c(\omega';\theta_0)\subset \rho(\Af_{\omega,{\wt\Ff_h}}),$
    \item[(b)] for some $C>0$ independent of $h,$ $\|R(\mu, \Af_{\omega,{\wt \Ff_h}})\|_{\Lc(\Hf)} \le \frac{C}{|\mu-\omega'|} \text{ for all } \mu (\neq \omega')\in \Sigma^c(\omega';\theta_0).$
\end{itemize}
The operator $\n\Af_{\omega,\wt \Ff_h}$ generates 
a uniformly analytic semigroup $\n\{e^{t\Af_{\omega,\wt \Ff_h}}\}_{t\ge 0}$ on $\n\Hf$ satisfying 
\begin{align*}
\|  e^{t\Af_{\omega,\wt \Ff_h}}\|_{\Lc(\Hf)}\le C e^{\omega' t} \text{ for all } t>0, \, h>0.
\end{align*}
%Moreover, this is true for $\wt \Ff_h=\Ff\in \Lc(\Hf,\Uf)$ which is independent of $h>0.$
\end{Lemma}

\noindent The proof is analogous to that of Lemma \ref{lem:uah} using (b) of $(\mathcal{A}_1)$ and is skipped.\qed

\medskip
\noindent In view of Lemmas \ref{lem:uah} and \ref{lem:uah-AwKh}, note that 
\begin{align}\label{eqn:allResol}
\Sigma^c(\omega'; \theta_0)\subset \rho(\Af_{\omega,\wt \Ff_h}) \cap\big\lbrace \displaystyle\cap_{h>0}\rho(\Af{_{\omega_h,\Ff_h}})\big\rbrace\cap \rho(\Aw)\cap \big\lbrace \displaystyle\cap_{h>0}\rho(\Af_{\omega_h})\big\rbrace,
\end{align}
(see Figure \ref{fig:d-spec region}).

\medskip
\noindent Next, we establish the convergence of the resolvent operators of $\Af_{\omega,\wt \Ff_h}$ and $\Af_{\omega_h,\Ff_h}$ under the hypothesis $(\mathcal{A}_4).$ This result is crucial to establish the stabilizability of the discrete system and error estimates. 
 
\begin{Lemma}[convergence of the perturbed operator] \label{lem:cvg-res-rc}
Let $(\mathcal{A}_1)$ hold, and $\omega'$ and $\theta_0$ be as defined in \eqref{eqdef:omega'} and Theorem \ref{thm:ua-sem}, respectively. 
Let for all $h>0,$ $\Ff_h\in \Lc(\Hf_h,\Uf)$ satisfy $(\mathcal{A}_3)$ and $\n\Af{_{\omega_h,\Ff_h}}$ be as defined in \eqref{eqAwpAwhph}. For all $h>0$, there exists a positive constant $C$, independent of $h$ and $\mu,$ such that the following results hold:\\
 (a) Let $\Ff\in \Lc(\Hf,\Uf)$ such that $\|\Ff\|_{\Lc(\Hf, \Uf)}\le C_2$ and $(\mathcal{A}_4)$ hold. Then for  $\Af_{\omega,\Ff}:=\Aw+\Bf\Ff,$  
\begin{equation}
\n\sup_{\mu\in \Sigma^c(\omega'; \theta_0)} \left\|R(\mu,\Af_{\omega,\Ff})-R(\mu,\Af_{\omega_h,\Ff_h})\Pi_h\right\|_{\Lc(\Hf)}\le Ch^s,
\end{equation}
where $s\in (0, 2]$ is the same as in $(\mathcal{A}_4)$. \\
(b) Let $\wt \Ff_h=\Ff_h\Pi_h$ and $\n\Aw{_{,\wt \Ff_h}}$ be as defined in \eqref{eqAwpAwhph}. Then 
\begin{equation}
\n\sup_{\mu\in \Sigma^c(\omega'; \theta_0)} \left\|R(\mu,\Af_{\omega,\wt \Ff_h})-R(\mu,\Af_{\omega_h,\Ff_h})\Pi_h\right\|_{\Lc(\Hf)}\le Ch^2.
\end{equation}
\end{Lemma}
\begin{proof}
(a) The proof is established in the following three steps.\\[1mm]
\textbf{Step 1.} Let $\mu \in \Sigma^c(\omega'; \theta_0)$ be arbitrary. Apply \eqref{eqn:allResol} for $\wt\Ff_h=\Ff,$ Lemma \ref{lem:id-3}(b) with $\mathfrak{T}_1=\mu\If-\Af_{\omega}$ and $\mathfrak{T}_2=\mu\If-\Af_{\omega,\Ff}$ to obtain
$R(\mu,\Af_{\omega,\Ff})=(\If-R(\mu,\Aw)\Bf\Ff)^{-1}$ $R(\mu,\Aw) $ and similarly $
		R(\mu,\Af{_{\omega_h,\Ff_h}})=(\If_h-R(\mu,\Af_{\omega_h})\Bf_h\Ff_h)^{-1}R(\mu,\Af_{\omega_h})$. Thus, an addition and  subtraction of $(\If-R(\mu,\Aw)\Bf\Ff)^{-1}R(\mu,\Af_{\omega_h})\Pi_h$ yields
\begin{align*}
R(\mu,\Af_{\omega,\Ff})-R(\mu,\Af{_{\omega_h,\Ff_h}})\Pi_h &  =: T_1+T_2,
\end{align*}
where
	\begin{align*}
	T_1& :=\left(\If-R(\mu,\Aw)\Bf\Ff)\right)^{-1}\left(R(\mu,\Aw)-R(\mu,\Af_{\omega_h})\Pi_h\right) \text{ and }\\
	T_2 & :=\Big(\big(\If-R(\mu,\Aw)\Bf\Ff\big)^{-1}-\big(\If_h-R(\mu,\Af_{\omega_h})\Bf_h\Ff_h\big)^{-1}\Big)R(\mu,\Af_{\omega_h})\Pi_h.
	\end{align*}

\noindent The aim is to show  $\|T_i\|_{\Lc(\Hf)}\le Ch^s,$ for each $i=1,2$ and for any $0<s\le 2.$\\[1mm]
\textbf{Step 2.} Using $(a)$ and $(b)$ of $(\mathcal{A}_1),$ $\|\Ff\|_{\Lc(\Hf,\Uf)}\le C_2$, \eqref{eq:trigres}, and \eqref{eqdef:omega'}, we have
\begin{equation} \label{eqn:estR(m,Awk)BK}
\begin{aligned}
\|R(\mu,\Af_{\omega})\Bf\Ff\|_{\Lc(\Hf)} & \leq \|R(\mu,\Af_{\omega})\|_{\Lc(\Hf)} \|\Bf\|_{\Lc(\Uf,\Hf)} \|\Ff\|_{\Lc(\Hf,\Uf)} \\
& \leq \frac{C_1 C_BC_2}{\vert \mu+\nuh-\omega\vert} \le \frac{C_1  C_B C_2}{\sin(\theta_0)\vert \omega'+\nuh-\omega\vert}=:\delta_0'<1\; \text{ for all } \mu\in \Sigma^c(\omega';\theta_0).
\end{aligned}
\end{equation}
Apply \eqref{id-1} with $\mathfrak{T}=R(\mu,\Aw)\Bf\Ff$ and  use \eqref{eqn:estR(m,Awk)BK} to obtain,
\begin{align}\label{unibound41}
\|\left(\If-R(\mu,\Aw)\Bf\Ff\right)^{-1}\|_{\Lc(\Hf)} \le \frac{1}{1-\delta_0'}\, \text{ for all }\mu\in \Sigma^c(\omega';\theta_0).
\end{align}
Thus, $(d)$ of $(\mathcal{A}_1)$ and \eqref{unibound41} leads to $\|T_1\|_{\Lc(\Hf)}\le Ch^2 $ uniform for all $ \mu\in \Sigma^c(\omega';\theta_0).$\\[1mm]
\textbf{Step 3.} Re-write $T_2$ as
\begin{equation} \label{eqn:T_2}
\begin{aligned}
    T_2 & =\Big(\big(\If-R(\mu,\Aw)\Bf\Ff\big)^{-1} \left( \big(\If_h-R(\mu,\Af_{\omega_h})\Bf_h\Ff_h\big) - \big(\If-R(\mu,\Aw)\Bf\Ff\big) \right) \big(\If_h-R(\mu,\Af_{\omega_h})\Bf_h\Ff_h\big)^{-1}\Big)\\
    & \qquad \qquad \times R(\mu,\Af_{\omega_h})\Pi_h \\
    & = \Big(\big(\If-R(\mu,\Aw)\Bf\Ff\big)^{-1} \left( R(\mu,\Aw)\Bf\Ff-R(\mu,\Af_{\omega_h})\Bf_h\Ff_h \right) \big(\If_h-R(\mu,\Af_{\omega_h})\Bf_h\Ff_h\big)^{-1}\Big)R(\mu,\Af_{\omega_h})\Pi_h.
\end{aligned}
\end{equation}
Utilize \eqref{eqn:estR(m,Awhkh)BhKh} and \eqref{id-1} with $\mathfrak{T}=R(\mu,\Af_{\omega_h})\Bf_h\Ff_h$ to obtain
\begin{align}\label{eqn:req-1_T2}
    \|\left(\If_h-R(\mu,\Af_{\omega_h})\Bf_h\Ff_h\right)^{-1}\|_{\Lc(\Hf_h)} \le \frac{1}{1-\delta_0'} \text{ for all }\mu\in \Sigma^c(\omega';\theta_0).
\end{align}
For all $\mu\in \Sigma^c(\omega';\theta_0),$ $(c)$ of $(\mathcal{A}_1)$ and \eqref{eq:trigres}  lead to
\begin{align}\label{eqn:req-2_T2}
    \|R(\mu,\Af_{\omega_h})\Pi_h\|_{\Lc(\Hf)}\le \frac{C_1}{|\mu+\nuh-\omega|}\le \frac{C_1}{\sin(\theta_0)|\omega'+\nuh-\omega|}.
\end{align}
% It remains to estimate $\|R(\mu,\Aw)\Bf\Ff-R(\mu,\Af_{\omega_h})\Bf_h\Ff_h\|_{\Lc(\Hf_h,\Hf)}.$ 
An addition and subtraction of $R(\mu,\Aw)\Bf\Ff_h$ followed by a triangle inequality yields
%$R(\mu,\Aw)\Bf\Ff-R(\mu,\Af_{\omega_h})\Bf_h\Ff_h=R(\mu,\Aw)\Bf\Ff-R(\mu,\Aw)\Bf\Ff_h+R(\mu,\Aw)\Bf\Ff_h-R(\mu,\Af_{\omega_h})\Bf_h\Ff_h$
% \begin{equation} \label{eqn:req-3_T3}
% \begin{aligned}
% \|R(\mu,\Aw)\Bf\Ff-R(\mu,\Af_{\omega_h})\Bf_h\Ff_h\|_{\Lc(\Hf_h,\Hf)} 		& \leq \|R(\mu,\Aw)\Bf(\Ff_h-\Ff)\|_{\Lc(\Hf_h,\Hf)} \\
% 			& \quad + \|\big(R(\mu,\Aw)\Bf-R(\mu,\Af_{\omega_h})\Pi_h\Bf\big)\Ff_h\|_{\Lc(\Hf_h)}\\
% 			& \le \frac{C_1C_B}{\sin(\theta_0)|\omega'+\nuh-\omega|}\|\Ff_h-\Ff\|_{\Lc(\Hf_h,\Uf)}  \\
% 			& \quad+ \|R(\mu,\Aw)-R(\mu,\Af_{\omega_h})\Pi_h\|_{\Lc(\Hf)}C_BC_3,
% \end{aligned}
% \end{equation} 
\begin{align*}
\|R(\mu,\Aw)\Bf\Ff-R(\mu,\Af_{\omega_h})\Bf_h\Ff_h\|_{\Lc(\Hf_h,\Hf)} 		& \leq \|R(\mu,\Aw)\Bf(\Ff_h-\Ff)\|_{\Lc(\Hf_h,\Hf)} \\
			& \quad + \|\big(R(\mu,\Aw)\Bf-R(\mu,\Af_{\omega_h})\Pi_h\Bf\big)\Ff_h\|_{\Lc(\Hf_h)}.
\end{align*}
A use of $(a)-(b)$ of $(\mathcal{A}_1),$ \eqref{eq:trigres}, and $(\mathcal{A}_3)$ in the above inequality yield
\begin{equation} \label{eqn:req-3_T3}
\begin{aligned}
\|R(\mu,\Aw)\Bf\Ff-R(\mu,\Af_{\omega_h})\Bf_h\Ff_h\|_{\Lc(\Hf_h,\Hf)} 		&  \le \frac{C_1C_B}{\sin(\theta_0)|\omega'+\nuh-\omega|}\|\Ff_h-\Ff\|_{\Lc(\Hf_h,\Uf)}  \\
			& \quad+ C_BC_3\|R(\mu,\Aw)-R(\mu,\Af_{\omega_h})\Pi_h\|_{\Lc(\Hf)}.
\end{aligned}
\end{equation} 
Utilize \eqref{unibound41}, \eqref{eqn:req-1_T2}, \eqref{eqn:req-2_T2}, and \eqref{eqn:req-3_T3} with $(\mathcal{A}_4)$ and $(d)$ of $(\mathcal{A}_1)$ in \eqref{eqn:T_2} to obtain the estimate  $\|T_2\|_{\Lc(\Hf)}\le Ch^s.$ A combination of Steps (1) - (3) concludes the proof of (a). \\[2mm]
\noindent (b) Since $\wt\Ff_h=\Ff_h\Pi_h$ and $\Ff_h$ satisfies $(\mathcal{A}_3),$ $\wt\Ff_h$ satisfies $(\mathcal{A}_2)$ with constant $C_3$ and $\wt\Ff_h=\Ff_h$ on $\Hf_h$. Consequently, $(\mathcal{A}_4),$
holds for $s=2$. Hence, we conclude the proof by arguments analogous to (a).
\end{proof}

\subsection{Uniform Stability of perturbed operators}\label{sec:pertstab}
% In this subsection, exponential stability of the semigroup that is perturbed by the feedback operator is studied. Also, some auxiliary results that will be used for the proof of main results in next section are discussed. Recall the feedback operators $\wt \Ff_h\in\Lc(\Hf,\Uf)$, $\Ff_h\in\Lc(\Hf_h,\Uf)$ and the corresponding perturbed operators 
% \begin{align}\label{eqAwpAwhph}
% \Aw{_{,\wt \Ff_h}}:=\Aw+\Bf\wt \Ff_h\text{ and }\Af{_{\omega_h,\Ff_h}}:=\Af_{\omega_h}+\Bf_h\Ff_h.
% \end{align}

\noindent 
Here, under additional assumption $(\mathcal{A}_5)$ below and for suitable perturbation, we establish a sharper estimate for the upper bound of the spectrum of the perturbed operators.\\[1mm] 
\noindent \textbf{Assumption.}
\begin{itemize}
\item[$(\mathcal{A}_5)$.] Let $\Ff\in \Lc(\Hf, \Uf)$ be such that $(\Af_\omega+\Bf \Ff)$ is exponentially stable in $\Hf$. In particular, denoting $\Af_{\omega,\Ff}:=\Aw+\Bf \Ff$, there exist positive constants $C$ and $\gamma$ such that $\|e^{t\Af_{\omega,\Ff}}\|_{\Lc(\Hf)}\le Ce^{-\gamma t} \text{ for all } t>0,$
and 
$\displaystyle\sup_{\Lambda\in \sigma(\Af_{\omega,\Ff})} \Re(\Lambda) \le -\gamma.$ There exists $\theta_P \in (\frac{\pi}{2}, \pi)$ such that $\Sigma^c(-\gamma;\theta_P)\subset \rho(\Af_{\omega,\Ff}).$
 \end{itemize}
Let $(\mathcal{A}_1)$ and $(\mathcal{A}_3)-(\mathcal{A}_5)$ hold. For any $\wh\gamma\in (0, \gamma),$ we want to show that there exists $h_{\wh\gamma}>0$ such that
$$\displaystyle \sup_{\Lambda\in \sigma(\Af_{\omega_h,\Ff_h})} \Re(\Lambda) \le -\wh\gamma, \text{ for all } 0<h<h_{\wh\gamma},$$
and that the uniform stability  estimate for $e^{t\Af_{\omega_h,\Ff_h}}$ holds with the exponential decay $-\wh\gamma.$
In particular, we want to determine a sector $\Sigma(-\wh\gamma;\phi_0)$ for some $\frac{\pi}{2}<\phi_0<\pi$ such that $\Sigma^c(-\wh\gamma;\phi_0)\subset \rho(\Af_{\omega_h,\Ff_h})$, for all $0<h<h_{\wh\gamma}$, and a uniform resolvent estimate holds. Let $\omega',$ $\theta_0,$ and $\theta_P$ be as in \eqref{eqdef:omega'}, Lemma \ref{lem:uah}, and $(\mathcal{A}_5),$ respectively. Let $\wt\theta_0:=\min\{\theta_0,\theta_P\}$ and fix 
\begin{align}\label{eqn:phi_0}
    \frac{\pi}{2}<\phi_0< \wt\theta_0,
\end{align}
(see Figure \ref{fig:d-spec region}).
Now, for any $\wh\gamma\in(0,\gamma)$ and $\phi_0$ as above, we consider $\Sigma^c(-\wh\gamma;\phi_0)$. Because of the choice of $\phi_0$, for any $\wh\gamma\in (0, \gamma)$, 
$ \Sigma^c(-\wh\gamma;\phi_0)\subset \overline{\Sigma^c(-\wh\gamma;\phi_0)}\subset \Sigma^c(-\gamma;\theta_P)\subset \rho(\Af_{\omega,\Ff}).$
Next, to prove $\Sigma^c(-\wh\gamma;\phi_0)\subset \rho(\Af_{\omega_h,\Ff_h})$, for all $0<h<h_{\wh\gamma}$, we set 
\begin{equation}\label{eqcompactset}
\mathcal{K}_{\wh\gamma}=\overline{\Sigma(\omega'+\ell;\wt\theta_0)\cap  \Sigma^c(-\wh\gamma;\phi_0)},
\end{equation}
(shaded region in Figure \ref{fig:d-spec region}) for any $\ell>0$, where $\phi_0$ is as in \eqref{eqn:phi_0}. Note that $\Sigma^c(-\wh\gamma;\phi_0) \subset \mathcal{K}_{\wh\gamma}\cup \Sigma^c(\omega'+\ell;\wt\theta_0)$ and 
$\Sigma^c(\omega'+\ell;\wt\theta_0) \subset \Sigma^c(\omega'+\ell;\theta_0)\subset \Sigma^c(\omega';\theta_0),$ since $\frac{\pi}{2}< \wt\theta_0\le \theta_0$ and $\ell>0$. Therefore, for all $h>0,$ Lemma \ref{lem:uah} yields  $\Sigma^c(\omega'+\ell;\wt\theta_0)\subset \rho(\Af_{\omega_h,\Ff_h})$ and the resolvent estimate for $R(\mu,\Af_{\omega_h,\Ff_h})$ holds for all $\mu\in \Sigma^c(\omega'+\ell;\wt\theta_0).$ 
Next, we show that there exists a $h_{\wh\gamma}>0$ such that for all $0<h<h_{\wh\gamma}$, $\mathcal{K}_{\wh\gamma}\subset \rho(\Af_{\omega_h,\Ff_h}),$ and the resolvent operator $R(\mu,\Af_{\omega_h,\Ff_h})$ is uniformly (in $\mu$ and $h$) bounded for all $\mu\in \mathcal{K}_{\wh\gamma}$ and for all $0<h<h_{\wh\gamma}$. This will give us the required estimate on the spectrum of $\Af_{\omega_h,\Ff_h}$ and the uniform stability estimate with the decay $-\wh\gamma.$

\medskip
\noindent Note that because of the choice of $\phi_0$ in \eqref{eqn:phi_0}, the lines $\{\omega'+\ell+re^{i\wt\theta_0}\mid r>0\}$ and $\{-\wh\gamma+\rho e^{i\phi_0}\mid \rho>0\}$ intersect at the point $\omega'+\ell+r_0e^{i\wt\theta_0}$, where $r_0=\frac{(\omega'+\ell+\wh\gamma)\sin\phi_0}{\sin(\wt\theta_0-\phi_0)}$ and it can be shown that for any $\mu\in \mathcal{K}_{\wh\gamma}$, $|\mu-\omega'-\ell|\le r_0$. 
Thus, $\mathcal{K}_{\wh\gamma}$, defined in \eqref{eqcompactset} is a compact set in $\mathbb{C}$ (see Figure \ref{fig:d-spec region}). To obtain our required result, we use the results on the spectrum of the operators under perturbation. For details, see \cite[Chapter 4]{Kato}.
Here, we mention the result applicable to our context.

\begin{Lemma}[invertibility of perturbed operator{\n  \cite[Lemma 6.7.4]{Badra-ths}}]\label{lem:main_genconv}
Let $\Xf$ and $\Wf$ be two Banach spaces and $E$ be a compact subset in $\Cb$. Let $S:E\rightarrow \Lc(\Xf,\Wf)$ be a bounded map. Assume that $\Sf:D(\Sf)\subset \Xf\rightarrow \Wf$ and $\Tf:D(\Tf)\subset \Xf\rightarrow \Wf$ are linear operators such that $\Sf^{-1},$ $\Tf^{-1}$ and $(\Sf+S(\mu))^{-1}$ belongs to $\Lc(\Wf,\Xf)$ for all $\mu\in E.$  Let for all $\mu\in E,$ 
\begin{align}\label{eqn:invT-invS}
    \|\Tf^{-1}-\Sf^{-1}\|_{\Lc(\Wf,\Xf)}< \frac{1}{C(\mu)},
\end{align}
where  
$C(\mu) = 2(1+\|S(\mu)\|_{\Lc(\Xf,\Wf)})^2\left\|\left( \Sf+S(\mu)  \right)^{-1}\right\|_{\Lc(\Wf,\Xf)}.$
Then $\left( \Tf+S(\mu)  \right)^{-1}$ exists and 
\begin{align*}
\left\| \left( \Tf+S(\mu)  \right)^{-1} - \left( \Sf+S(\mu)  \right)^{-1}\right\| \le C(\mu) \|\Tf^{-1}-\Sf^{-1}\|_{\Lc(\Wf,\Xf)}.
\end{align*}
\end{Lemma}

\noindent We establish a uniform resolvent estimate in the next proposition and extend the result obtained in Lemma \ref{lem:cvg-res-rc} for the compact set $\mathcal{K}_{\wh\gamma}$ in \eqref{eqcompactset}.

\begin{Proposition}[resolvent estimate on $\mathcal{K}_{\wh\gamma}$] \label{pps-extensionPertb}
Let $(\mathcal{A}_1)$ and $(\mathcal{A}_3)-(\mathcal{A}_5)$ hold. Let $\Af_{\omega,\Ff}$ and $\Af_{\omega_h,\Ff_h}$ be as defined in $(\mathcal{A}_5)$ and \eqref{eqAwpAwhph}, respectively. 
Then for any $\wh\gamma\in(0,\gamma)$ and $\mathcal{K}_{\wh\gamma}$ as in \eqref{eqcompactset}, there exists $h_{\wh\gamma}>0$ such that for all $0<h<h_{\wh\gamma}$ and for some $C>0$ independent of $h$ and $\mu$, the results stated below hold: 
\begin{itemize}
\item[(a)] $\mathcal{K}_{\wh\gamma}\subset \rho(\Af_{\omega_h,\Ff_h})$  and $\displaystyle \n\sup_{\mu\in \mathcal{K}_{\wh\gamma}} \left\|R(\mu,\Af_{\omega,\Ff})-R(\mu,\Af_{\omega_h,\Ff_h})\Pi_h\right\|_{\Lc(\Hf)}\le Ch^s,$ 
where $s\in (0,2]$ is same as in $(\mathcal{A}_4)$,
\item[(b)] $\|R(\mu,\Af_{\omega_h,\Ff_h})\|_{\Lc(\Hf_h)} \le C$ for all $\mu \in \mathcal{K}_{\wh\gamma}.$
\end{itemize}
\end{Proposition}

\begin{proof} 
Applying Lemma \ref{lem:main_genconv}, we establish the proof in the two steps below. In first step, we construct all the required tools needed to apply Lemma \ref{lem:main_genconv} and then the conclusion is made in Step 2.\\[1mm]
\textbf{Step 1.} Fix a real number $\mu_0 >\omega'+\ell.$ Note that $\mu_0 \in \Sigma^c(\omega';\theta_0)\subset \rho(\Af_{\omega,\Ff})\cap_{h>0} \rho(\Af_{\omega_h,\Ff_h})$ and Lemma \ref{lem:cvg-res-rc}(a) yields 
\begin{align}\label{eqn:est for mu_0}
\|R(\mu_0,\Af_{\omega,\Ff})-R(\mu_0, \Af_{\omega_h,\Ff_h})\Pi_h\|_{\Lc(\Hf)} \le C_0 h^s, \quad 0<s\le 2,
\end{align}
for some $C_0>0$ independent of $\mu_0$ and $h.$ 

\noindent Set $\Xf=\Hf\times D(\Af)'.$ Define $(\Sf, D(\Sf))$ (with $D(\Sf)= D(\Af)\times \{0\}$) and for each $h>0,$ define $(\Tf_h, D(\Tf_h))$ (with $D(\Tf_h)=\Hf_h\times \Hf_h^\perp,$ where $\Hf_h^\perp$ is orthogonal complement of $\Hf_h$ in $\Hf$) as 
\begin{align*}
\Sf(\xi,0)=(\mu_0 \If -\Af_{\omega,\Ff}) \xi \text{ and } \Tf_h(\xi_h, \zeta_h)=(\mu_0\If_h-\Af_{\omega_h,\Ff_h})\xi_h+\zeta_h.
\end{align*}  
Since $\mu_0\in  \rho(\Af_{\omega,\Ff})\cap_{h>0} \rho(\Af_{\omega_h,\Ff_h}),$ $\Sf$ and $\Tf_h$ are invertible operators for each $h>0.$ Note that $\Hf_h^\perp$ is orthogonal complement of $\Hf_h$ in $D(\Tf_h)$ and therefore
\begin{align*}
    \Sf^{-1}=\left( R(\mu_0,\Af_{\omega,\Ff}), 0\right)\in \Lc(\Hf, \Xf) \text{ and } \Tf_h^{-1}=\left( R(\mu_0,\Af_{\omega_h,\Ff_h})\Pi_h, (\If-\Pi_h) \right)\in \Lc(\Hf, \Xf).
\end{align*}

\noindent  For all $h>0,$ \eqref{eqn:est for mu_0} and the fact that $\|\If-\Pi_h\|_{\Lc(\Hf, D(\Af)')}=\|\If-\Pi_h\|_{\Lc(D(\Af), \Hf)}\le Ch^2$ (see Lemma \ref{lem:prpty Pi_h}(d)) lead to  
\begin{equation} \label{eqn:est Th-T-1}
\begin{aligned}
\|\Sf^{-1}-\Tf_h^{-1}\|_{\Lc(\Hf, \Xf)}  = \| R(\mu_0,\Af_{\omega,\Ff}) - R(\mu_0, \Af_{\omega_h,\Ff_h})\Pi_h\|_{\Lc(\Hf)} +\|\If-\Pi_h\|_{\Lc(\Hf, D(\Af)')} \le (C_0+1) h^s. 
\end{aligned}
\end{equation}

\noindent For all $\mu\in \mathcal{K}_{\wh\gamma},$ define $S(\mu)\in \Lc(\Xf, \Hf)$ by 
\begin{align*}
S(\mu)(\xi,\zeta)=(\mu-\mu_0) \xi.
\end{align*}

\noindent From $(\mathcal{A}_5)$ and \eqref{eqcompactset}, note that $\mathcal{K}_{\wh\gamma}\subset \overline{\Sigma^c(-\wh\gamma;\phi_0)}\subset \rho(\Af_{\omega,\Ff}).$ Thus for all $\mu\in \mathcal{K}_{\wh\gamma},$ the operator $S(\mu)+\Sf : D(\Sf)\subset \Xf\rightarrow \Hf$ satisfying $(S(\mu)+\Sf)(\xi,0)=(\mu\If-\Af_{\omega,\Ff})\xi$ is invertible and $$(S(\mu)+\Sf)^{-1}=\left( R(\mu,\Af_{\omega,\Ff}), 0\right)\in\Lc(\Hf,\Xf).$$
Our aim is to show the existence of $(\Tf_h+S(\mu))^{-1}$ for all $\mu\in \mathcal{K}_{\wh\gamma}$ by using Lemma \ref{lem:main_genconv} and we verify \eqref{eqn:invT-invS}. For this purpose, define  $C(\mu) = 2(1+\|S(\mu)\|_{\Lc(\Xf,\Wf)})^2\left\|\left( \Sf+S(\mu)  \right)^{-1}\right\|_{\Lc(\Wf,\Xf)},$ 
\begin{align} \label{eq-h_0-C2}
\wh C:=\displaystyle \sup_{\mu\in \mathcal{K}_{\wh\gamma}} \{ (1+C_0)|\mu-\mu_0| C(\mu)\}, \text{ and } h_{\wh\gamma}^s:=\frac{d(\mu_0,\mathcal{K}_{\wh\gamma})}{\wh C}>0.
\end{align}
Therefore, for all $0<h<h_{\wh\gamma}$ and for all $\mu\in\mathcal{K}_{\wh\gamma}$, \eqref{eqn:est Th-T-1}  and \eqref{eq-h_0-C2} lead to 
\begin{equation} \label{eqn:est Th-T}
\begin{aligned}
\|\Tf_h^{-1}-\Sf^{-1}\|_{\Lc(\Hf, \Xf)}\le (C_0+1) h^s < (C_0+1)h_{\wh\gamma}^s =  (C_0+1) \frac{d(\mu_0, \mathcal{K}_{\wh\gamma})}{\wh C} \le (C_0+1) \frac{|\mu-\mu_0|}{\wh C} \le C(\mu)^{-1}.
\end{aligned}
\end{equation}
Hence, all the assumptions in Lemma \ref{lem:main_genconv} are verified and Lemma \ref{lem:main_genconv} leads to the existence of $(\Tf_h+S(\mu))^{-1}$ and 
\begin{align} \label{eqn:invTh+S -S+S}
   \|(\Tf_h+S(\mu))^{-1}-(\Sf+S(\mu))^{-1}\|_{\Lc(\Hf,\Xf)}  \le C(\mu) \|\Tf_h^{-1}-\Sf^{-1}\|_{\Lc(\Xf,\Hf)}  
\end{align}
for all $\mu\in\mathcal{K}_{\wh\gamma}.$\\[1mm]
\noindent \textbf{Step 2.} For each $0<h<h_{\wh\gamma}$ and for all $\mu\in \mathcal{K}_{\wh\gamma},$ note that $\Tf_h+S(\mu):D(\Tf_h)\subset \Xf\rightarrow\Hf$ is such that $(\Tf_h+S(\mu))(\xi_h,\zeta_h)= (\mu\If_h-\Af_{\omega_h,\Ff_h})\xi_h+\zeta_h$ and $D(\Tf_h)=\Hf_h\times \Hf_h^\perp.$ Therefore, for all $\mu\in \mathcal{K}_{\wh\gamma}$ and for all $0<h<h_{\wh\gamma},$ $(\mu\If_h-\Af_{\omega_h,\Ff_h})$ is invertible and hence $\mathcal{K}_{\wh\gamma} \subset \rho(\Af_{\omega_h,\Ff_h}).$ Furthermore, $$(\Tf_h+S(\mu))^{-1}= \left( (\mu\If_h-\Af_{\omega_h,\Ff_h})^{-1}\Pi_h, \If-\Pi_h\right) \in \Lc(\Hf, \Xf).$$ 

\noindent Also, for all $\mu\in \mathcal{K}_{\wh\gamma},$ \eqref{eqn:invTh+S -S+S}, \eqref{eqn:est Th-T}, and \eqref{eq-h_0-C2} yield
\begin{equation} \label{eqn:Th+S^-1-T+s^-1}
\begin{aligned}
\|R(\mu,\Af_{\omega,\Ff}) - R(\mu,\Af_{\omega_h,\Ff_h})\Pi_h\|_{\Lc(\Hf)}  +\|\If-\Pi_h\|_{\Lc(\Hf,D(\Af)')} & = \|(\Tf_h+S(\mu))^{-1}-(\Sf+S(\mu))^{-1}\|_{\Lc(\Hf,\Xf)} \\
%  &  \le C(\mu) \|\Tf_h^{-1}-\Sf^{-1}\|_{\Lc(\Xf,\Hf)} 
& \le C(\mu)(C_0+1)h^s \le \frac{\wh C}{|\mu-\mu_0|}h^s\le  C h^s,
\end{aligned}
\end{equation}
where $ C=\displaystyle \sup_{\mu\in \mathcal{K}_{\wh\gamma}}\frac{\wh C}{|\mu-\mu_0|}<\infty$ is independent of $h$ and $\mu,$ as $\mu_0\not\in\mathcal{K}_{\wh\gamma}.$ 
%  obtain 
% \begin{align}\label{eqn:R(m,AwK)-R(mAwhKh)}
% \sup_{\mu\in \mathcal{K}_{\wh\gamma}}\|R(\mu,\Af_{\omega,\Ff}) - R(\mu,\Af_{\omega_h,\Ff_h})\Pi_h\|_{\Lc(\Hf)} \le\frac{\wh C}{|\mu-\mu_0|}h^s \le  \wt C h^s.
% \end{align}
The proof of (a) is complete. Now, to prove (b), apply \cite[Theorem 3.15, P-212]{Kato} for $R(\mu,\Af_{\omega,\Ff})$ on the compact set $\mathcal{K}_{\wh\gamma}$ to obtain  
\begin{align*}
\|R(\mu,\Af_{\omega,\Ff})\|_{\Lc(\Hf)} \le C \text{ for all }\mu\in \mathcal{K}_{\wh\gamma}. 
\end{align*}
A combination of this and the estimate in (a)
concludes (b).
\end{proof}

% \begin{Theorem}[Exponential stability]{\n\cite[Theorem 4.4.1.2]{Lasiecka1} }\label{th:ua-us}
% Let $(\mathcal{A}_1)$, $(\mathcal{A}_3-\mathcal{A}_4)$ hold. Let $\Af_{\omega_h,\Ff_h}$ be as defined in \eqref{eqAwpAwhph}. Then, for any given $\epsilon\in (0, \gamma)$, there exists $h_{\epsilon}>0$ such that  hold:
% \begin{itemize}
% \item[(1)] There exists $\phi_0\in (\frac{\pi}{2},\theta_0]$ such that for all $0<h<h_\epsilon,$ $\Sigma^c(-\gamma+2\epsilon;\phi_0)\subset \rho(\Af_{\omega_h,\Ff_h})$ and 
% \begin{align*}
% \|R(\mu,\Af_{\omega_h,\Ff_h})\|_{\Lc(\Hf_h)} \le \frac{C}{|\mu+\gamma-2\epsilon|} \text{ for all }\mu \in   \Sigma^c(-\gamma+2\epsilon;\phi_0).
% \end{align*}
% \item[(2)] For a given $\epsilon\in(0,\gamma/2),$ for all $0<h<h_\epsilon,$ the semigroup $\{e^{t\Af_{\omega_h,\Ff_h}}\}_{t\ge 0}$ generated by $\Af_{\omega_h,\Ff_h}$ satisfies
% \begin{align*}
% \|e^{t\Af_{\omega_h,\Ff_h}}\|_{\Lc(\Hf_h)} \le C e^{(-\gamma+2\epsilon)t} \text{ for all }t>0, 
% \end{align*}
% where $C$ is a positive constant independent of $h.$
% \end{itemize}
% \end{Theorem}

\noindent In the next theorem, the uniform exponential stability of $\{e^{t\Af_{\omega_h,\Ff_h}}\}_{t\ge 0}$ under assumptions as in Proposition \ref{pps-extensionPertb} is established.

\begin{Theorem}[exponential stability] \label{th:ua-us}
Let $(\mathcal{A}_1)$, $(\mathcal{A}_3)-(\mathcal{A}_5)$ hold. Let $\Af_{\omega_h,\Ff_h}$ be as defined in \eqref{eqAwpAwhph} and $\phi_0$ be as introduced in \eqref{eqn:phi_0}. Then, for any given $\wh\gamma\in (0, \gamma)$, there exist $h_{\wh\gamma}>0$ and  $C>0$ independent of $h$ such that for all $0<h<h_{\wh\gamma},$
\begin{itemize}
\item[(a)]  $\Sigma^c(-\wh\gamma;\phi_0)\subset \rho(\Af_{\omega_h,\Ff_h})$ and 
\begin{align*}
\|R(\mu,\Af_{\omega_h,\Ff_h})\|_{\Lc(\Hf_h)} \le \frac{C}{|\mu+\wh\gamma|} \text{ for all }\mu \in   \Sigma^c(-\wh\gamma;\phi_0),\, \mu\neq \wh\gamma,
\end{align*}
\item[(b)] the semigroup $\{e^{t\Af_{\omega_h,\Ff_h}}\}_{t\ge 0}$ generated by $\Af_{\omega_h,\Ff_h}$ satisfies
\begin{align*}
\|e^{t\Af_{\omega_h,\Ff_h}}\|_{\Lc(\Hf_h)} \le C e^{-\wh\gamma t} \text{ for all }t>0.
\end{align*}
\end{itemize}
\end{Theorem}

\begin{proof}
(a) Note that $\Sigma^c(-\wh\gamma;\phi_0) \subset \mathcal{K}_{\wh\gamma}\cup \Sigma^c(\omega'+\ell;\wt\theta_0)$ and $\Sigma^c(\omega'+\ell;\wt\theta_0)\subset \Sigma^c(\omega';\theta_0).$ This, Proposition \ref{pps-extensionPertb} and Lemma \ref{lem:uah}(a) lead to 
\begin{align*}
\Sigma^c(-\wh\gamma;\phi_0)\subset \rho(\Af_{\omega_h,\Ff_h}) \text{ for all }0<h<h_{\wh\gamma}.
\end{align*}
Now, observe that  for all $\mu\in \Sigma^c(\omega'+\ell;\wt\theta_0),$
\begin{align*}
\sup_{\mu\in\Sigma^c(\omega'+\ell;\theta_0)} \frac{|\mu+\wh\gamma|}{|\mu-\omega'|} < C_\ell<\infty,
\end{align*}
for some $C_\ell>0$ independent of $\mu$ and $h.$ Utilizing this and Lemma \ref{lem:uah} with the fact $\Sigma^c(\omega'+\ell;\wt\theta_0)\subset \Sigma^c(\omega';\theta_0),$ we have
\begin{align} \label{eqn:resolest outside}
\|R(\mu,\Af_{\omega_h,\Ff_h})\|_{\Lc(\Hf_h)} \le \frac{C}{|\mu-\omega'|} \le \frac{C}{|\mu+\wh\gamma|}\frac{|\mu+\wh\gamma|}{|\mu-\omega'|} \le \frac{CC_\ell}{|\mu+\wh\gamma|} \text{ for all } \mu \in \Sigma^c(\omega'+\ell;\wt\theta_0).
\end{align}
On $\mathcal{K}_{\wh\gamma},$ we have 
\begin{align} \label{eqn:resolest insidecopact}
\|R(\mu,\Af_{\omega_h,\Ff_h})\|_{\Lc(\Hf_h)} \le C =\frac{\wh C}{|\mu+\wh\gamma|} \text{ for all }\mu(\neq -\wh\gamma)\in \mathcal{K}_{\wh\gamma},
\end{align}
where $\displaystyle\wh C=C\sup_{\mu\in\mathcal{K}_{\wh\gamma}}|\mu+\wh\gamma|>0$ is independent of $\mu$ and $h.$ Combine \eqref{eqn:resolest outside} and \eqref{eqn:resolest insidecopact} to obtain 
\begin{align*}
\|R(\mu,\Af_{\omega_h,\Ff_h})\|_{\Lc(\Hf_h)} \le \frac{ C}{|\mu+\wh\gamma|} \text{ for all }\mu(\neq -\wh\gamma)\in \Sigma^c(-\wh\gamma;\theta_0),
\end{align*}
for some positive $C$ independent of $\mu$ and $h.$ This concludes (a).\\[1mm]
(b) Since, the constants appearing in (a) are independent of $h,$ an analogous argument to Theorem \ref{thm:ua-sem}(c) leads to (b).
\end{proof}

\noindent In the next theorem, we establish the converse of Proposition \ref{pps-extensionPertb}. That is, if the discrete system \eqref{eqn:d-state-2} is uniformly stabilizable by a feedback control operator $\Ff_h$ that satisfies $(\mathcal{A}_3),$ then the continuous system \eqref{eqn: main shifted_PCE} is also stabilizable by the feedback operator $\Ff_h\Pi_h.$ 

\begin{Theorem}[intermediate stability] \label{th:unif_stab_cnvrs}
Let $\Ff_h\in \Lc(\Hf_h,\Uf)$ satisfying $(\mathcal{A}_3)$ and $\Af_{\omega_h,\Ff_h}$ as in \eqref{eqAwpAwhph} be such that the semigroup $\{e^{t\Af_{\omega_h,\Ff_h}}\}_{t\ge 0}$ is uniformly stable, that is, there exist positive constants $M_S,\omega_S$ (both independent of $h$) such that
\begin{align*}
\|e^{t\Af_{\omega_h,\Ff_h}}\|_{\Lc(\Hf_h)} \le M_S e^{-\omega_S t} \text{ for all }t>0, \, h>0.
\end{align*}
For each $h>0,$ let $\wt\Ff_h=\Ff_h\Pi_h.$ Then for any given $\wh\omega_P\in(0,\omega_S),$ there exist $h_{\wh\omega_P}>0,$ $\phi_0'\in (\frac{\pi}{2},\pi)$ and $C>0$ independent of $h$ such that for all $0<h<h_{\wh\omega_P},$
\begin{enumerate}
\item[(a)] $\displaystyle \sup_{\mu\in \Sigma^c(-\wh\omega_P;\phi_0')}\|R(\mu,\Af_{\omega,\wt\Ff_h})-R(\mu,\Af_{\omega_h,\Ff_h})\Pi_h\|_{\Lc(\Hf)}\le Ch^2,$
\item[(b)] $\Sigma^c(-\wh\omega_P;\phi_0')\subset \rho(\Af_{\omega,\wt\Ff_h})$ and 
\begin{align*}
\|R(\mu,\Af_{\omega,\wt\Ff_h})\|_{\Lc(\Hf)} \le \frac{C}{|\mu+\wh\omega_P|} \text{ for all }\mu \in   \Sigma^c(-\wh\omega_P;\phi_0'),\, \mu \neq -\wh\omega_P,
\end{align*}
\item[(c)] the semigroup $\{e^{t\Af_{\omega,\wt\Ff_h}}\}_{t\ge 0}$ generated by $\Af_{\omega,\wt\Ff_h}$ satisfies
\begin{align*}
\|e^{t\Af_{\omega,\wt\Ff_h}}\|_{\Lc(\Hf)} \le C e^{-\wh\omega_Pt} \text{ for all }t>0. 
\end{align*}
\end{enumerate}
\end{Theorem}

\begin{proof}
To prove (a) and (b), note that there exists $\theta_P'\in(\frac{\pi}{2}, \pi)$ such that $\sigma(\Af_{\omega_h,\Ff_h})\subset \Sigma(-\omega_S;\theta_P')$ for all $h>0$ and then $\phi_0'$ similar to $\phi_0$ is constructed as in \eqref{eqn:phi_0}. Then, fixing a $\ell>0,$ we observe $\Sigma^c(-\wh\omega_P;\phi_0')\subset \mathcal{K}_{\wh\omega_P}\cup \Sigma^c(\omega'+\ell;\wt\theta_0),$ where $\wt\theta_0:=\min\{\theta_0,\theta_P'\}$ and $\mathcal{K}_{\wh\omega_P}:=\overline{\Sigma(\omega'+\ell;\wt\theta_0)\cap \Sigma^c(-\wh\omega_P,\phi_0')}.$ An analogous proof to Proposition \ref{pps-extensionPertb} leads to the required estimates on  the compact set $\mathcal{K}_{\wh\omega_P}$ and then an analogous argument to Theorem \ref{th:ua-us} concludes the proof. The detailed proof is skipped and here, we just highlight the main changes needed.

\medskip\noindent Since, $\Ff_h$ satisfies $(\mathcal{A}_3)$ and $\wt\Ff_h=\Ff_h\Pi_h,$ $\wt\Ff_h$ satisfies $(\mathcal{A}_2)$. Thus, from Lemma \ref{lem:uah-AwKh}, we have
\begin{align*}
\|R(\mu,\Af_{\omega,\wt\Ff_h})\|_{\Lc(\Hf)} \le \frac{C}{|\mu-\omega'|} \text{ for all }\mu\in \Sigma^c(\omega'+\ell;\wt\theta_0)\subset \Sigma^c(\omega';\theta_0).
\end{align*}
Also, for such $\Ff_h$ and $\wt\Ff_h,$ Lemma \ref{lem:cvg-res-rc} implies 
\begin{align*}
\sup_{\mu \in \Sigma^c(\omega';\theta_0)}\|R(\mu,\Af_{\omega,\wt\Ff_h})-R(\mu,\Af_{\omega_h,\Ff_h})\Pi_h\|_{\Lc(\Hf)} \le Ch^2.
\end{align*} 
For each fixed $h,$ define $(\Sf_h,D(\Sf_h))$ with $D(\Sf_h)=D(\Af)\times \{0\}$ by 
\begin{align*}
\Sf_h(\xi,0)=(\mu_0\If-\Af_{\omega,\wt\Ff_h})\xi
\end{align*}
and $\Tf_h$ as in Proposition \ref{pps-extensionPertb}. Other parameters remain the same as in the proof of Proposition \ref{pps-extensionPertb}. Then for all $0<h<h_{\wh\omega_P},$  $\|\Sf_h^{-1}-\Tf_h^{-1}\|_{\Lc(\Hf,\Xf)}\le C(\mu)^{-1}$ can be obtained as in Proposition \ref{pps-extensionPertb}, where $C(\mu)$ is as in Lemma \ref{lem:main_genconv}. Now, proceed in the same line as in the proof of Proposition \ref{pps-extensionPertb} by replacing $\Sf$ by $\Sf_h$ and $\mathcal{K}_{\wh\gamma}$ by $\mathcal{K}_{\wh\omega_P}$ to obtain the required result.

\medskip\noindent An analogous proof to Theorem \ref{thm:ua-sem}(c) using (b) leads to (c). 
\end{proof}

\section{Stabilization of the Approximate system and discrete Riccati}\label{sec:disRiccati}	
\noindent Theorem \ref{th:stb cnt} shows that for any $\omega>0$, $(\Af_\omega, \Bf)$ is feedback stabilizable with exponential decay $-\gamma<0$ and the feedback control is obtained using the solution of the algebraic Riccati equation \eqref{eqn:ARE}. Let $(\Af_{\omega_h}, \Bf_h)$, the approximate operators be as introduced in Section \ref{sec:Appopp}.  In this section, we study the feedback stabilizability of $(\Af_{\omega_h}, \Bf_h)$ by solving the corresponding discrete algebraic Riccati equation. 

\medskip\noindent Consider the finite dimensional system:
\begin{align} \label{eqn:d-state-2_again}
			\Yt'{_h}(t)=\Af_{\omega_h}\Yt{_h}(t)+\Bf_h\ut_h(t)\text{ for all }t>0, \quad \Yt{_h}(0)=\Yf_{0_h},
		\end{align}
where $\Af_{\omega_h}$ and $\Bf_h$ are as defined in \eqref{eqn:d-state-2} and \eqref{eqn:app B}, respectively. Note that, $\Bf_h^*\in \Lc(\Hf_h,\Uf)$, the adjoint of $\Bf_h,$ is defined by
$\Bf_h^* \begin{pmatrix}
				\varphi_h \\ \varkappa_h
			\end{pmatrix}= \varphi_h \chi_{\mathcal{O}}$ for all $\begin{pmatrix}
				\varphi_h \\ \varkappa_h
			\end{pmatrix}\in \Hf_h$, and 
			\begin{equation} \label{B_h-B_h^*-norm}
			\|\Bf_h^*\|_{\Lc(\Hf_h,\Uf)}=\|\Bf_h\|_{\Lc(\Uf,\Hf)} \le \|\Bf\|_{\Lc(\Uf,\Hf)}.
			\end{equation}

%\subsection{Structure of the proof of Theorem \ref{th:dro}}\label{sec:proof of disc Riccati}
\noindent
Note that for each $h>0,$ since the pair $(\Af_{\omega_h},\Bf_h)$ is finite-dimensional, Kalman rank condition or Hautus type of condition can be used to check the stabilizability of  $(\Af_{\omega_h},\Bf_h)$ on $\Hf_h$. But our aim is to obtain uniform stabilizability by finite dimensional feedback control that is stated in Theorem \ref{th:dro}. To prove that, we need the next lemma.
\begin{Lemma}[intermediate stability - II]\label{lem:temp-1}
Let $\{e^{t\Af_{\omega,\Pf}}\}_{t\ge 0}$, the semigroup generated by $\n \Af_{\omega,\Pf}=\Af_{\omega}-\Bf\Bf^*\Pf,$ be exponentially stable semigroup in $\n\Hf$ with decay $-\gamma<0$ as obtained in Theorem \ref{th:stb cnt}. Then for any $\widehat{\gamma} \in (0,\gamma),$ there exists a $h_{\wh\gamma}>0$ such that for all $0<h<h_{\wh\gamma}$, $\n\Af_{\omega_h,\Pf}:=\Af_{\omega_h}-\Bf_h\Bf^*\Pf$ generates a uniformly analytic semigroup $\n\{e^{t\Af_{\omega_h,\Pf}}\}_{t\ge 0}$  on $\n\Hf_h$ with the exponential decay $-\wh\gamma$, that is, 
$$\n\left\|e^{t\Af_{\omega_h,\Pf}}\right\|_{\Lc(\Hf_h)}\le \widehat{M}e^{-\wh\gamma t}  \text{ for all } t>0, \text{ for all } 0<h<h_{\wh\gamma},$$
for some positive constant $\widehat{M}$ independent of $h$. 
\end{Lemma}
\begin{proof}
Let $\n\Pf\in \mathcal{L}(\Hf)$ be the solution of \eqref{eqn:ARE}. Choose $\Ff=-\Bf^*\Pf$ defined on $\Hf$, $\Ff_h=-\Bf^*\Pf{|_{\Hf_h}}$ restricted onto $\Hf_h$ for all $h>0.$ Note that $\Ff\in \Lc(\Hf,\Uf)$ and $\Ff_h\in \Lc(\Hf_h,\Uf)$ are uniformly bounded as $\|\Ff_h\|_{\Lc(\Hf_h,\Uf)}\le \|\Bf^*\Pf\|_{\Lc(\Hf,\Uf)}\le C_3$ for some positive constant $C_3$ independent of $h.$ Therefore, Lemma \ref{lem:uah} yields the uniform analyticity of the semigroup generated by $\n\Af_{\omega_h,\Pf}:=\Af_{\omega_h}-\Bf_h\Bf^*\Pf$. Since $\Ff=\Ff_h$ on $\Hf_h,$ $(\mathcal{A}_4)$ is satisfied for such $\Ff$ and $\Ff_h.$ Finally, Theorem \ref{th:ua-us} concludes the proof.
\end{proof}

% \noindent Note that the stabilizing control in the above lemma is in feedback form given by $\ut_h(t)=-\Bf^*\Pf\Yt_h(t)$, where $\Pf$ is solution of the Riccati equation \eqref{eqn:ARE} which is infinite dimensional. To obtain the stabilizing control through a finite dimensional Riccati equation, need Theorem \ref{th:dro} stated in Section \ref{sec:mainresults-st} and the proof is given below.\\[1mm]

\noindent \textbf{Proof of Theorem \ref{th:dro}.} 
(a) and (b). Choosing any $\wh\gamma \in (0,\gamma),$ from Lemma \ref{lem:temp-1},  it follows that there exists $h_0:=h_{\wh\gamma}>0$ such that for all $0<h<h_0$, $\n(\Af_{\omega_h},\Bf_h)$ is exponentially stabilizable on $\Hf_h$. Hence (a) and (b) of Theorem \ref{th:dro} follow from \cite[Theorem 3.1, Remark 3.1, Corollary 4.2, part-V, Ch-1]{BDDM}.

\medskip\noindent
(c). First we show that there exists a positive constant $\widetilde{C}$ independent of $h$ such that 
\begin{equation}\label{equnifbddRiccati}
\normalfont \|\Pf_h\|_{\Lc(\Hf_h)} \leq \widetilde{C},\quad 
\normalfont\|\Bf^*_h\Pf_h\|_{\Lc(\Hf_h, U)}\le \widetilde{C} \text{ for all } 0<h<h_0. 
\end{equation}
Note that from \eqref{eq:minvalfun-d}, we have
\begin{align} \label{eq:comp J_h}
\langle \Pf_h\Yf_{0_h},\Yf_{0_h}\rangle =  J_h(\Yf^\sharp_h,u^\sharp_h) < J_h(\widehat{\Yf}_h,\widehat{u}_h),
\end{align}
where 
\begin{align}\label{eqn:whYwhu}
\widehat{\Yf}_h(t)=e^{t\Af_{\omega_h,\Pf}}\Yf_{0_h}  \text{ and }\widehat{u}_h(t)=-\Bf_h\Bf^*\Pf e^{t\Af_{\omega_h,\Pf}}\Yf_{0_h} \text{ for all }t>0, 
\end{align}
and $\Af_{\omega_h,\Pf}$ is as introduced in Lemma \ref{lem:temp-1}. 
Then from Lemma \ref{lem:temp-1} and \eqref{B_h-B_h^*-norm}, it follows that there exists a constant $\wt C>0$ independent of $h$ such that 
\begin{equation} \label{eqn:bdd Ph}
\begin{aligned}
 J_h(\widehat{\Yf}_h,\widehat{u}_h)=\int_0^\infty \Big(\|e^{t\Af_{\omega_h,\Pf}}\textbf{Y}_{0_h}\|^2+\|-\Bf_h\Bf^*\Pf e^{t\Af_{\omega_h,\Pf}}\textbf{Y}_{0_h}\|_{\Uf}^2\Big)\,dt \le \widetilde{C} \|\textbf{Y}_{0_h}\|^2.
\end{aligned}
\end{equation}
Since $\Pf_h\in\Lc(\Hf_h)$ is self adjoint, $\displaystyle \n \|\Pf_h\|_{\Lc(\Hf_h)}=\sup_{\{\Yf_{0_h}\in \Hf_h,\,\|\Yf_{0_h}\|=1\}} \langle \Pf_h\Yf_{0_h},\Yf_{0_h}\rangle\le \widetilde{C}.$ The second estimate in \eqref{equnifbddRiccati} follows from the above estimate and \eqref{B_h-B_h^*-norm}.\\[1mm]
Choose $\Ff_h=-\Bf_h^*\Pf_h$ in Lemma \ref{lem:uah} and use \eqref{equnifbddRiccati} to conclude that $\Af_{\omega_h,\Pf_h}$ generates a uniformly analytic semigroup $\{e^{t\Af_{\omega_h,\Pf_h}}\}_{t\ge 0}$ on $\Hf_h$. 

\noindent
The uniform exponential stability estimate 
is a version of a well-known theorem of Datko (\cite[Chap. 4, Theorem 4.1, p. 116]{Pazy}) for a family of semigroups depending on the parameter $h$.
Utilize \eqref{equnifbddRiccati} and apply Lemma \ref{lem:uah} with $\Ff_h:=-\Bf_h^*\Pf_h$ to obtain $$\|e^{t\Af{_{\omega_h,\Pf_h}}}\|_{\Lc(\Hf_h)} \le C_1 e^{\omega' t} \text{ for all }t>0, \text{ for all } 0<h< h_0,$$ 
for some $C_1>0$ and $\omega'>0$ independent of $h$. Furthermore, from \eqref{eqn:bdd Ph} and Theorem \ref{th:dro}(b), for any $\Yf_{0_h}\in \Hf_h$, there exists $\wt C$ independent of $h$ such that 
$$\int_0^\infty \|e^{t\Af{_{\omega_h,\Pf_h}}}\textbf{Y}_{0_h}\|^2\,dt\le \widetilde{C} \|\textbf{Y}_{0_h}\|^2 \text{ for all } 0<h<h_0.$$
Since, the positive constants $C_1, \omega'$ and $\wt C$ are independent of the parameter $h$ for all $0<h<h_0,$ the assumptions of \cite[Theorem 4A.2]{Lasiecka1} are verified and hence we obtain positive constants $M_p, \omega_P$ independent of $h$ such that 
\begin{align*}
\left\|e^{t\Af{_{\omega_h,\Pf_h}}}\right\|_{\Lc(\Hf_h)} \le M_Pe^{-\omega_P t} \text{ for all } t>0 \text{ and for all } 0<h<h_0.
\end{align*}
This concludes the proof of (c).
\qed

\section{Error estimates} \label{sec:results}
\noindent In this section, we prove the error estimates of solutions of algebraic Riccati equations, stabilized solutions, and stabilizing controls. 

\subsection{Proof of Theorems \ref{th:main-conv-P} - \ref{th:main-conv-new}} \label{subsec:Proof of main 3}

\begin{Lemma}[intermediate stability - III] \label{lem:fo}
Let $h_0>0$ and $\omega_P>0$ be as obtained in Theorem \ref{th:dro}. Let $\aph:=\Aw-\Bf\Bf_h^*\Pf_h\Pi_h,$ where $\Pf_h\in \Lc(\Hf_h)$ is the solution of \eqref{eqn:d-ARE_PCE}. Then for any $\wh{\omega}_P\in(0,\omega_P),$ there exists $h_{\wh{\omega}_P}\in (0,h_0)$ such that for all $0<h<h_{\wh{\omega}_P},$  $\aph$ generates a uniformly analytic semigroup $\{e^{\aph}\}_{t\ge 0}$ on $\Hf$ with exponential decay $-\wh{\omega}_P<0,$ that is, 
\begin{align*}
\|e^{t\aph}\|_{\Lc(\Hf)}\le C e^{-\wh{\omega}_P t} \text{ for all }t>0, \text{ for all }0<h<h_{\wh{\omega}_P},
\end{align*}
and for some $C>0$ independent of $ h.$
\end{Lemma}
\begin{proof}
Due to \eqref{equnifbddRiccati} and Lemma \ref{lem:prpty Pi_h}(c), $(\mathcal{A}_2)$ is satisfied for $\wt\Ff_h=-\Bf_h^*\Pf_h\Pi_h.$ Therefore,  Lemma \ref{lem:uah-AwKh} implies that $\aph$ generates a uniformly analytic semigroup $\{e^{t\aph}\}_{t\ge 0}$ on $\Hf.$ Utilizing \eqref{equnifbddRiccati}, the assumptions in Theorem \ref{th:unif_stab_cnvrs} are satisfied with $\Ff_h=-\Bf_h^*\Pf_h.$ Thus for the choices $\omega_S=\omega_P$ and $\wt\Ff_h=\Ff_h\Pi_h=-\Bf_h^*\Pf_h\Pi_h,$ Theorem \ref{th:unif_stab_cnvrs}(c) leads to the desired result.
% Being $(\Yt(t),\ut(t))$ unique minimizer of the optimal control problem \ref{eqoptinf_PCE} (see Theorem \ref{th:stb cnt}), have \eqref{eq:comp_J}.
\end{proof}

\noindent Our next aim is to prove the estimate between the discrete Riccati and the continuous Riccati solution, that is, to prove Theorem \ref{th:main-conv-P}. We state and prove an auxiliary lemma needed for this.
\begin{Lemma}[intermediate estimate] \label{lem:fouropt}
Let $\ap,$ $\ahp,$ $\aph,$ and $\ahph$ be as in Theorem \ref{th:stb cnt}, Lemmas \ref{lem:temp-1}, \ref{lem:fo}, and Theorem \ref{th:dro}, respectively. Let $\gamma,$ $h_0,$ and $\omega_P$ be as in Theorems \ref{th:stb cnt} and \ref{th:dro}, respectively. Then for any $\wt\gamma$ satisfying $0<\wt\gamma<\min\{\gamma, \omega_P\},$ there exists $\wt h_0 \in (0,h_0)$  such that
\begin{align*}
\left\|e^{t\ahp}\Pi_h-e^{t\ap}\right\|_{\Lc(\Hf)} + \left\|e^{t\ahph}\Pi_h-e^{t\aph}\right\|_{\Lc(\Hf)} \le Ch^2 \frac{e^{-\wt\gamma t}}{t} \text{ for all }t>0,\, 0<h<\wt h_0,
\end{align*}
and for some $C>0$ independent of $h.$
\end{Lemma}

\begin{figure}[ht!] 
\begin{center}
\begin{tikzpicture}
			\path[fill=gray!30] (-6,4.5)--(-2.5,2.9)--(-1,0)--(-2.5,-2.9)--(-6,-4.5)--cycle;
			\draw[thick,<->,dashed] (-6,0)--(6,0);
			\draw[thick,<->,dashed] (0,-4)--(0,4);

			\draw[orange] (-2.5,-2.9)--(-2.5,2.9);

			\draw[thick,->] (-2.5,2.9)--(-4,3.5);
			\draw[thick,,->] (-2.5,-2.9)--(-4,-3.5);
			
			\draw[thick,dashed] (4.5,0)--(-2.5,2.9);
			\draw[thick,dashed] (4.5,0)--(-2.5,-2.9);
			
			\node[] at (4.5,-0.2){$\omega'$};
			\draw[] (5,0) arc (0:150:0.5cm);
			\node[] at (5.1,0.4){$\theta_0$};
			
			\node[red] at (-2.2,0.3){$-\gamma$};
			\node[red] at (-2.5,0){$\bullet$};
			\node[blue] at (-0.8,0.2){$-\wh\gamma$};
			\node[blue] at (-1,0){$\bullet$};
			
			\node[] at (0.2,0.2){O};
			% \node[] at (-2.5,3.1){$P$};
			% \node[] at (-2.5,-3.1){$Q$};
			% \node[] at (-0.8,-0.2){$R$};
			
			\draw[red,>->] (-1.8,-4)--(-1.2,-1.2);
			\draw[red,<-<] (-1.8,4)--(-1.2,1.2);
			\draw[red] (-1.2,-1.2) arc (-110:110:1.26cm);
			\draw[red,dashed,<->] (1.5,0) arc (0:100:2.5cm);
			\node[red] at (1,2){$\phi'$};
			
			\node[red] at (-2,-3.5){$\Gamma_-$};
			\node[red] at (-2,3.5){$\Gamma_+$};
			\node[red] at (0.7,-0.2){$\Gamma_0$};
			
			\draw[blue,dashed] (-1,0)--(-2.5,-2.9);
			\draw[blue,dashed] (-1,0)--(-2.5,2.9);
			
			\draw[blue,dashed] (-1,0)--(-0.7,-1.2);
			\node[blue] at (-1,-0.8){$r_0$};

%			\draw[blue,dashed] (-0.2,0) arc(0:120:0.8cm);
%			\node[blue] at (0.6,0.7){$\phi_0$};
			
			\node[] at (-3.5,1.2){Spectrum};
			\node[] at (-3.5,0.85){region for $\Af{_{\omega_h,\Pf}}$};
\end{tikzpicture}
\end{center}
\caption{Spectrum region for $\Af{_{\omega_h,\Pf}}$ and $\color{red}{\Gamma=\Gamma_\pm\cup\Gamma_0}$} \label{fig:main-Gamma}
\end{figure}

\begin{proof}
Let $0<\wt\gamma<\min\{\gamma,\omega_P\}$ be any given number. Fix $(\wh\gamma,h_{\wh\gamma})$ and $(\wh\omega_P,h_{\wh\omega_P})$ as in Lemmas \ref{lem:temp-1} and \ref{lem:fo}, respectively, such that $0<\wt\gamma \le \min\{\wh\gamma, \wh\omega_P\}$. Let $\wt h_0:=\min\{ h_0, h_{\wh\gamma}, h_{\wh\omega_P}\}.$ We first estimate $\left\|e^{t\ahp}\Pi_h-e^{t\ap}\right\|_{\Lc(\Hf)}.$ Now, choosing $\Ff=-\Bf^*\Pf,$ $\Ff_h$ as restriction of $-\Bf^*\Pf$ on $\Hf_h,$ the assumptions of Proposition \ref{pps-extensionPertb}(a) and Theorem \ref{th:ua-us} are satisfied. Thus, Proposition \ref{pps-extensionPertb} leads to 
\begin{align*}
\sup_{\mu\in \Sigma^c(-\wh\gamma;\phi_0)}\|R(\mu,\ap)-R(\mu,\ahp)\Pi_h\|_{\Lc(\Hf)}\le Ch^2,
\end{align*}
where $\phi_0$ is as mentioned in \eqref{eqn:phi_0}. Utilize this, choose $\Gamma=\Gamma_\pm\cup \Gamma_0$ (see Figure \ref{fig:main-Gamma}), where $\Gamma_\pm=\{-\wh\gamma +re^{\pm i\phi'}\, |\, r\ge r_0\}$ and $\Gamma_0=\{-\wh\gamma+r_0e^{i\phi}\,|\, |\phi|\le \phi'\},$ for some $r_0>0$ and $\frac{\pi}{2}<\phi'<\phi_0,$ and proceed as in Theorem \ref{th:res-sg-est} to obtain
\begin{align*}
\left\|e^{t\ahp}\Pi_h-e^{t\ap}\right\|_{\Lc(\Hf)} \le \sup_{\mu\in\Gamma}\|R(\mu,\ap)-R(\mu,\ahp)\Pi_h\|_{\Lc(\Hf)}\int_\Gamma e^{\mu t}\,d\mu\le Ch^2\frac{e^{-\wh\gamma t}}{t}\le Ch^2\frac{e^{-\wt\gamma t}}{t},
\end{align*}
for all $0<h<\wt h_0.$\\[1mm]
To estimate the second term, that is, $\left\|e^{t\ahph}\Pi_h-e^{t\aph}\right\|_{\Lc(\Hf)},$ note that the assumptions of Theorem \ref{th:unif_stab_cnvrs} are satisfied with $\Ff_h=-\Bf_h^*\Pf_h$ thanks to \eqref{equnifbddRiccati}. Therefore, for  $\omega_S=\omega_P$, Theorem \ref{th:unif_stab_cnvrs}, and a similar argument as above lead to 
\begin{align*}
\left\|e^{t\aph} - e^{t\ahph}\Pi_h\right\|_{\Lc(\Hf)} \le \sup_{\mu\in\Gamma}\|R(\mu,\aph)-R(\mu,\ahph)\Pi_h\|_{\Lc(\Hf)}\int_\Gamma e^{\mu t}\,d\mu\le Ch^2\frac{e^{-\wh\omega_P t}}{t}\le Ch^2\frac{e^{-\wt{\gamma} t}}{t},
\end{align*}
for all $0<h<\wt h_0.$ The proof is complete.
\end{proof}

\noindent For our later analysis, set 
\begin{equation} \label{eqn:Yht-uht}
    \begin{array}{ll}
    \Yf^\sharp(t)=e^{t\ap}\Yf_0, & u^\sharp(t)=-\Bf^*\Pf e^{t\ap}\Yf_0, \\
    \Yf^\sharp_h(t)=e^{t\ahph}\Pi_h\Yf_0, & u^\sharp_h(t)=-\Bf_h^*\Pf_h e^{t\ahph}\Pi_h\Yf_0, \\
    \widehat{\Yf}_h(t)=e^{t\Af_{\omega_h,\Pf}}\Pi_h\Yf_0 , & \widehat{u}_h(t)=-\Bf_h\Bf^*\Pf e^{t\Af_{\omega_h,\Pf}}\Pi_h\Yf_0,\\
    \overline{\Yf}(t)=e^{t\Af_{\omega,\Pf_h}}\Yf_0, & \overline{u}(t)=-\Bf_h^*\Pf_he^{t\aph}\Yf_0.\\
     \end{array}
\end{equation}
Observe that $(\Yf^\sharp(t),u^\sharp(t)),$  $(\Yf^\sharp_h(t),u^\sharp_h(t)),$ and $(\wh\Yf_h(t),\wh u_h(t))$ are the same as in Theorems \ref{th:stb cnt} - \ref{th:dro}, and \eqref{eqn:whYwhu}, respectively.

\medskip
\noindent Since $(\Yf^\sharp(t),u^\sharp(t))$  is the unique minimizing pair for \eqref{eqoptinf_PCE}, for the pair $(\overline{\Yf}(t), \overline{u}(t))$ as in \eqref{eqn:Yht-uht}, we have 
\begin{align}  \label{eq:comp-J}
J(\Yf^\sharp,u^\sharp)<J(\overline{\Yf},\overline{u}).
\end{align}
Now, if $J(\Yf^\sharp,u^\sharp)>J_h(\Yf^\sharp_h,u^\sharp_h),$ \eqref{eq:comp-J} yields
\begin{equation}\label{fun-comp-11}
\begin{aligned} 
 0<J(\Yf^\sharp,u^\sharp)-J_h(\Yf^\sharp_h,u^\sharp_h)\le J(\overline{\Yf},\overline{u})-J_h(\Yf^\sharp_h,u^\sharp_h)=\vert J(\overline{\Yf},\overline{u})-J_h(\Yf^\sharp_h,u^\sharp_h)\vert ,
\end{aligned}
\end{equation}
and if $J(\Yf^\sharp,u^\sharp)<J_h(\Yf^\sharp_h,u^\sharp_h)$, then \eqref{eq:comp J_h} implies 
\begin{equation}\label{fun-comp-22}
\begin{aligned} 
0<J_h(\Yf^\sharp_h,u^\sharp_h)-J(\Yf^\sharp,u^\sharp)\le J_h(\widehat{\Yf}_h,\widehat{u}_h)-J(\Yf^\sharp,u^\sharp)=\vert J_h(\widehat{\Yf}_h,\widehat{u}_h)-J(\Yf^\sharp,u^\sharp) \vert.
\end{aligned}
\end{equation}

\noindent\textit{Proof of Theorem \ref{th:main-conv-P}.}
(a)-(b). From Theorems \ref{th:stb cnt} and \ref{th:dro} with $\Yf_{0_h}=\Pi_h\Yf_0,$ note that 
\begin{equation} \label{eqn:ph-p-3}
\begin{aligned}
\big\vert\langle (\Pf_h\Pi_h-\Pf)\Yf_0,\Yf_0\rangle\big\vert & =\big\vert \langle \Pf_h\Pi_h\Yf_0,\Pi_h\Yf_0\rangle-\langle \Pf\Yf_0,\Yf_0\rangle\big\vert  =\big\vert J_h(\Yf^\sharp_h,u^\sharp_h)-J(\Yf^\sharp,u^\sharp)\big\vert .
\end{aligned}
\end{equation}
Now, \eqref{fun-comp-11} and \eqref{fun-comp-22} imply  $\vert J_h(\Yf^\sharp_h,u^\sharp_h)-J(\Yf^\sharp,u^\sharp)\vert \le \vert J(\overline{\Yf},\overline{u})-J_h(\Yf^\sharp_h,u^\sharp_h)\vert+\vert J_h(\widehat{\Yf}_h,\widehat{u}_h)-J(\Yf^\sharp,u^\sharp) \vert,$ where $(\overline{\Yf},\overline{u})$ and $(\widehat{\Yf}_h,\widehat{u}_h)$ are as in \eqref{eqn:Yht-uht}. The expressions for $J(\cdot,\cdot)$ and $J_h(\cdot,\cdot),$ and a triangle inequality yield 
\begin{align*}
\big\vert\langle (\Pf_h\Pi_h-\Pf)\Yf_0,\Yf_0\rangle\big\vert & \le \int_0^\infty\Big\vert\|\overline{u}(t)\|_\Uf^2-\|u^\sharp_h(t)\|_\Uf^2\Big\vert\, dt  +\int_0^\infty\Big\vert\|\overline{\Yf}(t)\|^2-\|\Yf^\sharp_h(t)\|^2\Big\vert\, dt \\
 & \hspace{1cm} +\int_0^\infty\Big\vert\|\widehat{u}_h(t)\|_\Uf^2-\|u^\sharp(t)\|_\Uf^2\Big\vert \,dt  +\int_0^\infty\Big\vert\|\widehat{\Yf}_h(t)\|^2-\|\Yf^\sharp(t)\|^2\Big\vert\, dt \\
& \le C\left( \int_0^\infty\Big\vert\|\overline{\Yf}(t)\|^2-\|\Yf^\sharp_h(t)\|^2\Big\vert\, dt + \int_0^\infty\Big\vert\|\widehat{\Yf}_h(t)\|^2-\|\Yf^\sharp(t)\|^2\Big\vert \,dt\right),
\end{align*}
where the values of $u^\sharp(t)$, $u^\sharp_h(t),\;\overline{u}(t),$ and $\widehat{u}_h(t)$ (given in \eqref{eqn:Yht-uht}) along with \eqref{B_h-B_h^*-norm} and \eqref{equnifbddRiccati} are utilized in the last inequality. Substitute the values of $\Yf^\sharp(t),\;\Yf^\sharp_h(t),\;\overline{\Yf}(t),$ and $\widehat{\Yf}_h(t)$ from \eqref{eqn:Yht-uht} in the last expression, use the inequality $\big\vert\|a\|^2-\|b\|^2\big\vert \leq \|a-b\|(\|a\|+\|b\|) $ and the H\"{o}lder's inequality to obtain
\begin{equation} \label{eqn:ph-p-120}
\begin{aligned}
 \left\vert\left\langle (\Pf_h\Pi_h-\Pf)\Yf_0,\Yf_0\right\rangle\right\vert 
  &\le C \Bigg(\int_0^\infty\left(\|e^{t\aph}\Yf_0-e^{t\ahph}\Pi_h\Yf_0\|\right) \Big(\|e^{t\aph}\Yf_0\| \\
& \qquad +\|e^{t\ahph}\Pi_h\Yf_0\| \Big)\,dt 
+ \int_0^\infty\left(\|e^{t\ahp}\Pi_h\Yf_0-e^{t\ap}\Yf_0\|\right)\\ 
& \qquad\quad\times\Big(\|e^{t\ahp}\Pi_h\Yf_0\| +\|e^{t\ap}\Yf_0\| \Big)\,dt \Bigg).
\end{aligned}
\end{equation}

\noindent For any $0<\wt\gamma<\min\{\gamma, {\omega_P}\},$ fix $(\wh\gamma,h_{\wh\gamma}),$ $(\wh\omega_P, h_{\wh\omega_P}),$ and $\wt h_0=\min\{h_0, h_{\wh\gamma}, h_{\wh {\omega_P}}\}$ as in Lemma \ref{lem:fouropt}. This and the exponential stability in Lemma \ref{lem:fo}, Theorem \ref{th:dro}, Lemma \ref{lem:temp-1}, and Theorem \ref{th:stb cnt} imply that the right hand side of \eqref{eqn:ph-p-120} is bounded by 
\begin{align}\label{eqn:bound of ph-p-120}
C\|\Yf_0\|^2\Big(\int_0^\infty e^{-\wt\gamma t}\left\|e^{t\aph}-e^{t\ahph}\Pi_h\right\|_{\Lc(\Hf)} dt+\int_0^\infty e^{-\wt\gamma t} \left\|e^{t\ahp}\Pi_h-e^{t\ap}\right\|_{\Lc(\Hf)} dt\Big),
\end{align}
for all $0<h<\wt h_0.$

\noindent Lemmas \ref{lem:fo}, \ref{lem:fouropt}, and Theorem \ref{th:dro} lead to 
\begin{align}\label{eq:est4P-1}
\left\|e^{t\aph}-e^{t\ahph}\Pi_h\right\|_{\Lc(\Hf)}\le C e^{-\wt\gamma t} \text{ and }
\left\|e^{t\aph}-e^{t\ahph}\Pi_h\right\|_{\Lc(\Hf)} \le C h^2 \frac{e^{-\wt\gamma t}}{t}.
\end{align}
Let $0<\epsilon<1$ be arbitrary small number. An interpolation between the inequalities in \eqref{eq:est4P-1} leads to 
\begin{align}\label{eq:est4P-n2}
\left\|e^{t\aph}-e^{t\ahph}\Pi_h\right\|_{\Lc(\Hf)} \le C h^{2(1-\epsilon)} \frac{e^{-\wt\gamma t}}{t^{1-\epsilon}},
\end{align}
which further yields
\begin{align}\label{eq:est4P-n3}
\int_0^\infty e^{-\wt\gamma t} \left\|e^{t\ahp}\Pi_h-e^{t\ap}\right\|_{\Lc(\Hf)} \, dt  \le C h^{2(1-\epsilon)} \text{ for all }0<h<\wt h_0,
\end{align}
where the constant $C>0$ depends on $\wt\gamma$ and $\epsilon$ but is independent $h.$

\noindent To estimate the second term in \eqref{eqn:bound of ph-p-120}, proceed in a similar way utilizing Theorem \ref{th:stb cnt}, Lemmas \ref{lem:temp-1}, and \ref{lem:fouropt} to obtain
\begin{align*}
\int_0^\infty  e^{-\wt\gamma t} \left\|e^{t\ahp}\Pi_h-e^{t\ap}\right\|_{\Lc(\Hf)} \, dt  \le C h^{2(1-\epsilon)} \text{ for all }0<h<\wt h_0.
\end{align*}

\noindent Utilize above two inequalities in \eqref{eqn:ph-p-120} and \eqref{eqn:bound of ph-p-120} to obtain
\begin{equation*}
 \left\vert\left\langle (\Pf_h\Pi_h-\Pf)\Yf_0,\Yf_0\right\rangle\right\vert\le C h^{2(1-\epsilon)} \|\Yf_0\|^2,
\end{equation*}
and thus 
\begin{align*}
\|\Pf_h\Pi_h-\Pf\|_{\Lc(\Hf)}=\sup_{\substack{ \Yf_0\in \Hf \\ \|\Yf_0\|=1}}\left\vert\left\langle (\Pf_h\Pi_h-\Pf)\Yf_0,\Yf_0\right\rangle\right\vert\le C h^{2(1-\epsilon)} \text{ for all }0< h< \wt h_0.
\end{align*}
This completes the proof.

\medskip
\noindent (c) Note that the obtained feedback operators stabilizing \eqref{eqn: main shifted_PCE} and \eqref{eqn:d-state-2} are $\n-\Bf^*\Pf$ and $\n-\Bf_h^*\Pf_h\Pi_h$, respectively. Theorem \ref{th:main-conv-P}(a) and $\normalfont\Pi_h \Bf=\Bf_h$ lead to
\begin{equation*} \label{eqn:b*p}
\normalfont
\begin{aligned}
\normalfont\|\Bf_h^*\Pf_h\Pi_h-\Bf^*\Pf\|_{\Lc(\Hf,\Uf)}  =\|\Pf_h\Bf_h-\Pf\Bf\|_{\Lc(\Uf,\Hf)}  \le \|\Pf_h\Pi_h-\Pf\|_{\Lc(\Hf)}\|\Bf\|_{\Lc(\Uf,\Hf)}  \le C h^{2(1-\epsilon)} .
\end{aligned}
\end{equation*}
(d) Finally, (a) leads to 
\begin{align*}
\|\Bf^*\Pf-\Bf_h^*\Pf_h\|_{\Lc(\Hf_h,\Uf)} \le \| \Bf^*(\Pf-\Pf_h)\|_{\Lc(\Hf_h,\Uf)}+\| (\Bf^*-\Bf_h^*)\Pf_h\|_{\Lc(\Hf_h,\Uf)} \le Ch^{2(1-\epsilon)}.
\end{align*}\qed

\noindent Now, we prove that the discrete stabilized solution, that is, the solution $\Yf^\sharp_h(\cdot)$ of \eqref{eq-d-cl-loop} converges to the stabilized solution $\Yf^\sharp(\cdot)$ of \eqref{eqcl-loop} and their error estimate. We also establish an error estimate for the stabilizing control.
%

% \begin{Theorem}[Convergence]\label{th:main-conv-new}
% Let $\n\Yt(t)$ (resp. $\n\Yt_h(t)$) be the solution of \eqref{eqcl-loop} (resp. \eqref{eq-d-cl-loop}) and $\n\ut(t)$ (resp. $\n\ut_h(t))$) be as defined in \eqref{eqoptcntrl} (resp. \eqref{eqn:Yht-uht}). Then, it holds that 
% \begin{itemize}
% \item[(a)] $\n\|\Yt_h(t)-\Yt(t)\|_{\Hf}\le C h^{2(1-\epsilon)} \frac{e^{-\frac{\wt\gamma}{2}t}}{t} \|\Yf_0\|_{\Hf}  $ for all $t>0,$  (b) $\n\|\Yt_h(\cdot)-\Yt(\cdot)\|_{L^2(0,\infty;\Hf)}\le C h^{2(1-\epsilon)} ,$
% \item[(c)] $\n\|\ut_h(t)-\ut(t)\|_{\Uf}\le C h^{2(1-\epsilon)}  \frac{e^{-\frac{\wt\gamma}{2}t}}{t} \|\Yf_0\|_{\Hf} $ for all $t>0,$ (d) $\n\|\ut_h(\cdot)-\ut(\cdot)\|_{L^2(0,\infty;\Uf)}\le C h^{2(1-\epsilon)} .$ 
% \end{itemize}
% \end{Theorem}
\medskip\noindent \textbf{Proof of Theorem \ref{th:main-conv-new}.}
% Recall that the stabilized solutions $\wt\Yf(t)$ and $\wt\Yf_h(t)$ are represented by  $\wt\Yf(t)=e^{t\ap}\Yf_0$ and $\wt\Yf_h(t)=e^{t\ahph}\Pi_h\Yf_0,$ respectively, where $\Af_{\omega,\Pf}=\Aw-\Bf\Bf^*\Pf$ and $\Af_{\omega_h,\Pf_h}=\Af_{\omega_h}-\Bf_h\Bf_h^*\Pf.$
For any $0<\wt\gamma<\min\{\gamma, {\omega_P}\},$ as in Lemma \ref{lem:fouropt}, fix $(\wh\gamma,h_{\wh\gamma}),$ $(\wh\omega_P, h_{\wh\omega_P})$ such that $0<\wt\gamma\le \min\{\wh\gamma, \wh\omega_P\}$ and $\wt h_0=\min\{h_0, h_{\wh\gamma}, h_{\wh {\omega_P}}\}$. Utilizing
\eqref{equnifbddRiccati} and Theorem \ref{th:main-conv-P}(c) for $\Ff=-\Bf^*\Pf$ and $\Ff_h=-\Bf_h\Pf_h$, the assumptions in Proposition \ref{pps-extensionPertb} hold. Let $\wt\phi_0:=\min\{\phi_0,\phi_0'\},$ where $\phi_0$ and $\phi_0'$ are as in Theorem \ref{th:ua-us} and Theorem \ref{th:unif_stab_cnvrs}, respectively. Therefore, for any $0<\wt\gamma<\min\{\gamma,{\omega_P}\}$  the spectrum of $\ap$ and $\ahph$ are contained in $\Sigma(-\wt\gamma;\wt\phi_0),$ and  $$\sup_{\mu\in \Sigma^c(-\wt\gamma;\wt\phi_0)}\|R(\mu,\ap)-R(\mu,\ahph)\Pi_h\|_{\Lc(\Hf)}\le C h^{2(1-\epsilon)}\text{ for all } 0<h<\wt h_0.$$   Let $\Gamma=\Gamma_\pm\cup\Gamma_0$ be a path in $\Sigma^c(-\wt\gamma;\wt\phi_0)$, where $\Gamma_\pm=\{-\wt\gamma+re^{\pm i\vartheta_0}\,|\, r\ge r_0\}$ and $\Gamma_0=\{-\wt\gamma+r_0e^{i\phi}\,|\,|\phi|\le \vartheta_0\},$ for some $r_0>0,$ $\frac{\pi}{2}<\vartheta_0<\wt\phi_0$ and utilize
\begin{align*}
\Yf^\sharp(t) - \Yf^\sharp_h(t) = e^{t\ap}\Yf_0 - e^{t\ahph}\Pi_h\Yf_0=\frac{1}{2\pi i}\int_\Gamma e^{\mu t}\left(R(\mu,\ap) -  R(\mu,\ahph)\Pi_h\right)\Yf_0 \,d\mu,
\end{align*}
and proceed as in Theorem \ref{th:res-sg-est} to obtain
\begin{align}\label{eqYh-Y-1}
\|\Yf^\sharp(t) - \Yf^\sharp_h(t)\|\le  C h^{2(1-\epsilon)}\frac{e^{-\wt\gamma t}}{t} \|\Yf_0\| \text{ for all }0< h<\wt h_0, \text{ for all }t>0. 
\end{align}
Hence (a) follows. Taking an interpolation between the above inequality and $\|\Yf^\sharp(t) - \Yf^\sharp_h(t)\|\le Ce^{-\wt\gamma}\|\Yf_0\|,$ followed by integration with respect to $t$ over $(0,\infty)$ leads to (b).\\
 Note that, we have $u^\sharp_h(t)=-\Bf_h^*\Pf_h\Yf^\sharp_h(t)$ and $u^\sharp(t)=-\Bf^*\Pf\Yf^\sharp(t).$ Thus 
\begin{align*}
u^\sharp_h(t)-u^\sharp(t)=-\Bf_h^*\Pf_h\Yf^\sharp_h(t)+\Bf^*\Pf\Yf^\sharp(t)=-(\Bf_h^*\Pf_h-\Bf^*\Pf)\Yf^\sharp_h(t)+\Bf^*\Pf(\Yf^\sharp(t)-\Yf^\sharp_h(t)).
\end{align*}
Utilize Theorem \ref{th:main-conv-P}(d) with Theorem \ref{th:dro}  and the fact that $\wt\gamma<\omega_P$ in the first term above and \eqref{eqYh-Y-1} in the second term to obtain
\begin{align*}
\|u^\sharp_h(t)-u^\sharp(t)\|_{\Uf}\le Ch^{2(1-\epsilon)}e^{-\omega_P t}\|\Yf_0\| +\|\Bf^*\Pf\|_{\Lc(\Hf)}Ch^{2(1-\epsilon)} \frac{e^{-\wt\gamma t}}{t}\|\Yf_0\| \text{ for all }0< h< \wt h_0.
\end{align*}
Noting that $\wt\gamma<\omega_P,$ (c) follows from the above inequality. Taking an interpolation between the above inequality and $\|u^\sharp(t) - u^\sharp_h(t)\|\le Ce^{-\wt\gamma}\|\Yf_0\|,$ followed by integration with respect to $t$ over $(0,\infty)$ leads to (d). \qed

\section{Numerical Results} \label{sec:NI}

\noindent 
We present a numerical example in this section. 
% Recall the coupled system of equations in \eqref{eqn: main shifted_PCE} that seeks $\begin{pmatrix}\widetilde{y}\\ \widetilde{z}\end{pmatrix}$ such that
% \begin{equation}\label{eqlinregu_again}
% \begin{aligned}
% \begin{aligned}
% & \widetilde{y}_t-\eta_0 \Delta \widetilde{y}+\eta_1  \widetilde{z}+(\nu_0-\omega) \widetilde{y}=\widetilde{u}\chi_\mathcal{O} \quad \mathrm{in} \quad \Omega\times (0,\infty),\\
% & \widetilde{z}_t-\beta_0\Delta\widetilde{z} +\kappa \widetilde{z}-\widetilde{y}+(\nu_0-\omega) \widetilde{z}=0 \quad \mathrm{in} \quad \Omega\times (0,\infty),\\
% & \widetilde{y}(x,t)=0, \; \widetilde{z}(x,t)=0\quad \mathrm{in} \quad \partial\Omega\times (0,\infty), \\
% & \widetilde{y}(\cdot,0)=y_0, \; \widetilde{z}(\cdot,0)=z_0 \text{ in }\Omega.
% \end{aligned}
% \end{aligned}
% \end{equation}
%Subsection~\ref{sub:1} describes the data for the examples used in the implementation.
The first subsection  details the space and time discretizations and an implementation procedure for stabilization. Subsection \ref{subs:EOC} describes the error computation methodology, and Subsection~\ref{sub:5} presents the results of the numerical experiment that validate the theoretical results.  % taken in Subsection~\ref{sub:1}. 

%\subsection{Data for the example}\footnote{Do now need this as separate subsection?} \label{sub:1}
% discuss a numerical example by choosing  $\omega=25, \kappa=1,\eta_0=1, \eta_1=5,\,\beta_0=0.8$ and $\nu_0=0$ in \eqref{eqn:d-state-2}.  With these chosen parameter the system without control is unstable.  choose the initial conditions as $y_0(x_1,x_2)=x_1(x_1-1)x_2(x_2-1)$ and $z_0(x_1,x_2)=\sin(\pi x_1)\sin(\pi x_2).$ Though the solution to the system converges in the energy and $L^2$ norms with the expected order of convergences, it is shown that without control, the solution is unstable.  It is further illustrated that the system is stabilizable  by applying a feedback control and the stabilized solution converges with expected rate of convergence in the energy and $L^2$ norms.
\subsection{Discretization and stabilization} \label{sub:2}
Discretize the space $\Omega$ by triangulation $\mathcal{T}_h$ with discretization parameter $h$ and consider $V_h\subset \Hio$.  For all  $\phi_h\in V_h,$ the semi-discrete formulation that corresponds to \eqref{eqn:d-state-2} seeks $\widetilde{y}_h$, $\widetilde{z}_h$, and $\widetilde{u}_h$ such that 
\begin{align*}
\langle \widetilde{y}_h'(t), \phi_h\rangle & =-\eta_0 \langle\nabla \widetilde{y}_h(t),\nabla\phi_h\rangle -\eta_1\langle \widetilde{z}_h(t), \phi_h\rangle+(\omega-\nu_0)\langle \widetilde{y}_h(t),\phi_h\rangle + \langle \widetilde{u}_h , \phi_h \rangle,\\
\langle \widetilde{z}_h'(t),\phi_h \rangle & =-\beta_0\langle\nabla \widetilde{z}_h(t),\nabla\phi_h \rangle +(-\kappa+\omega-\nu_0)\langle \widetilde{z}_h(t),\phi_h\rangle + \langle \widetilde{y}_h(t),\phi_h \rangle, \\
\langle \wt{y}_h(0), \phi_h\rangle &  =\langle y_0, \phi_h \rangle, \quad \langle \wt{z}_h(0), \phi_h\rangle  = \langle z_0, \phi_h \rangle.
\end{align*}
Recall from Section \ref{sec:approx} that  $n_h$ denotes the  dimension of $V_h$. Let  $\widetilde{y}_h(t):=\sum_{i=1}^{n_h}y_i(t)\phi_h^i$ and  $\widetilde{z}_h(t):=\sum_{i=1}^{n_h}z_i(t)\phi_h^i$, where $\{\phi_h^i\}_{i=1}^{n_h}$ are 
the canonical basis functions   of $V_h$.   A substitution of this to the semi-discrete system above leads to the matrix system 
\begin{equation} \label{eqn:discrete}
\Mcl \Yc_h'(t)=\Ach \Yc_h(t)+\Bc_h\mathfrak{u}_h,\quad \Yc_h(0)= \begin{pmatrix} \left(\langle y_0 , \phi_h^i\rangle\right) \\  \left(\langle z_0 , \phi_h^i\rangle\right) \end{pmatrix} ,
\end{equation}
with 
 $\Mcl=\begin{pmatrix}
\mathcal{G}_h & O \\
O & \mathcal{G}_h
\end{pmatrix},$ $\Ach=\begin{pmatrix}
-\eta_0\mathcal{K}_h+(\omega-\nu_0)\mathcal{G}_h & -\eta_1\mathcal{G}_h\\
\mathcal{G}_h & -\beta_0\mathcal{K}_h+(-\kappa+\omega-\nu_0)\mathcal{G}_h
\end{pmatrix}$, $\Bc_h=\begin{pmatrix}
\mathcal{G}_h\\
O 
\end{pmatrix}$, \\$O$ being the zero matrix of size $n_h\times n_h$,  $\mathcal{K}_h=(\langle \nabla \phi_h^i,\nabla\phi_h^j\rangle)_{1\le i,j\le n_h},$  $\mathcal{G}_h=(\langle \phi_h^i,\phi_h^j\rangle)_{1\le i,j\le n_h},$ $\Yc_h:=(y_1,...,y_{n_h},$ $z_1,...,z_{n_h})^T$ and $\mathfrak{u}_h\in \mathbb{R}^{n_h}$ being the control we seek for stabilization (see Step 5 below). Note that for each $h$, $\mathcal{G}_h$ is a gram matrix, $\mathcal{K}_h$ is stiffness matrix and hence both are invertible. Thus the matrix $\Mcl$ is also invertible. This and Picard's existence theorem imply that for each $h,$ \eqref{eqn:discrete} has a unique global solution.\\[1mm]
\noindent  Note that  $\Ach$ with order $2n_h\times 2n_h$ is the matrix representation of $\Af_{\omega_h}$  and hence for each $h>0$ and $\omega\in \mathbb{R},$  $\Ach$ and $\Af_{\omega_h}$ have the same set of eigenvalues. Also, $\Bc_h$ with order $2n_h \times n_h$ is  the matrix representation of $\Bf_h$. \\[1mm]
\noindent \textbf{Implementation procedure for stabilization.}
{We closely follow steps described in \cite{CRRS,Thev-ths,JPR6-2017}.}
We start with an unstable discrete system \eqref{eqn:discrete} and describe the procedure for stabilization below.  For a fixed mesh-size $h,$ perform the Steps 1-5 and repeat for each refinement. 
 \begin{itemize}
\item[Step 1.] Calculate the matrices $\mathcal{G}_h$, $\mathcal{K}_h$ and then $\Ach$, $\Bc_h$, and $\Mcl$.

\item[Step 2.] Compute the eigenvalues and corresponding eigenvectors of $\Ach$. Denote the unstable eigenvalues as $\{\Lambda_{h,i},\overline{\Lambda_{h,i}}\}_{i=1}^{n_h^\mathfrak{u}}$ and then the corresponding eigenvectors as $\{w_{h,i}+iv_{h,i},w_{h,i}-iv_{h,i}\}_{i=1}^{n_h^\mathfrak{u}}$, where $2n_h^\mathfrak{u}$ %(for $\omega=6$, $n_u=1$) 
is the total number of unstable eigenvalues of $\Ach$. 
Construct $$\mathcal{E}_h^{\mathfrak{u}}:=(w_{h,1} \quad v_{h,1} \quad w_{h,2}\quad v_{h,2}\quad \cdots \quad w_{h,n_h^\mathfrak{u}}\quad v_{h,n_h^\mathfrak{u}})_{2n_h\times 2n_h^\mathfrak{u}}.$$
For the case of real unstable eigenvalues and corresponding eigenvectors, we construct $\mathcal{E}_h^{\mathfrak{u}}$ by taking eigenvectors corresponding to the unstable eigenvalues. Repeat the same for the transpose $\Ach^T$  of $\Ach$ and construct $\mathbf{\Xi}_h^\mathfrak{u}$ similar to $\mathcal{E}_h^{\mathfrak{u}}$ by taking eigenvectors corresponding to the unstable eigenvalues of $\Ach^T$.

\item[Step 3.]
%Construct the span of all the unstable column vectors in $\mathcal{E}_{\mathfrak{u}}$  and 
Compute the projected matrices $\Ac{_h^\mathfrak{u}}$, $\mathcal{B}_h^\mathfrak{u}$ and $\mathcal{Q}_h^\mathfrak{u}$ %$\mathcal{E}_{\mathfrak{u}}$ 
$$\Ac{_h^\mathfrak{u}}={(\mathbf{\Xi}_h^\mathfrak{u})}^T\Ach \mathcal{E}_h^{\mathfrak{u}},\quad \mathcal{B}_h^\mathfrak{u}={(\mathbf{\Xi}_h^\mathfrak{u})}^T \Bc_h, \text{ and } \mathcal{Q}_h^\mathfrak{u}={(\mathcal{E}_h^{\mathfrak{u}})}^T \Mcl \mathcal{E}_h^{\mathfrak{u}}.$$
%\footnote{Is it correct? or will it be ${(\mathbf{\Xi}_h^\mathfrak{u})}^T$?}

\item[Step 4.]  Solve the Riccati equation 
$$\Ac{_h^\mathfrak{u}} \mathcal{P}_h^{\mathfrak{u}}+\mathcal{P}_h^\mathfrak{u}{(\Ac{_h^\mathfrak{u}})}^T-\mathcal{P}_h^{\mathfrak{u}}\mathcal{B}_h^\mathfrak{u}{(\mathcal{B}_h^\mathfrak{u})}^T\mathcal{P}_h^{\mathfrak{u}}+\mathcal{Q}_h^\mathfrak{u}=0$$
for $\mathcal{P}_h^{\mathfrak{u}}$ in MATLAB using the command {\it care}.
 
\item[Step 5.] Substitute the feedback matrix $-{(\mathcal{B}_h^\mathfrak{u})}^T\mathcal{P}_h^\mathfrak{u}{(\mathbf{\Xi}_h^\mathfrak{u})}^T$ in \eqref{eqn:discrete} to obtain 
$$\Mcl\Yc_h'(t)=\Ach \mathcal{Y}_h(t)-\Bc_h{(\mathcal{B}_h^\mathfrak{u})}^T\mathcal{P}_h^\mathfrak{u}{(\mathbf{\Xi}_h^\mathfrak{u})}^T\Yc_h(t).$$
%and solve using the backward Euler difference formula in the time step.
\end{itemize}

\noindent \textbf{Time solver.} A time discretization using a backward Euler method leads to a system
\begin{align*}
\Mcl\frac{\Yc_h^1-\Yc_h^{0}}{\Delta t}=(\Ach-\Bc_h{(\mathcal{B}_h^\mathfrak{u})}^T\mathcal{P}_h^\mathfrak{u}{(\mathbf{\Xi}_h^\mathfrak{u})}^T) \Yc_h^{1},\quad \Yc_h^0=\Yf_{h,0}, 
\end{align*}
for the first time step. We choose a fixed step $\Delta t =0.001$. This is a linear system of equations and since $\Mcl$ is invertible, for each $h,$ the  system has a unique solution $\Yc_h^1$. {Starting from the second time step, apply the backward difference formula 2 (BDF2, \cite{MatESAIM})  below  with time step  $\Delta t=0.001.$}
%\footnote{what happens at the first time step? In the formula for the first step also, $\Delta t$ is appearing in the denominator. What value was taken at first-step? Is that time step changing or fixed as 0.0001?}
\begin{align*}
\Mcl\frac{1.5\Yc_h^{n+2}-2\Yc_h^{n+1}+0.5\Yc_h^{n}}{\Delta t}=\left(\Ach-\Bc_h{(\mathcal{B}_h^\mathfrak{u})}^T\mathcal{P}_h^\mathfrak{u}{(\mathbf{\Xi}_h^\mathfrak{u})}^T\right) \Yc_h^{n+2},\quad n=0,1,2,...,. 
\end{align*}
Since $\Mcl$ is invertible, the above linear system has a unique solution $\Yc_h^{n+2}$ for each $h$ and $n=0,1,2,\ldots.$
%For a stable system , $\mathfrak{u}_h=0$ and Steps 1-5 are skipped.

\subsection{Error and order of convergence.} \label{subs:EOC} The computational errors and orders of  convergence of discrete solutions are calculated as follows. Let $\begin{pmatrix}\widetilde{y}_{h_i} \\ \widetilde{z}_{h_i}\end{pmatrix}$ and $\begin{pmatrix}\widetilde{y}_{h_{i+1}} \\ \widetilde{z}_{h_{i+1}} \end{pmatrix}$ be the computed solutions at $i$-th and $(i+1)$-th levels, and $\widetilde{u}_{h_i}$ and $\widetilde{u}_{h_{i+1}}$ denote the computed stabilizing control at $i$-th and $(i+1)$-th levels, for $i=1,2,\ldots$.  The errors in different norms are denoted as  
%Since the exact solution  is not available, 
%The errors in $\Lt$ and $H^1(\Omega)$ norm are computed as 
\begin{align*}
& \text{err}_{L^2}(\widetilde{y}_{h_i})=\|\widetilde{y}_{h_{i+1}}-\widetilde{y}_{h_i}\|,\quad  \text{err}_{H^1}(\widetilde{y}_{h_i})=\|\widetilde{y}_{h_{i+1}}-\widetilde{y}_{h_i}\|_{H^1(\Omega)}, \quad \text{err}_{L^2}(\widetilde{z}_{h_i})=\|\widetilde{z}_{h_{i+1}}-\widetilde{z}_{h_i}\|,  \\
&  \text{err}_{H^1}(\widetilde{z}_{h_i})=\|\widetilde{z}_{h_{i+1}}-\widetilde{z}_{h_i}\|_{H^1(\Omega)} \text{ and } \text{err}_{L^2}(\widetilde{u}_{h_i})=\|\widetilde{u}_{h_{i+1}}-\widetilde{u}_{h_i}\|.
\end{align*} 
\noindent Let $e_i$ and $h_i$ be the error and the discretization parameter at the $i$-th level, respectively. {Then the numerical order of convergence $\alpha_{i+1}$ at the $i$-th level is approximated  using
\begin{equation}\label{roc formula}
\alpha_{i+1}\approx \log(e_{i+1}/e_i)/\log(h_{i+1}/h_i) \text{ for }i=1,2,3,\ldots .
\end{equation}
}
\subsection{Numerical results} \label{sub:5}
\noindent  
Choose $\Omega=\mathcal{O}=(0,1) \times (0,1)$,  $\omega=25, \kappa=1,\eta_0=1, \eta_1=5,\,\beta_0=0.8$, and $\nu_0=0$ in \eqref{eqn:d-state-2}.  Choose the initial conditions as $y_0(x_1,x_2)=x_1(1-x_1)x_2(1-x_2)$ and $z_0(x_1,x_2)=\sin(\pi x_1)\sin(\pi x_2).$ Though the solution to the system \eqref{eqn:discrete} with $\mathfrak{u}_h=0$  converges in the energy and $L^2$ norms with the expected order of convergence, with the parameters chosen above,  we illustrate that the solution is unstable.  Furthermore, by applying a feedback control, the system is stabilized  and the stabilized solution converges with expected rate of convergence in the energy and $L^2$ norms.

\medskip \noindent 
Though the theoretical results are established for $C^2$-boundary, the proposed method works for domains with Lipschitz boundary as evident from the example.

\medskip
\noindent The eigenvalues of $-\Delta$ in $\Omega$ are  $\lambda_{m,n}:=(n^2+m^2)\pi^2$, $n,m=1,2,3,\cdots.$ Now, utilizing this, Proposition  \ref{pps:spec Af} yields the exact eigenvalues of $\Af$ and 
 $\Aw=\Af+\omega \If$ for any $\omega\in\mathbb{R}$. We compute the eigenvalues of $\Aw$ with $\omega=25$ which leads to two unstable eigenvalues.
 %, that is, two eigenvalues with positive real part.
 Next, we compute the eigenvalues of $\Ach$ in MATLAB using command {\it eigs}. 
Figure \ref{fig:spec-A6} shows the plot of a few exact (resp. approximate) eigenvalues of $\Aw$ (resp. $\Ach$)  for $n,m=1,\cdots,5$ and the choice of the mesh-size  $h=\frac{1}{2^6}$. The plots show that the computed eigenvalues indeed provide a good approximation of the exact eigenvalues.  Table \ref{tab:roc_eig_2} validates the convergence of the first two eigenvalues with quadratic rate of convergence as discussed in Lemma \ref{lem:eig-conv}.  The  errors and orders of convergences computed using \eqref{roc formula} for the two eigenvalues are presented in Table \ref{tab:roc_eig_2}. 
%To calculate the order of convergence,  use the formula \eqref{roc formula}.

\begin{figure}[ht!]
            \includegraphics[width=.37\textwidth]{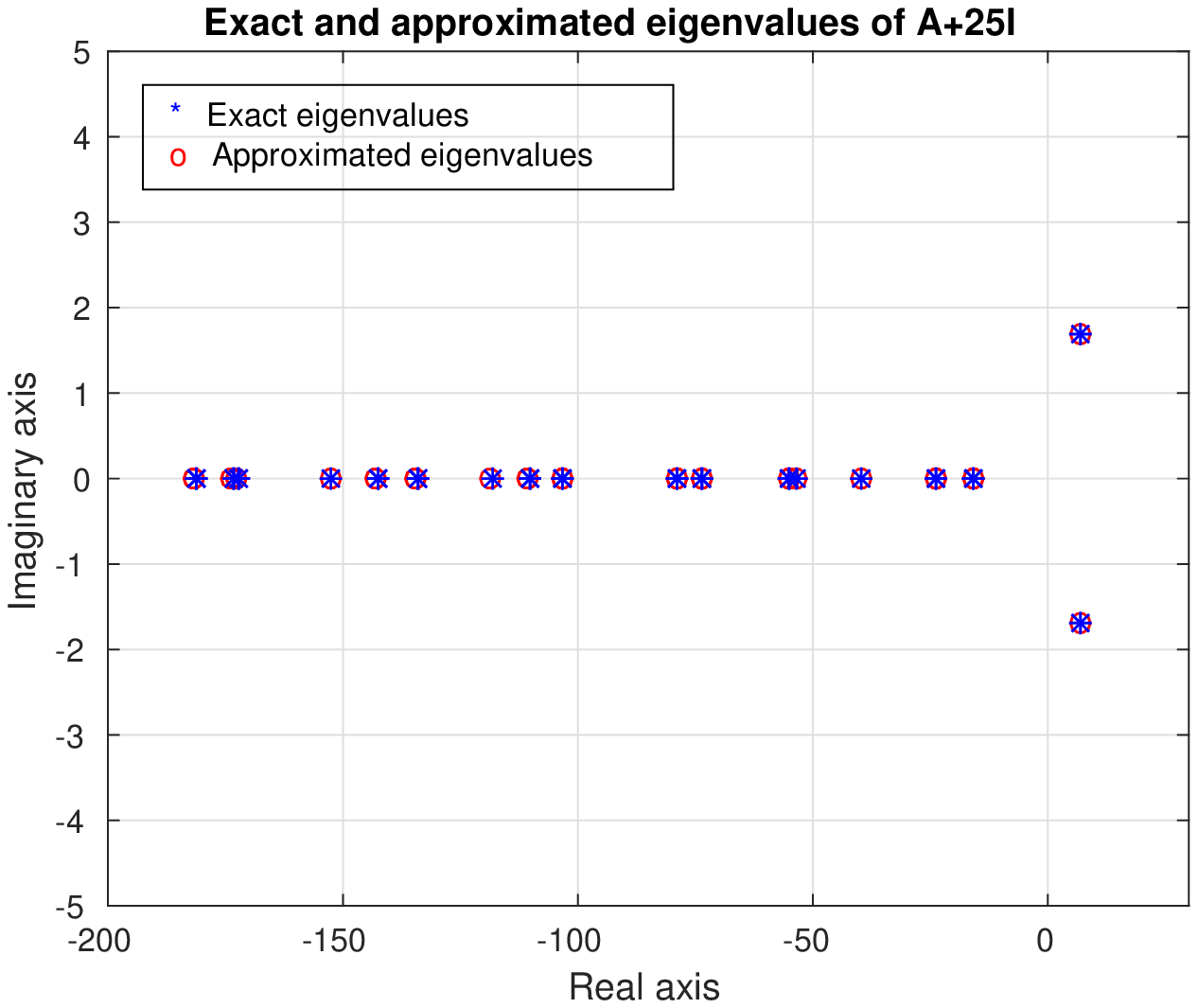}
            \includegraphics[width=.37\textwidth]{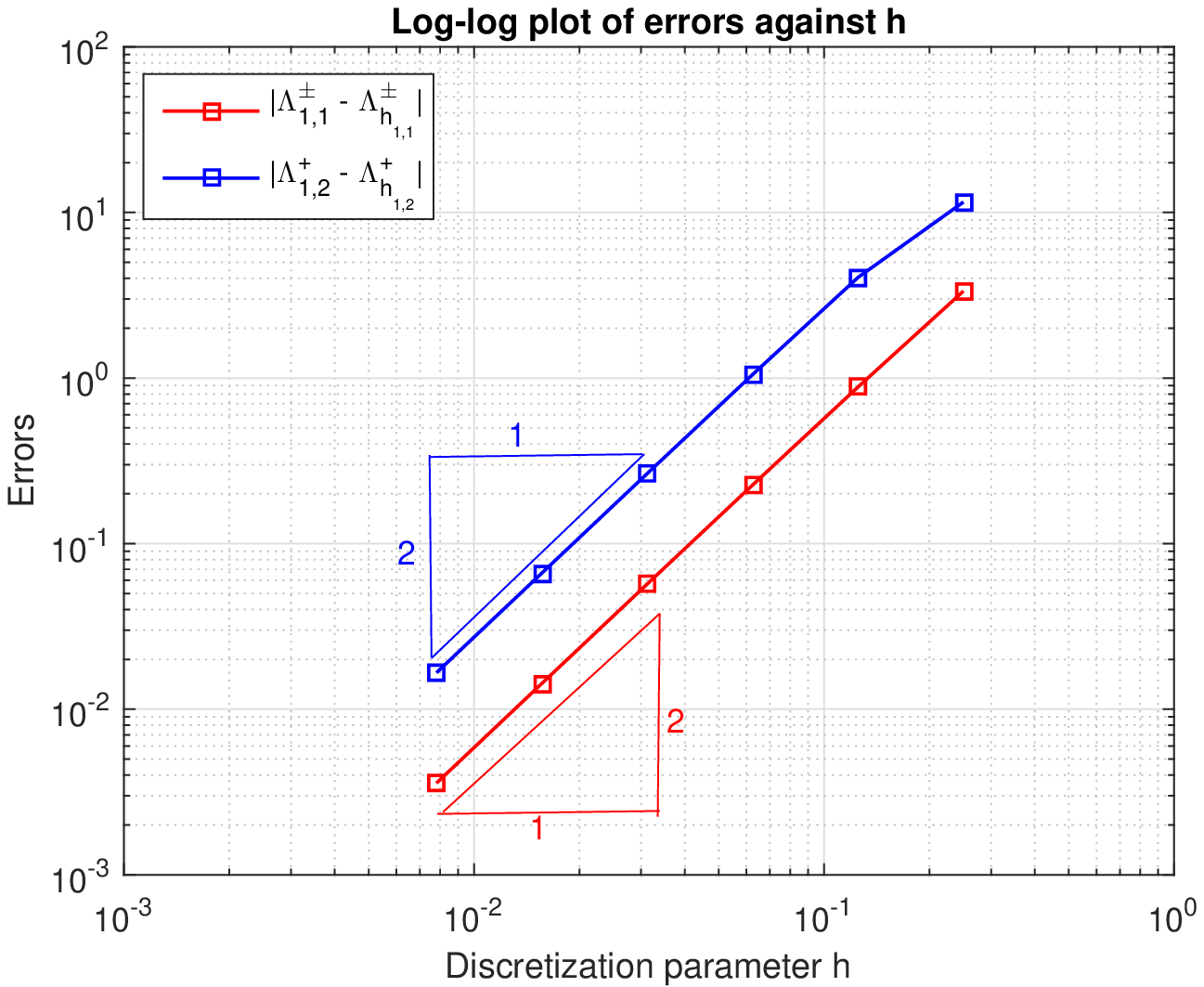}
            \caption{(a) Few exact and approximated eigenvalues (b) log-log plot of errors against discretization parameter $h$}\label{fig:spec-A6}       
\end{figure}

\begin{table}[ht!]
\footnotesize
\centering
\begin{tabular}{|c||c|c|c||c|c|c||}
\hline
$h$ & $\Lambda_{h_{1,1}}^\pm$ & $|\Lambda_{1,1}^\pm-\Lambda_{h_{1,1}}^\pm|$ & Order & $\Lambda_{h_{1,2}}^+$ & $|\Lambda_{1,2}^+-\Lambda_{h_{1,2}}^+|$ & Order   \\ 
\hline
\hline
$1/2^2 $ & 3.41226 $\pm$ 1.26611i & 3.34832 & --- & -27.64014 & 11.55674 & --- \\  
\hline
$1/2^3$ & 5.85591 $\pm$ 1.59065i & 0.88348 &  1.92215 & -20.13492 & 4.05151 & 1.51221\\
\hline
$1/2^4$ & 6.50970 $\pm$ 1.65928i & 0.22610 & 1.96619  & -17.13401 & 1.05060 & 1.94724  \\
\hline
$1/2^5$ & 6.67791 $\pm$ 1.67598i & 0.05707 & 1.98619 & -16.34851 & 0.26512 & 1.98655 \\
\hline
$1/2^6$ & 6.72046 $\pm$ 1.68014i & 0.01431 & 1.99522 & -16.14984 & 0.06644  & 1.99641 \\
\hline
$1/2^7$ & 6.73114 $\pm$ 1.68118i & 0.00358 & 1.99848 & -16.10002 & 0.01662  & 1.99907 \\
\hline
\hline
Exact & 6.73471 $\pm$ 1.68153i & --- & --- & -16.08341 & --- & ---\\
\hline
\end{tabular}
\caption{Computed errors and orders of convergence of the first two eigenvalues\hspace{2.5cm}$\,$}\label{tab:roc_eig_2}
\end{table}

\noindent First the solution to  \eqref{eqn:discrete} is computed with $\mathfrak{u}_h:=0.$ Figure \ref{fig:energy_un-ctrl}(a) indicates  that the solution  without control is unstable as the energy increases with time $t$ while Figure \ref{fig:energy_un-ctrl}(b) represents the evolution of the energy on log-log scale. {At time level $T=0.1,$ the computed errors are plotted on log-log scale against $h$ in Figure \ref{fig:energy_un-ctrl}(c) and here, we observe a quadratic order of convergence in $L^2$-norm and linear order of convergence in $H^1$-norm, even for the unstable solutions}.

\begin{figure}[ht!]
            \includegraphics[width=.33\textwidth]{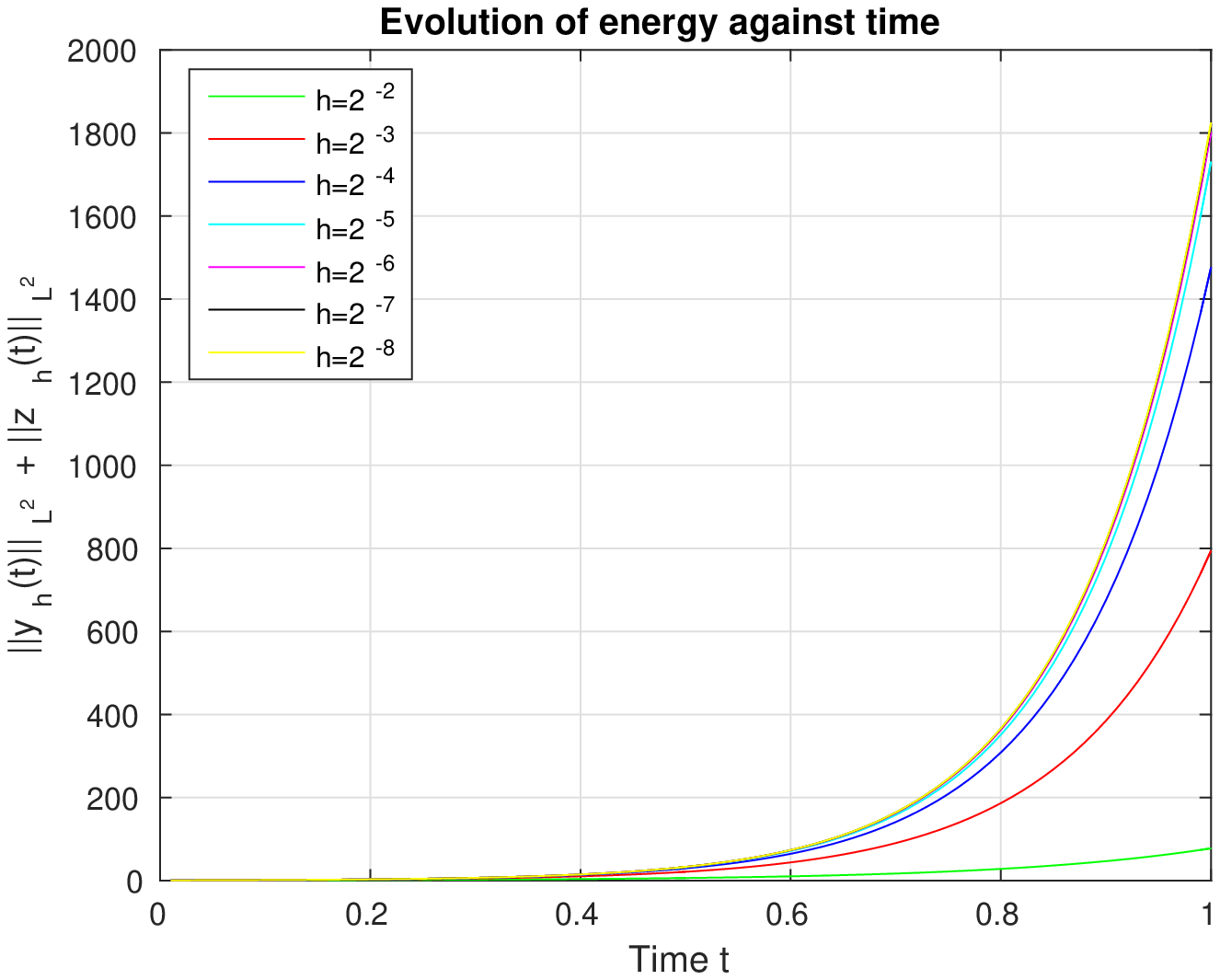}\hfill
            \includegraphics[width=.33\textwidth]{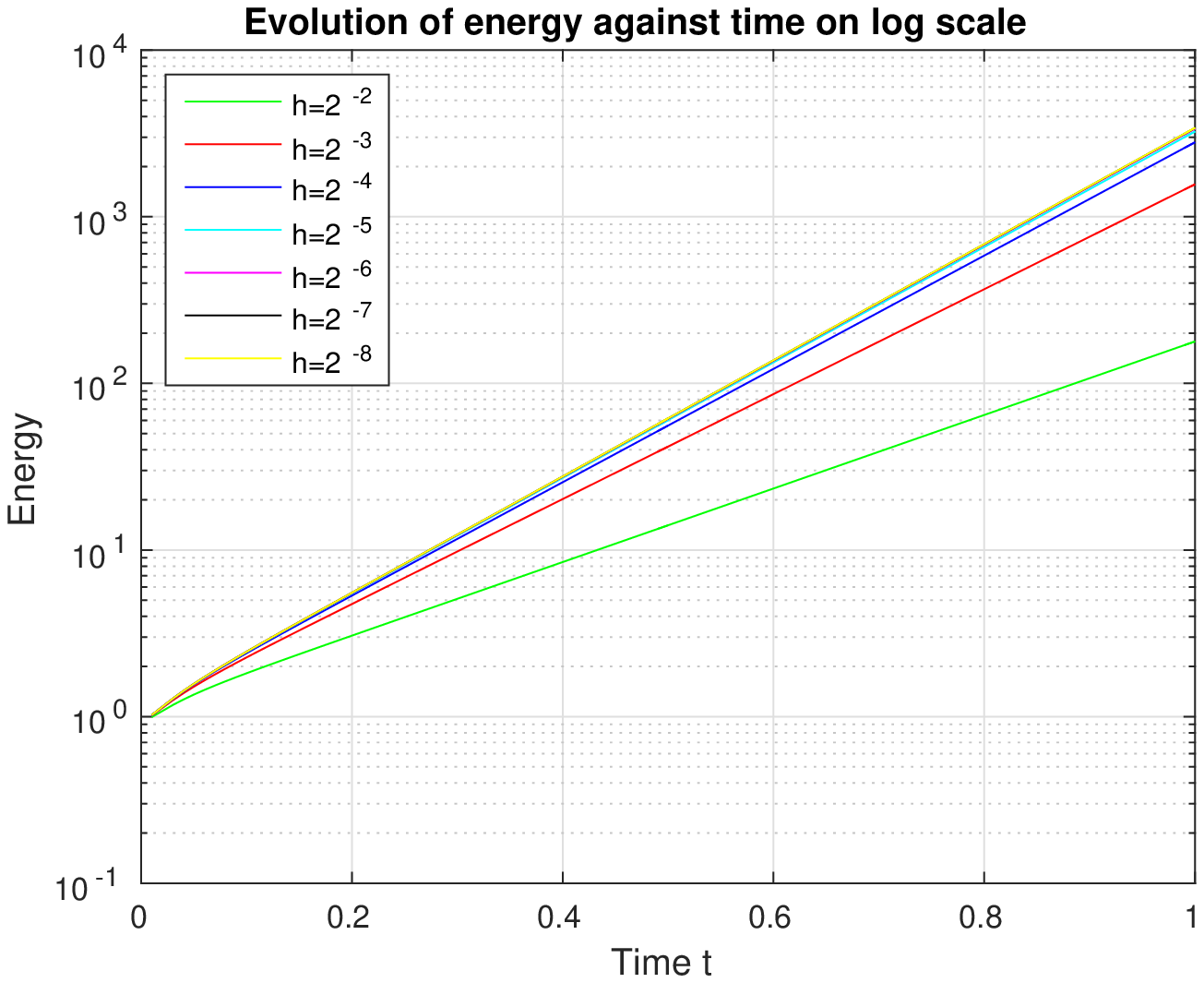}\hfill
            \includegraphics[width=.33\textwidth]{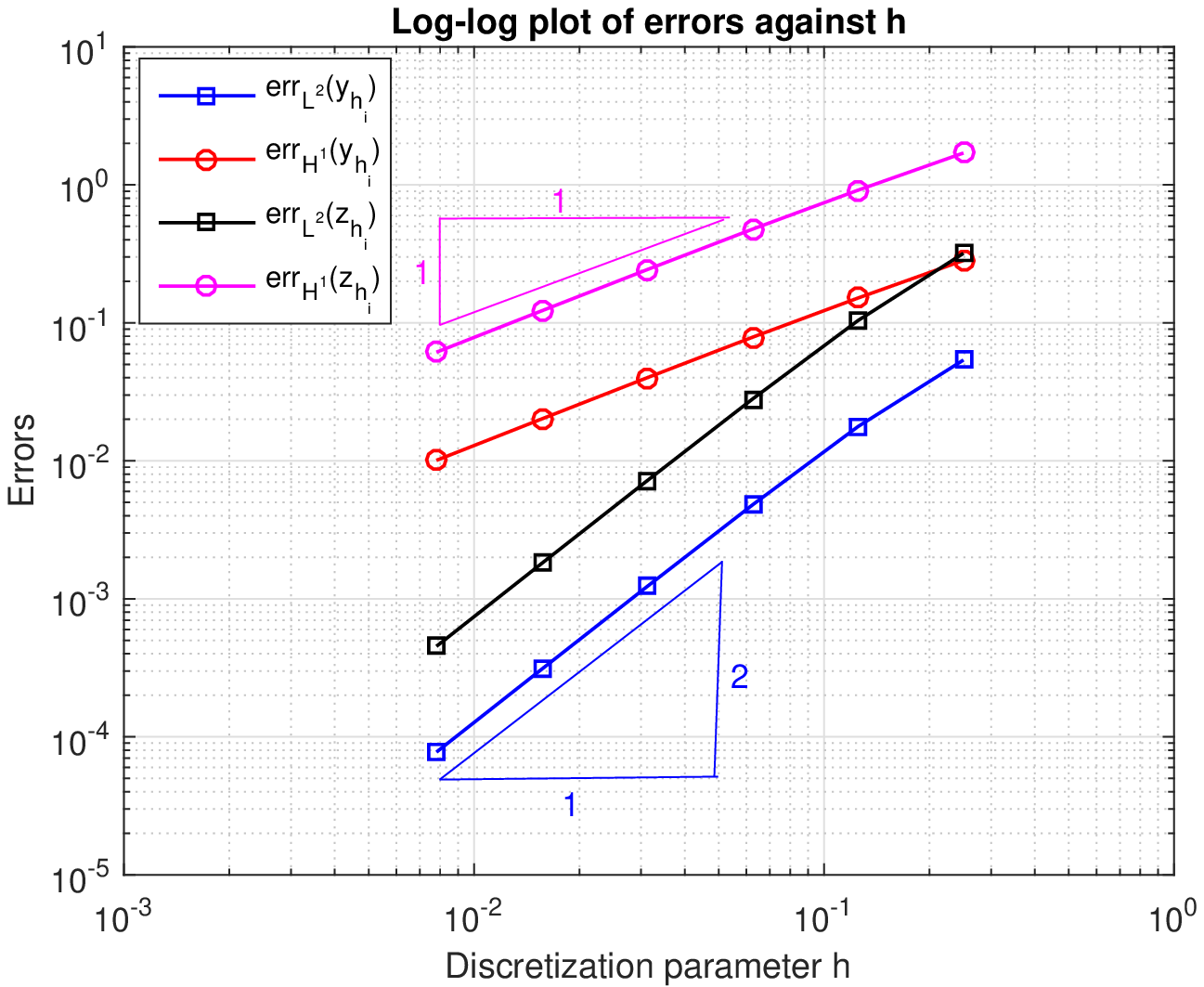}
            \caption{(a) Evolution of the solution in $L^2$- norm, (b) on log-log scale (c) log-log plot of errors against discretization parameter $h$}\label{fig:energy_un-ctrl}     
\end{figure}

\noindent \textbf{Solution with control.} Following the implementation  procedure outlined in Section \ref{sub:2}, we compute feedback control  and obtain the stabilized solution. Figure \ref{fig:stab data}(a) and Figure \ref{fig:stab data2}(a)  represent the evolution of $L^2$-norm of the stabilized solution $\begin{pmatrix} \widetilde{y_h} \\ \widetilde{z_h}\end{pmatrix}$ and the stabilizing control $\widetilde{u_h},$ respectively. Figure \ref{fig:stab data}(b) and Figure \ref{fig:stab data2}(b) represent the evolution of $L^2$-norm of the stabilized solution $\begin{pmatrix} \widetilde{y_h} \\ \widetilde{z_h}\end{pmatrix}$ and the stabilizing control $\widetilde{u_h}$ in log-log scale. Plots of errors against discretization parameter $h$ on log-log scale is presented in Figure \ref{fig:stab data2}(c). Table \ref{tab:stab_roc_y} presents the relative errors and orders of convergence for the computed stabilized solution  $\widetilde{y}_h,$ $\widetilde{z}_h$ and stabilizing control $\widetilde{u}_h$ at time level $T=0.1$. A few eigenvalues before and after stabilization are plotted in Figure \ref{fig:stab data}(c).

 % \noindent Observe that the numerical experiments in this example validate the convergence results obtained in Lemma \ref{lem:eig-conv}, Theorems~\ref{th:res-sg-est}, \ref{th:dro} and \ref{th:main-conv-new}. 
\medskip
 % \noindent Observe that Figure \ref{fig:spec-A6} and Table \ref{tab:roc_eig_2} validate Lemma \ref{lem:eig-conv} and Figure \ref{fig:energy_un-ctrl}(c) validates Theorem~\ref{th:res-sg-est}. 
\noindent The stabilizability stated in Theorem~\ref{th:dro} is verified by Figure~\ref{fig:stab data}(a)-(b). Figure \ref{fig:stab data2}(c) and Table~\ref{tab:stab_roc_y} endorse the error estimates obtained in Theorem~\ref{th:main-conv-new}.

\medskip
\noindent {Computationally, we observe a linear order of convergence in energy norm for the unstable and stabilized solutions as shown in Figures \ref{fig:energy_un-ctrl}(c), \ref{fig:stab data2}(c) and Table \ref{tab:stab_roc_y}.}
% \begin{table}[ht!]
% \centering
% \footnotesize
% \begin{tabular}{|c||c|c||c|c||c|c||c|c||}
% \hline
% $h$ & $\text{err}_{L^2}(\widetilde{y}_{h_i})$ & Order  & $\text{err}_{H^1}(\widetilde{y}_{h_i})$ & Order & $\text{err}_{L^2}(\widetilde{z}_{h_i})$ & Order  & $\text{err}_{H^1}(\widetilde{z}_{h_i})$ & Order    \\ 
% \hline
% \hline
% $1/2^2 $ & 5.37603e-02 & --- & 2.83335e-01 & --- & 3.18796e-01 & --- & 1.70141 & --- \\  
% \hline
% $1/2^3$ & 1.77181e-02  &  1.60131 & 1.51361e-01 & 0.90451 & 1.03933e-01 &  1.61697 & 9.15021e-01 & 0.89485 \\
% \hline
% $1/2^4$ & 4.80313e-03  & 1.88317  & 7.87472e-02 & 0.94269 & 2.80892e-02 & 1.88757  & 4.77849e-01 & 0.93724 \\
% \hline
% $1/2^5$ &  1.230186e-03 & 1.96509 & 4.00504e-02  & 0.97541 & 7.18851e-03 & 1.96625 & 2.43285e-01 & 0.97391 \\
% \hline
% $1/2^6$ &  3.09721e-04 & 1.98983 & 2.01425e-02  & 0.99157 & 1.80947e-03 & 1.99012 & 1.22387e-01 & 0.99119 \\
% \hline 
% $1/2^7$ &  7.75865e-05 & 1.99709 & 1.00897e-02  & 0.99736 & 4.53259e-04 & 1.99716 & 6.13096e-02 & 0.99727 \\
% \hline 
% \end{tabular}
% \caption{Computed errors and orders of convergence of $y$ at time $T=0.1$ in $\Lt$ and $H^1(\Omega)$ norm for solutions without control}\label{tab:stab_roc_y_woc}
% \end{table}

\begin{figure}[ht!]
            \includegraphics[width=.32\textwidth]{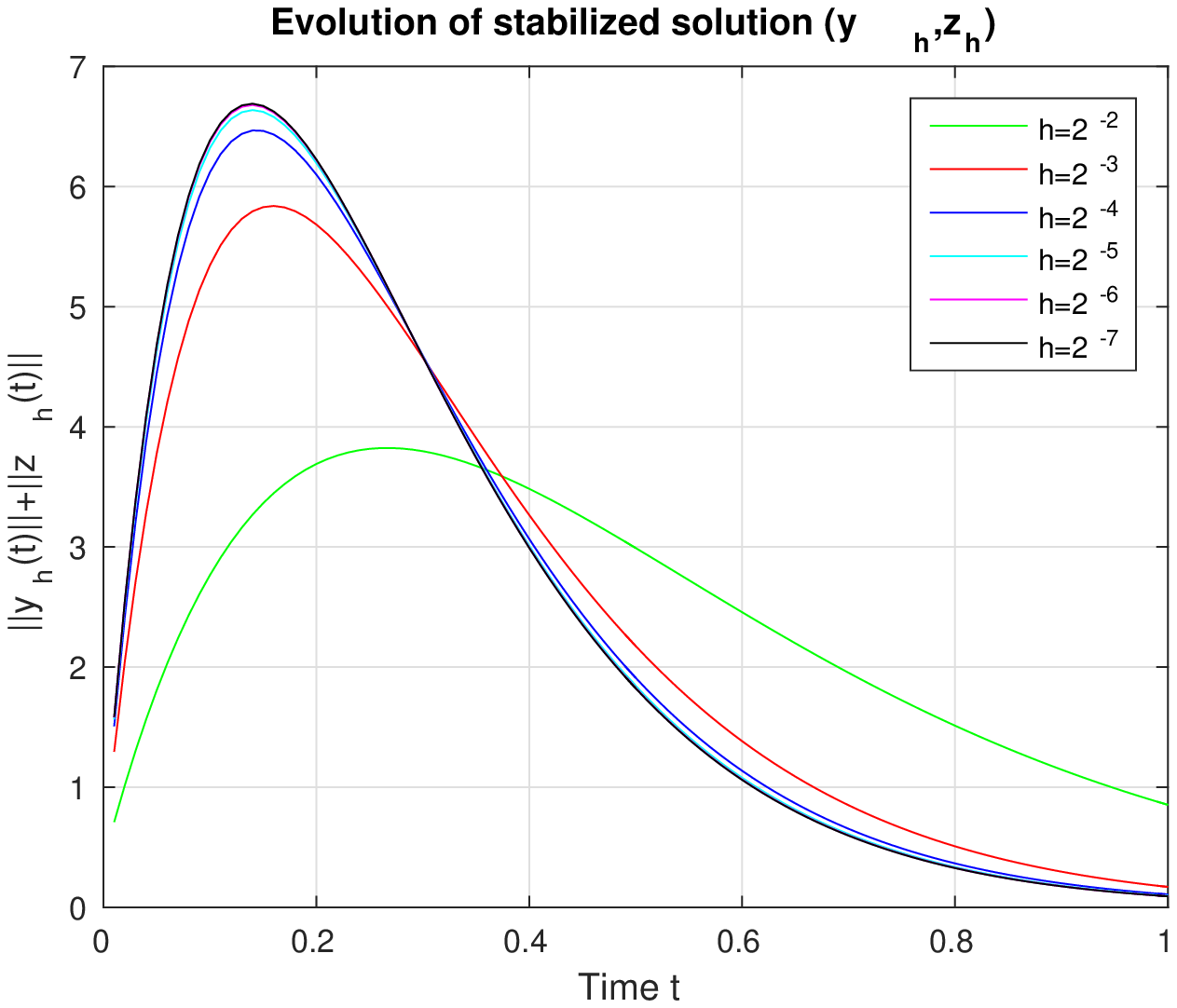}\hfill
            \includegraphics[width=.32\textwidth]{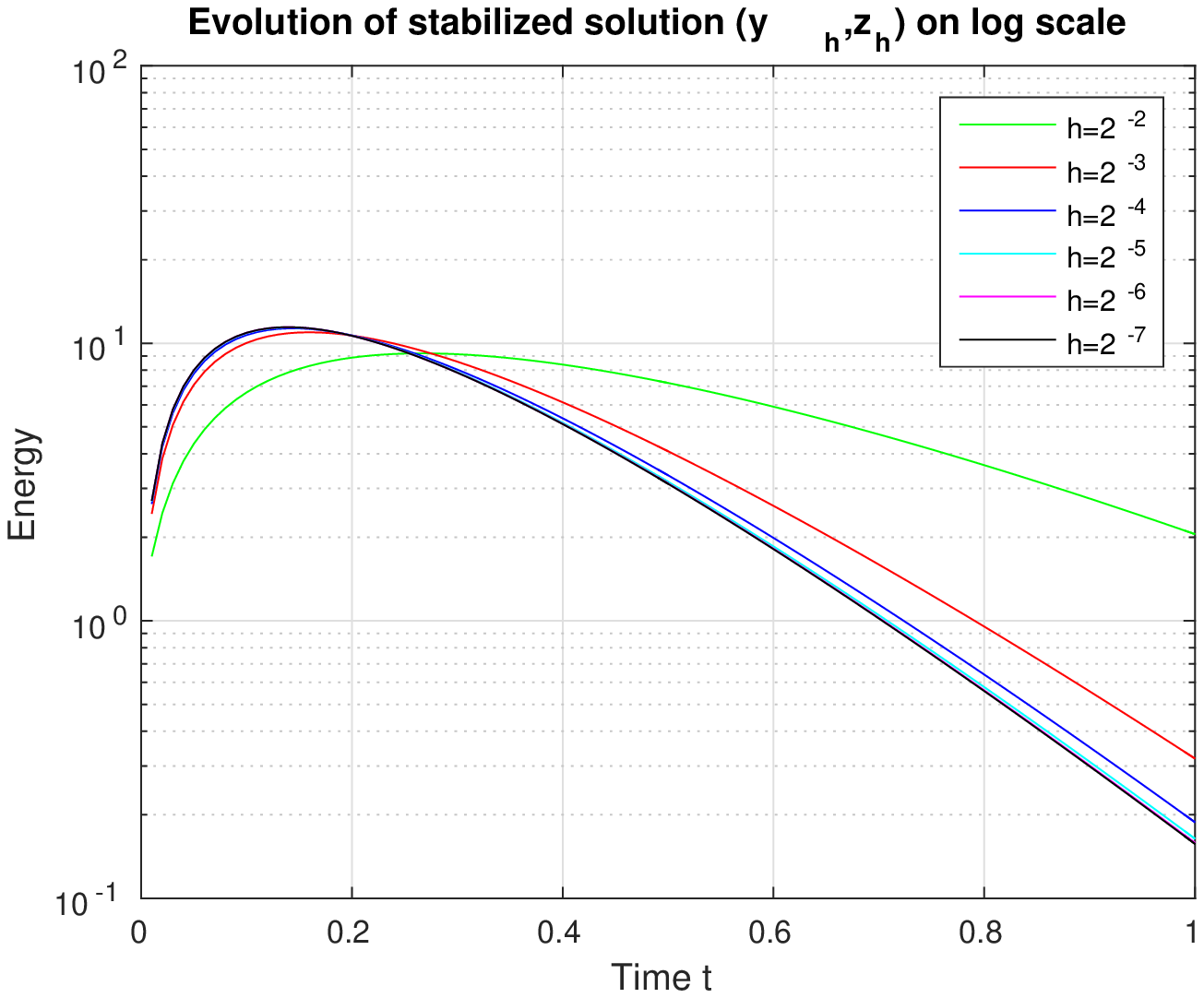}
            \hfill
            \includegraphics[width=.32\textwidth]{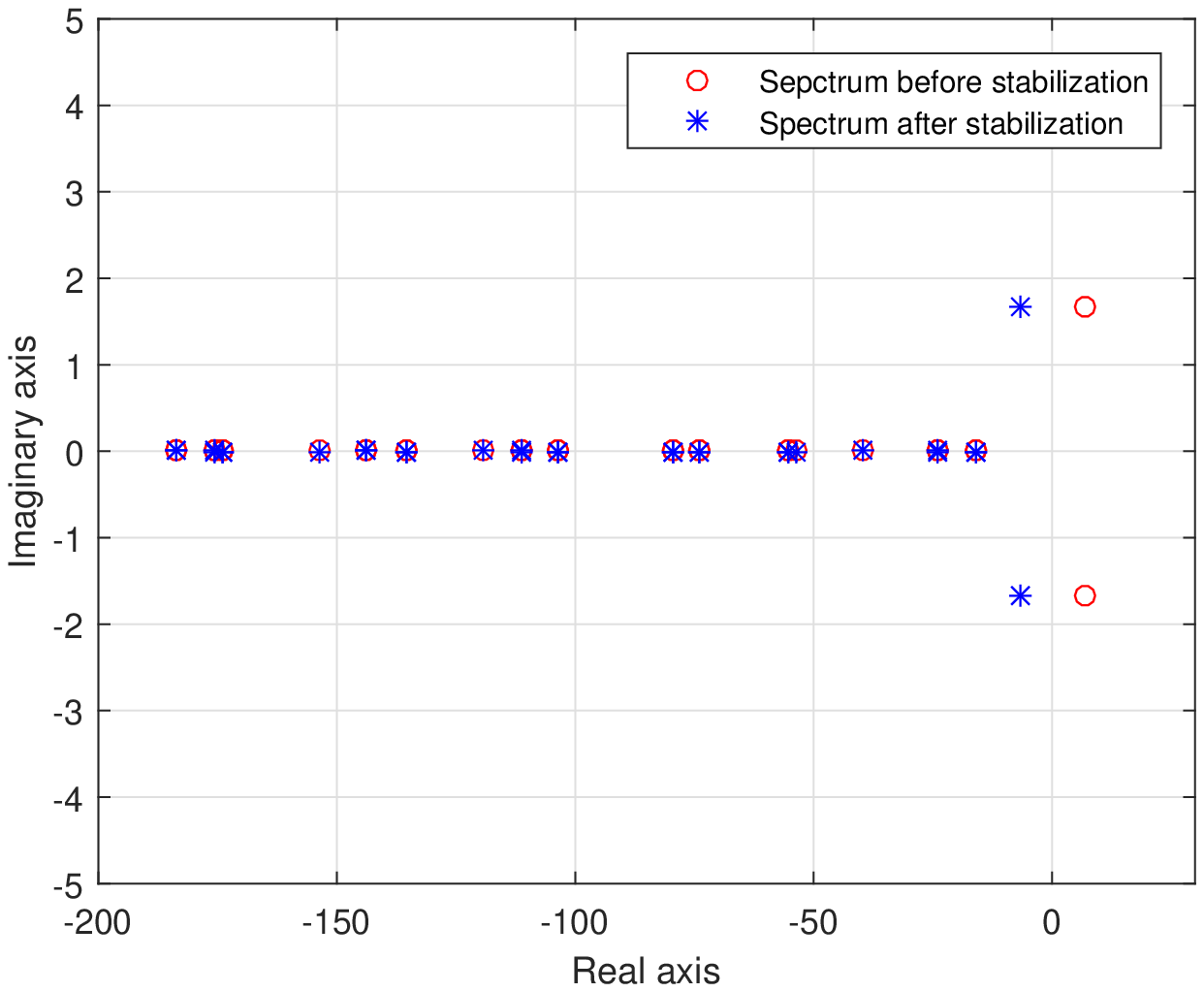}
            \caption{(a) Evolution of the stabilized solution in $L^2$-norm, (b) on log-log scale, (c) spectrum before and after stabilization}\label{fig:stab data}       
\end{figure}
% \begin{figure}[ht!]
%             \includegraphics[width=.32\textwidth]{Ex3-Norm-Time-stab_control}\hfill
%             \includegraphics[width=.32\textwidth]{Ex3-NormSL-Time-stab_ctrl}
%             \hfill
%             \includegraphics[width=.32\textwidth]{Ex3-loglog_error_stab_Tri3}
%             \caption{(a) Evolution of the stabilizing control in $L^2$ - norm, (b) on log-log scale against time $t$ and (c) log-log plot of errors against discretization parameter $h$}\label{fig:stab data2}       
% \end{figure}

\begin{figure}[ht!]
            \includegraphics[width=.32\textwidth]{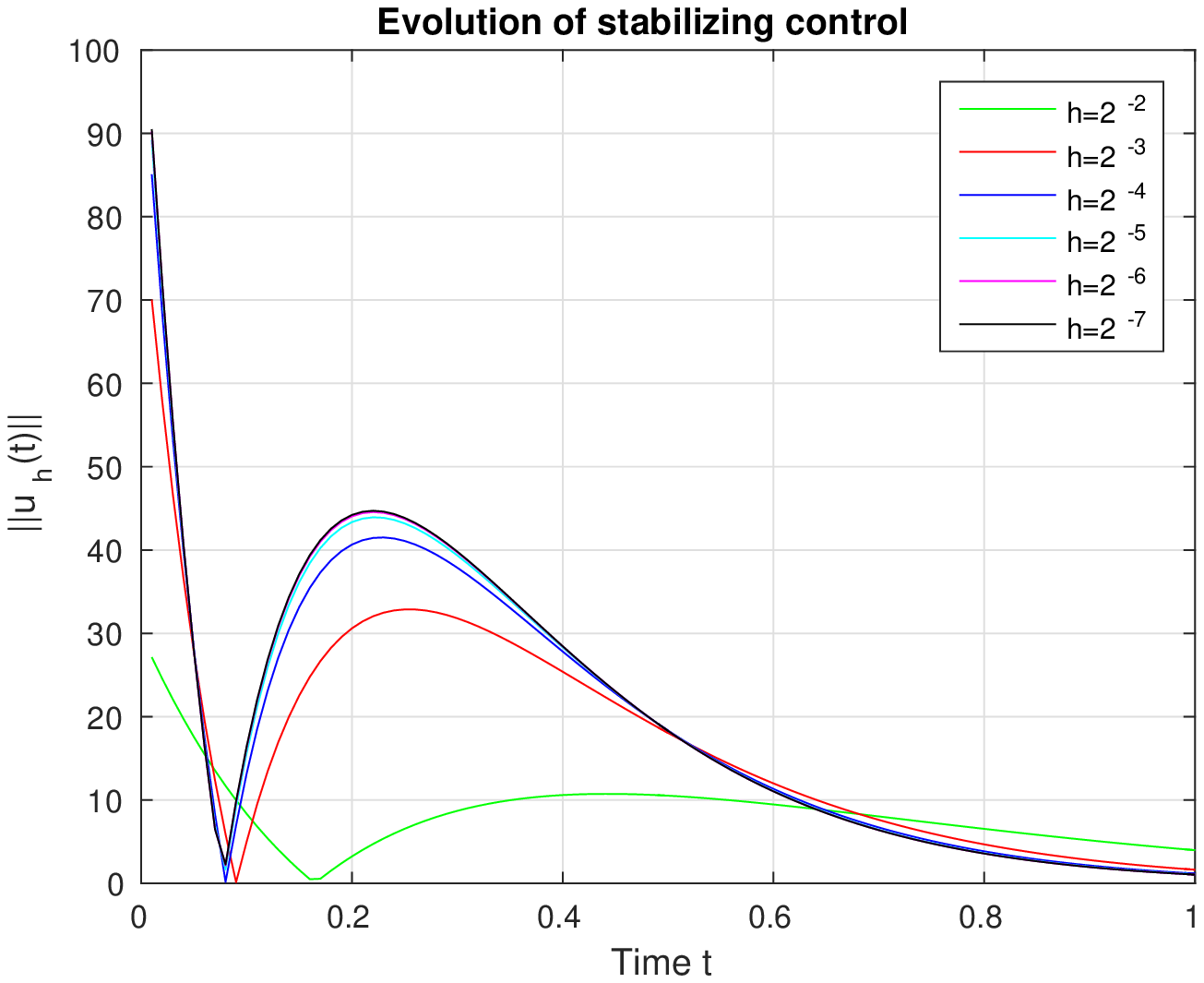}\hfill
            \includegraphics[width=.32\textwidth]{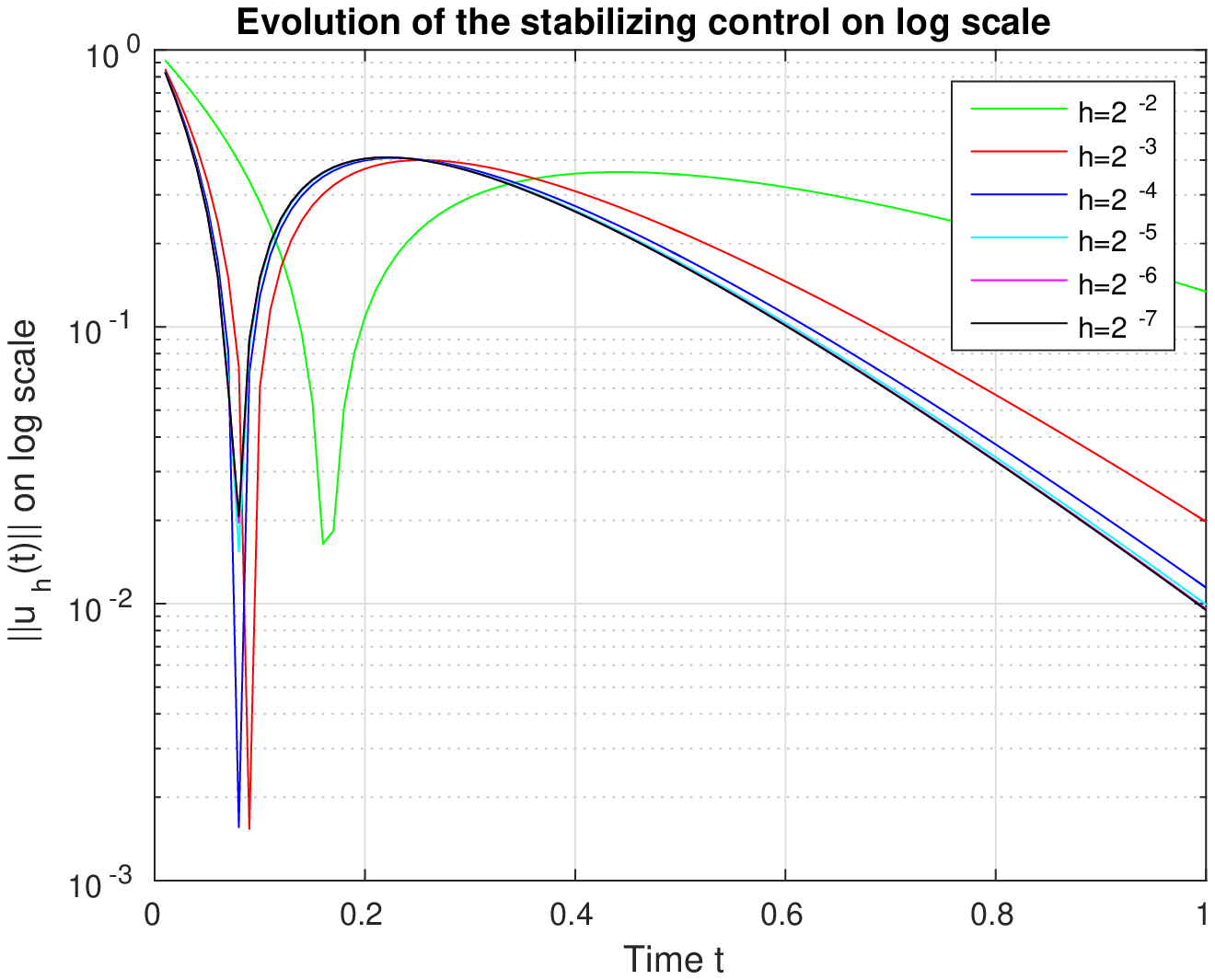}
            \hfill
            \includegraphics[width=.32\textwidth]{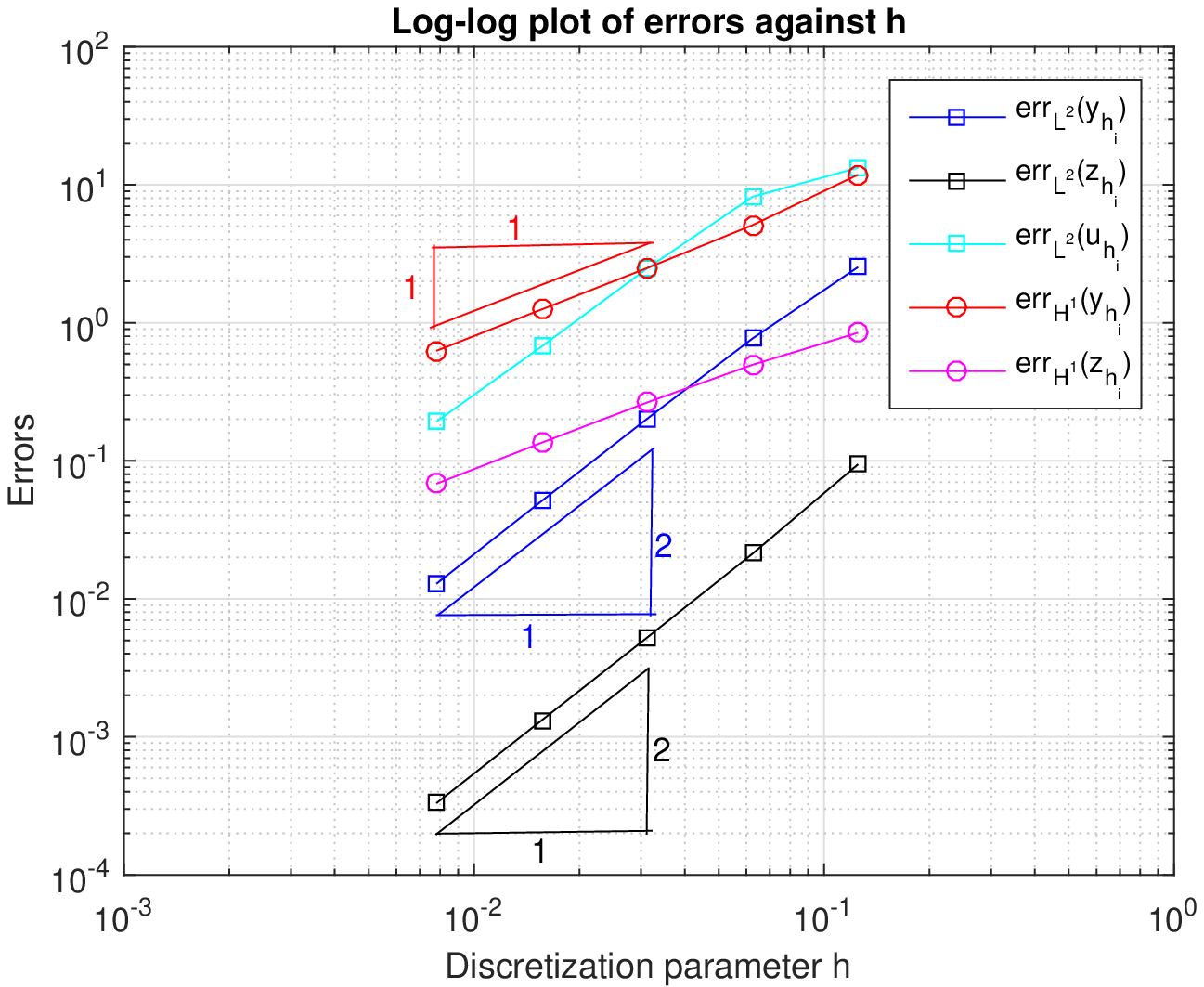}
            \caption{(a) Evolution of the stabilizing control in $L^2$ - norm, (b) on log-log scale against time $t$ and (c) log-log plot of errors against discretization parameter $h$}\label{fig:stab data2}       
\end{figure}

% \begin{table}[ht!]
% \centering
% \footnotesize
% \begin{tabular}{|c||c|c||c|c||c|c||c|c||c|c||}
% \hline
% $h$ & $\text{err}_{L^2}(\widetilde{y}_{h_i})$ & Order  & $\text{err}_{H^1}(\widetilde{y}_{h_i})$ & Order & $\text{err}_{L^2}(\widetilde{z}_{h_i})$ & Order  & $\text{err}_{H^1}(\widetilde{z}_{h_i})$ & Order & $\text{err}_{L^2}(\widetilde{u}_{h_i})$ & Order   \\ 
% \hline
% \hline
% $1/2^2 $ & 6.06294e-01 & --- & 3.01552e-00 & --- & 7.83206e-02 & --- & 3.84828e-01 & --- & 2.367607e-00 & ---\\  
% \hline
% $1/2^3$ & 5.76263e-02  &  3.39521 & 2.88963e-01 & 3.38345 & 4.99422e-02 &  3.97106 & 2.43841e-02 & 3.98020 & 4.679057e-01 & 2.33914\\
% \hline
% $1/2^4$ & 1.11407e-02  & 2.37088  & 6.56084e-02 & 2.13893 & 9.00457e-03 & 2.47153  & 4.97672e-03 & 2.29267 & 1.043022e-01 & 2.16545\\
% \hline
% $1/2^5$ &  2.62129e-03 & 2.08749 & 2.25968e-02  & 1.53776 & 2.08487e-04 & 2.1107 & 1.60352e-03 & 1.63395 & 2.770409e-02 & 1.91259\\
% \hline
% $1/2^6$ &  6.46168e-04 & 2.02029 & 9.86422e-03  & 1.19584 & 5.11910e-05 & 2.02599 & 6.75408e-04 & 1.24741 & 8.834777e-03 & 1.64883\\
% \hline 
% %$1/2^7$ &  4.15910805674937 & 1.99241653131381 & 139.578938501936  & 1.02502938319481 \\
% %\hline 
% \end{tabular}
% \caption{Computed errors and orders of convergence of $y$ at time $T=1$ in $\Lt$ and $H^1(\Omega)$ norm for stabilized solutions and stabilizing control}\label{tab:stab_roc_y}
% \end{table}

\begin{table}[ht!]
\centering
\footnotesize
\begin{tabular}{|c||c|c||c|c||c|c||c|c||c|c||}
\hline
$h$ & $\text{err}_{L^2}(\widetilde{y}_{h_i})$ & Order  & $\text{err}_{H^1}(\widetilde{y}_{h_i})$ & Order & $\text{err}_{L^2}(\widetilde{z}_{h_i})$ & Order  & $\text{err}_{H^1}(\widetilde{z}_{h_i})$ & Order & $\text{err}_{L^2}(\widetilde{u}_{h_i})$ & Order   \\ 
\hline
\hline
$1/2^2 $ & 2.53411 & --- & 11.83988 & --- & 0.09431 & --- & 0.84806 & --- & 13.29115 & ---\\  
\hline
$1/2^3$ & 0.77118  &  1.71633 & 5.10696 & 1.21311 & 0.02127 &  2.14789 & 0.49699 & 0.77095 & 8.19389 & 0.69784\\
\hline
$1/2^4$ & 0.20213  & 1.93173  & 2.49708 & 1.03222 & 0.00525 & 2.01754  & 0.26509 & 0.90674 & 2.48418 & 1.72177\\
\hline
$1/2^5$ &  0.05128 & 1.97874 & 1.25199  & 0.99601 & 0.00132 & 1.99252 & 0.13551 & 0.96818 & 0.67589 & 1.87791\\
\hline
$1/2^6$ &  0.01287 & 1.99342 & 0.62745  & 0.99664 & 0.000331 & 1.99436 & 0.06822 & 0.98997 & 0.19243 & 1.81239\\
\hline 
\end{tabular}
\caption{Computed errors and orders of convergence of $y$ at time $T=0.1$ in $\Lt$ and $H^1(\Omega)$ norm for stabilized solutions and stabilizing control}\label{tab:stab_roc_y}
\end{table}

\end{document}